\documentclass{amsart}
\usepackage{amsmath,amssymb,latexsym, amsfonts, amscd, amsthm}
\usepackage{hyperref}
\hypersetup{colorlinks=false}

\usepackage{bbm}
\usepackage{enumerate}
\usepackage[all]{xy}

\allowdisplaybreaks
\newtheorem{thm}{Theorem}[section]
\newtheorem{pro}[thm]{Proposition}
\newtheorem{lm}[thm]{Lemma}

\numberwithin{equation}{section}

\theoremstyle{remark}

\newtheorem{rem}[thm]{Remark}

\theoremstyle{definition}

\DeclareMathOperator*{\End}{\textsf{End}}

\DeclareMathOperator*{\Irr}{\textsf{Irr}}
\DeclareMathOperator*{\disc}{\textsf{disc}}
\DeclareMathOperator*{\Gal}{\textsf{Gal}}
\DeclareMathOperator*{\dimi}{\textsf{dim}}

\DeclareMathOperator*{\Sp}{Sp}
\DeclareMathOperator*{\PSp}{PSp}

\DeclareMathOperator*{\GU}{GU}
\DeclareMathOperator*{\U}{U}
\DeclareMathOperator*{\SO}{SO}
\DeclareMathOperator*{\PSO}{PSO}
\DeclareMathOperator*{\SL}{SL}
\DeclareMathOperator*{\GL}{GL}

\DeclareMathOperator*{\GSO}{GSO}
\DeclareMathOperator*{\GSp}{GSp}
\DeclareMathOperator*{\Spin}{Spin}
\DeclareMathOperator*{\PSL}{PSL}
\DeclareMathOperator*{\PGL}{PGL}
\DeclareMathOperator*{\GSpin}{GSpin}

\DeclareMathOperator*{\WD}{\textsf{WD}}
\DeclareMathOperator*{\Out}{\textsf{Out}}
\DeclareMathOperator*{\Ind}{\textsf{Ind}}
\DeclareMathOperator*{\simi}{\textsf{sim}}
\DeclareMathOperator*{\sgn}{\textsf{sgn}}

\DeclareMathOperator*{\temp}{\textsf{temp}}
\DeclareMathOperator*{\ess}{\textsf{ess}}
\DeclareMathOperator*{\unit}{\textsf{unit}}
\DeclareMathOperator*{\ad}{\textsf{ad}}
\DeclareMathOperator*{\scn}{\textsf{sc}}
\DeclareMathOperator*{\el}{\textsf{ell}}

\DeclareMathOperator*{\Nrd}{\textsf{Nrd}}
\DeclareMathOperator*{\Res}{\textsf{Res}}
\DeclareMathOperator*{\der}{\textsf{der}}

\DeclareMathOperator*{\std}{\textsf{std}}
\DeclareMathOperator*{\Aut}{\textsf{Aut}}
\DeclareMathOperator*{\id}{\textsf{id}}

\DeclareMathOperator*{\Hom}{\textsf{Hom}}

\newcommand{\mcA}{\mathcal{A}}

\newcommand{\vp}{\varphi}
\newcommand{\tvp}{\widetilde{\varphi}}
\newcommand{\ttau}{\widetilde{\tau}}
\newcommand{\si}{\sigma}

\newcommand{\s}{\simeq}

\newcommand{\ts}{\widetilde{\sigma}}

\newcommand{\CC}{\mathbb{C}}
\newcommand{\NN}{\mathbb{N}}

\newcommand{\QQ}{\mathbb{Q}}
\newcommand{\ZZ}{\mathbb{Z}}

\def\L{\mathcal L}

\def\cS{\mathcal S}

\newcommand{\tG}{\widetilde{G}}

\title[The local Langlands conjecture for the $p$-adic inner form of $\Sp_4$]{The local Langlands conjecture for the $p$-adic inner form of $\Sp_4$}

\author[Kwangho Choiy]{Kwangho Choiy}

\address{Department of Mathematics,
Southern Illinois University,
Carbondale, IL 62901-4408,
U.S.A.}
\email{kchoiy@siu.edu}

\begin{document}
\maketitle 
\begin{abstract}
This paper proves the local Langlands conjecture for the non quasi-split inner form $\Sp_{1,1}$ of $\Sp_4$ over a $p$-adic field of characteristic 0,
by studying the restriction of representations from the non quasi-split inner form $\GSp_{1,1}$ of $\GSp_4$ to $\Sp_{1,1}.$
The $L$-packets for $\Sp_{1,1}$ are constructed based on the earlier work on the local Langlands correspondence for $\GSp_{1,1}$ by Gan and Tantono.
To parameterize them in terms of so-called $S$-groups, we establish and utilize the local Langlands correspondence for reductive dual groups which participate in the theta correspondence with $\Sp_{1,1}$ and $\GSp_{1,1}.$
An interesting phenomenon arises when two distinct members in an $L$-packet of $\GSp_{1,1}$ are restricted to $\Sp_{1,1}.$
\end{abstract}
\setcounter{tocdepth}{1}
\tableofcontents
\section{Introduction} \label{intro}
The primary aim of this paper is to prove the local Langlands conjecture (LLC) for the non quasi-split inner form $\Sp_{1,1}$ of $\Sp_4$ over a $p$-adic field $F$ of characteristic 0.
The conjecture for a general connected reductive group $G$ over $F$ predicts that there is a surjective, finite-to-one map
\[
{\L}: \Pi(G) \longrightarrow \Phi(G),
\] 
where
$\Pi(G)$  denotes the set of isomorphism classes of irreducible smooth complex representations of $G(F),$
$\Phi(G)$ denotes the set of $\widehat{G}$-conjugacy classes of $L$-parameters, and $\widehat{G}$ is the complex
dual group of $G$ \cite{bo79}.
The map ${\L}$ is expected to satisfy a number of natural properties, for instance, it preserves certain $\gamma$-facotrs, $L$-factors, and $\epsilon$-factors, as long as they can be defined in both sides (cf. \cite{ht01, he00}).
For each $L$-parameter $\vp \in \Phi(G),$ its fiber, denoted by $\Pi_{\vp}(G),$ is called an $L$-packet for $G.$
As a part of the conjecture, of great interest is how to parameterize the $L$-packet $\Pi_{\vp}(G).$
For this purpose, we consider a central extension $\cS_{\vp, \scn}(\widehat{G})$ which sits into the following exact sequence
\begin{equation} \label{central ext in intro}
1 \longrightarrow \widehat Z_{\vp, \scn}(G)  \longrightarrow \cS_{\vp, \scn}(\widehat{G}) \longrightarrow \cS_{\vp}(\widehat{G}) \longrightarrow 1,
\end{equation} 
where $\widehat Z_{\vp, \scn}(G)$ is a certain quotient group of the center of a simply connected covering group and $\cS_{\vp}(\widehat{G})$ is the connected component group of a certain group in the adjoint group $\widehat{G}/Z(\widehat{G}).$
The precise definitions will be reviewed in Section \ref{conj str of L-packets}.
It is conjectured that there is a one-to-one correspondence
\begin{equation} \label{1-1 in intro}
\Pi_{\vp}(G) \overset{1-1}{\longleftrightarrow} \Irr(\cS_{\vp, \scn}(\widehat{G}), \zeta_{G})
\end{equation}
(see \cite{vog93, art06, art12}),
where $\Irr(\cS_{\vp, \scn}(\widehat{G}), \zeta_{G})$ denotes the set of irreducible representations of $\cS_{\vp, \scn}(\widehat{G})$ with central character $\zeta_{G},$ which is determined by $G$ via the Kottwitz isomorphism \cite{kot86}.
When $G$ is quasi-split, the character $\zeta_{G}$ turns out to be the trivial character $\mathbbm{1},$ so that 
\[
\Irr(\cS_{\vp, \scn}(\widehat{G}), \mathbbm{1}) = \Irr(\cS_{\vp}(\widehat{G})).
\]

The LLC has been proved for several cases for the last few decades \cite{gk82, rog90, he00, ht01,    gtsp10, hs11,  gt, scholze13, art12, gtan12, kmsw14, mok13}.
Among them, our proof of the LLC for $\Sp_{1,1}$ is deeply rooted in the approach to the LLC for $\SL_n$ \cite{gk82}.
In their paper, Gelbart and Knapp studied the restriction of $L$-packets of $\GL_n,$ consisting of singletons, and established the LLC for $\SL_n,$ assuming the LLC for $\GL_n$ which was proved afterwords \cite{ht01, he00, scholze13}.
It is a consequence of their study that $\cS_{\vp}(\widehat{\SL_n})$ is abelian.
This method was adjusted to the case of $\Sp_4$ \cite{gtsp10} whose $L$-packets were constructed from those of $\GSp_4$ \cite{gt} by the restriction. 
Since $L$-packets for $\GSp_4$ consist of one or two members unlike $\GL_n,$ Gan and Takeda verified that those two members do not give the same restriction in $\Sp_4$ (see \cite[Proposition 2.2]{gtsp10}).
They also determined precisely the size of $L$-packets $\Pi_{\vp}(\Sp_4)$ for any $\vp,$ so that one can have an explicit description of $\cS_{\vp}(\widehat{\Sp_4}),$ which turns out to be an elementary 2-group (cf. \cite{gp92, art12}). 

The method of restriction was also utilized in the case of non quasi-split $F$-inner forms $\SL_n'$ of $\SL_n$ \cite{hs11}.
Based on the LLC for $\GL_n$ and the local Jacquet-Langlands correspondence \cite{jl, dkv, rog83, ba08}, Hiraga and Saito established the LLC for non quasi-split $F$-inner forms $\GL_n'$ of $\GL_n$ and constructed $L$-packets for $\GL_n',$ which are all singletons.
Extending Gelbart and Knapp's work for $\SL_n$ \cite{gk82} and Labesse and Langlands' result for $\SL_2'$ \cite{ll79}, they restricted the $L$-packets of $\GL_n'$ to  $\SL_n'$ and constructed $L$-packets for $\SL_n'.$
Unlike $\SL_n$ and $\Sp_4,$ the conjectural bijection \eqref{1-1 in intro} for $\SL_n',$ which was proved in \cite{hs11} using the simple trace formula, implies that the multiplicity one property fails in this restriction.
This follows from the fact that the central extension $\cS_{\vp, \scn}(\widehat{\SL_n'})$ is not always abelian (cf. \cite{ll79, art06, choiymulti}).
Later, more general intermediate groups between a given group and its derived group were also carried out in \cite{chaoli}.

For our case of $G=\Sp_{1,1},$ as did in the previous studies above for $\SL_n,$ $\Sp_4,$ and $\SL_n',$ it is natural to begin with the non quasi-split $F$-inner form $\GSp_{1,1}$ of $\GSp_4,$ whose derived group is $\Sp_{1,1}.$ 
The LLC for $\GSp_{1,1}$ was established by Gan and Tantono in \cite{gtan12}, where their $L$-packets consist of one or two members.
Restricting these $L$-packets from $\GSp_{1,1}$ to $\Sp_{1,1},$ we construct $L$-packets for $\Sp_{1,1}.$
Interestingly, we here encounter a new phenomenon which does not occur in the aforementioned cases;
not only the multiplicity one property fails, but it is also possible that two members in an $L$-packet of $\GSp_{1,1}$ have the same restriction.
In response to this phenomenon, we study the restriction of the reductive dual groups which participate in the theta correspondence with $\GSp_{1,1}$ and $\Sp_{1,1}.$
It turns out that their group structures fit into 
a category of intermediate groups between a finite product of $\GL_n$ and a finite product of $\SL_n,$ or between their $F$-inner forms.
This allows us to use the well-developed theory of the restriction in \cite{ll79, gk82, sh83, la85, tad92,  hs11}, and \cite{chaoli} with some modifications.
We thus construct $L$-packets, prove the conjectural bijection \eqref{1-1 in intro}, and establish the LLC for these reductive dual groups.

We finally define a surjective, finite-to-one map 
\begin{equation} \label{our L in intro}
{\L}_{1,1} : {\Pi}({\Sp}_{1,1}) \longrightarrow \Phi({\Sp}_{1,1})
\end{equation}
for $G=\Sp_{1,1}.$
We then prove the conjectural bijection \eqref{1-1 in intro}, using the LLC for each reductive dual group and studying a relationship of the central extension \eqref{central ext in intro} between $\Sp_{1,1}$ and each reductive dual group (see Theorem \ref{1-1 for Sp11}).
At this point, we bring in a uniquely determined bijection between restrictions of reductive dual pairs (see Proposition \ref{bij bw consti}).
This bijection is provided by the theta correspondence and preserves the multiplicity in the restriction (see Section \ref{sec bij in res}).
Furthermore, we take the same idea to complete the conjectural bijection \eqref{1-1 in intro} for the split group $\Sp_4$ as well (see Appendix \ref{revisit}), which was also discussed in \cite[pp.3002-3003]{gtsp10} in another way.

Considering two facts: there is no Whittaker model for the non quasi-split group $\Sp_{1,1}$ and our method relies on that of $\SL_n'$ by Hiraga and Saito in \cite{hs11}, 
we should mention that our parameterization of $L$-packets for $\Sp_{1,1}$ has no natural base point and depends on the choice of a certain homomorphism $\Lambda_{\SL_n}$ described in Section \ref{for sl(m,D)}.
We also note that there is one exceptional case of $L$-parameters, denoted by `Case I-(b)' in Section \ref{para}, 
where we verify only the bijection \eqref{1-1 in intro} and establish no decomposition unlike other cases (see Theorem \ref{1-1 for Sp11}). 
This is because two members in an $L$-packet of $\GSp_{1,1}$ share the same restriction for Case I-(b), 
which leads to a difference between the dimension in restriction from $\GSp_{1,1}$ to $\Sp_{1,1}$ and the dimension of irreducible representations in  $\Irr(\cS_{\vp, \scn}(\widehat{\Sp_{1,1}}), \zeta_{\Sp_{1,1}}).$
This will be also interpreted in terms of the center of the finite group $\cS_{\vp, \scn}(\widehat{G})$ for the non quasi-split group $\Sp_{1,1}$ and its reductive dual groups (see Remark \ref{rem special for a-1b}).

The last part of the paper is dedicated to classifying the central extension \eqref{central ext in intro} for all $\vp \in \Phi(\Sp_{1,1})$ 
and illuminating all sizes of $L$-packets of $\Sp_{1,1}$ as well as all multiplicities in restriction from $\GSp_{1,1}$ to $\Sp_{1,1}.$
The group $\cS_{\vp, \scn}(\widehat{\Sp_{1,1}})$ in \eqref{central ext in intro} turns out to be isomorphic to one of the following seven groups: $\ZZ/2\ZZ,$ $(\ZZ/2\ZZ)^2,$ $\ZZ/4\ZZ,$  $\ZZ/2\ZZ \times \ZZ/4\ZZ,$  the dihedral group $\mathcal{D}_8$ of order 8, the Pauli group, and the central product of $\mathcal{D}_8$ and the quaternion group.
The size is either 1, 2, or 4, and the multiplicity is either 1, 2, or 4. 
Furthermore, we give an explicit example where the new phenomenon arises, that is, the multiplicity one property fails and two members in an $L$-packet of $\GSp_{1,1}$ have the same restriction in $\Sp_{1,1}.$

We note that, since the definitions of the local $L$-, $\epsilon$-, $\gamma$-factors on the representation side are not yet available for the non quasi-split group $\Sp_{1,1},$ our paper does not contain arguments regarding the preservation of the local factors or the uniqueness of our $\L$-map \eqref{our L in intro}.
Recently, M. Asgari and the author established the LLC for $\GSpin_4,$ $\GSpin_6,$ and their $F$-inner forms in \cite{acgspin}, 
where the local factors for generic representations of two split cases of $\GSpin_4$ and $\GSpin_6$ are available via the Langlands-Shahidi method \cite{sh90} and those local factors are proved to be preserved via the $\L$-map for $\GSpin_4$ and $\GSpin_6.$

The organization of this paper is as follows. In Section \ref{prelim}, we review basic notions and backgrounds such as inner forms, the LLC in a general setting, and the conjectural structure of $L$-packets. 
Section \ref{gp structure} describes the structure of groups under consideration and their mutual relations.
Some well-known results on the restriction are recalled and modified for our case in Section \ref{whole sec of rest}.
We prove the LLC for $\GSO_{2,2},$ $\SO_{2,2},$ and their $F$-inner forms in Section \ref{LLC so22 and its inner forms} and the LLC for $\GSO_{3,3}$, $\SO_{3,3},$ and their $F$-inner forms in Section \ref{LLC so33 and its inner forms}.
These groups are all reductive dual groups which participate in the theta correspondence with $\GSp_{1,1}$ and $\Sp_{1,1}.$ 
In Section \ref{LLC for Sp_{1,1}}, we state and prove our main result, classify all cases of the central extension \eqref{central ext in intro} for $\Sp_{1,1},$ describe all sizes of $L$-packets of $\Sp_{1,1},$ and give an explicit example in which the new phenomenon appears.
In Appendix \ref{revisit}, we apply the same method developed in Section \ref{LLC for Sp_{1,1}} to parameterize $L$-packets for $\Sp_4.$
\section{Preliminaries} \label{prelim}
In this section, we recall basic notions and backgrounds, and review the local Langlands conjecture in a general setting and the conjectural structure of $L$-packets.
\subsection{Basic definitions and backgrounds} \label{notations}
Throughout the paper, we denote by $F$ a finite extension of $\QQ_p$ for any prime $p.$  
Fix an algebraic closure $\bar F$ of 
$F.$
For any topological group $G,$ we write $Z(G)$ for the center of $G.$
We let $\pi_0(G)$ denote the group $G/G^\circ$ of connected components of $G,$ where $G^\circ$ is the identity component of $G.$

Given a connected reductive algebraic group $G$ over $
F,$ we let $\Pi(G)$ denote the set of isomorphism classes of irreducible smooth complex representations of the group $G(F)$ of $F$-points of $G.$
By abuse of notation, we identify an isomorphism class with its representative. 
We let $\Pi_{\disc}(G),$ $\Pi_{\temp}(G),$ and $\Pi_{\unit}(G)$ denote the subsets of $\Pi(G)$ which respectively consist of discrete series, tempered, and unitary representations. 
We further denote by $\Pi_{\ess, \disc}(G)$ and $\Pi_{\ess, \temp}(G)$ the subsets of $\Pi(G)$ which respectively consist of essentially square-integrable and essentially tempered representations.
Note that we have
\[
\Pi_{\disc}(G) \subset \Pi_{\temp}(G) \subset \Pi_{\unit}(G) \quad \text{and} \quad 
\Pi_{\ess, \disc}(G) \subset \Pi_{\ess, \temp}(G) \subset \Pi(G).
\]
We denote by $W_F$ the Weil group of $F$ and by $\Gamma$ the absolute Galois group $\Gal(\bar{F} / F).$ 
Let $\WD_F = W_F \times \SL_2(\CC)$ be the Weil-Deligne group.
Fixing $\Gamma$-invariant splitting data, 
we define the $L$-group of $G$ as a semi-direct product $^{L}G := \widehat{G} \rtimes \Gamma$ (see \cite[Section 2]{bo79}).
Let $\Phi(G)$ denote the set of $\widehat{G}$-conjugacy classes of $L$-parameters, i.e., admissible homomorphisms 
\[
\vp: {\WD}_F \longrightarrow {^L}G,
\]
(see \cite[Section 8.2]{bo79}). 
We denote by $C_{\vp}(\widehat{G})$ the centralizer of the image of $\vp$ in $\widehat{G}.$
Note that $C_{\vp}$ contains the center of $^{L}G$ that is the $\Gamma$-invariant group $Z(\widehat{G})^{\Gamma}.$
We say that $\vp$ is elliptic if the quotient group $C_{\vp} / Z(\widehat{G})^{\Gamma}$ is finite, and $\vp$ is tempered if $\vp(W_F)$ is bounded. 
We denote by $\Phi_{\el}(G)$ and $\Phi_{\temp}(G)$ the subset of $\Phi(G)$ which respectively consist of elliptic and tempered $L$-parameters of $G.$ 
We set $\Phi_{\disc}(G) = \Phi_{\el}(G) \cap \Phi_{\temp}(G).$
\subsection{Inner forms} \label{inner forms}
Let $G$ and $G'$ be connected reductive groups over $F.$ 
We say that $G$ and $G'$ are \textit{$F$-inner forms} with respect to an $\bar{F}$-isomorphism $\phi: G' \overset{\sim}{\rightarrow} G$ if $\phi \circ \tau(\varphi)^{-1}$ is an inner automorphism ($g \mapsto xgx^{-1}$) defined over $\bar{F}$ for all $\tau \in \Gal (\bar{F} / F)$ (see \cite[2.4(3)]{bo79} or \cite[p.280]{kot97}). If there is no confusion, we often omit the references to $F$ and $\phi.$ 

When $G$ and $G'$ are inner forms of each other, we have $^L{G} \s {^LG'}$ \cite[Section 2.4(3)]{bo79}. In particular, if $G'$ is an inner form of an $F$-split group $G$ and the action of $\Gamma$ on $\widehat{G}$ is trivial, we write $^L{G} = \widehat{G} \s {^LG'} = \widehat{G'}.$
\subsection{General notion of the local Langlands conjecture} \label{llc for general}
For any connected reductive group $G$ over $F,$ 
the local Langlands conjecture (LLC) predicts that there is a surjective, finite-to-one map
\[
{\L}: \Pi(G) \longrightarrow \Phi(G).
\]
This map is supposed to satisfy a number of natural properties, for instance, it preserves certain $\gamma$-facotrs, $L$-factors, and $\epsilon$-factors, as long as they can be defined in both sides  (cf. \cite{ht01, he00}).
Moreover, considering the fibers of the map, one can partition $\Pi(G)$ into disjoint finite subsets, called $L$-packets. 
Each packet is conjectured to be characterized by component groups in the $L$-group, which we will discuss in detail in Section \ref{conj str of L-packets}.
It is also expected that $\Phi_{\disc}(G)$ and $\Phi_{\temp}(G)$ respectively parameterize $\Pi_{\disc}(G)$ and $\Pi_{\temp}(G).$

The LLC is known for several cases: $\GL_n$ \cite{ht01, he00, scholze13}, $\SL_n$ \cite{gk82}, 
$\U_2$ and $\U_3$ \cite{rog90}, 
$F$-inner forms of $\GL_n$ and $\SL_n$ \cite{hs11},
$\GSp_4$ \cite{gt},
$\Sp_4$ \cite{gtsp10}, 
the $F$-inner form of $\GSp_4$ \cite{gtan12}, 
quasi-split classical groups \cite{art12},
unitary groups \cite{mok13}, 
non quasi-split inner forms of unitary groups \cite{kmsw14}, $\GSpin_4,$ $\GSpin_6$ and their $F$-inner forms \cite{acgspin}.
\subsection{Conjectrual structure of $L$-packets} \label{conj str of L-packets}
We denote by $\widehat{G}_{\scn}$ the simply connected cover of the derived group $\widehat{G}_{\der}$ of $\widehat{G},$ and by $\widehat{G}_{\ad}$ the adjoint group $\widehat{G}/Z(\widehat{G}).$
We consider 
\[
S_{\vp}(\widehat{G}):=C_{\vp}(\widehat G) / Z(\widehat{G})^{\Gamma} \subset \widehat{G}_{\ad}.
\]
Write $S_{\vp, \scn}(\widehat{G})$ for the full pre-image of $S_{\vp}(\widehat{G})$ via the isogeny $\widehat{G}_{\scn} \twoheadrightarrow \widehat{G}_{\ad}.$ 
We then have an exact sequence
\begin{equation} \label{exact isogeny}
1 \longrightarrow Z(\widehat G_{\scn}) \longrightarrow S_{\vp, \scn}(\widehat{G}) \longrightarrow S_{\vp}(\widehat{G}) \longrightarrow 1.
\end{equation}
We let: 
\begin{align*}
\cS_{\vp}(\widehat{G}) &:= \pi_0(S_{\vp}(\widehat{G})), \\
\cS_{\vp, \scn}(\widehat{G}) & := \pi_0(S_{\vp, \scn}(\widehat{G})), \\
\widehat Z_{\vp, \scn}(G) & := Z(\widehat G_{\scn}) / (Z(\widehat G_{\scn}) \cap S_{\vp, \scn}(\widehat{G})^{\circ}), \\
\cS_{\vp, \scn}(\widehat{G}) & := \pi_0(S_{\vp, \scn}(\widehat{G})).
\end{align*}
We then have a central extension 
\begin{equation} \label{central ext}
1 \longrightarrow \widehat Z_{\vp, \scn}(G)  \longrightarrow \cS_{\vp, \scn}(\widehat{G}) \longrightarrow \cS_{\vp}(\widehat{G}) \longrightarrow 1,
\end{equation}
(cf. \cite[(9.2.2)]{art12}).
Suppose $G$ is quasi-split and $G'$ is an $F$-inner form of $G.$
Let $\zeta_{G'}$ be a unique character on $Z(\widehat G_{\scn})$ whose restriction to $Z(\widehat G_{\scn})^{\Gamma}$ corresponds to 
the class of the $F$-inner form $G'$ of $G$ via the Kottwitz isomorphism \cite[Theorem 1.2]{kot86}.  
We denote by $\Irr(\cS_{\vp, \scn}(\widehat{G}), \zeta_{G'})$ the set of irreducible representations of $\cS_{\vp, \scn}(\widehat{G})$ 
with central character $\zeta_{G'}$ on  $Z(\widehat G_{\scn}).$ 
It is expected that, given an $L$-parameter $\vp$ for $G',$ 
there is a one-to-one correspondence between the $L$-packet $\Pi_{\vp}(G')$ associated to $\vp$ and the set $\Irr(\cS_{\vp, \scn}(\widehat{G}), \zeta_{G'})$ \cite[Section 3]{art06}. 
We note that, for the case of $G'=G,$  the character $\zeta_{G'}$ equals the trivial character $\mathbbm{1},$ so that 
\[
\Irr(\cS_{\vp, \scn}(\widehat{G}), \mathbbm{1}) = \Irr(\cS_{\vp}(\widehat{G})).
\]
In particular, if $\vp$ is elliptic, 
since $C_{\phi}(\widehat M) / Z(\widehat{M})^{\Gamma}$ is finite and $Z(\widehat{M})^{\Gamma}$ contains $S_{\phi}(\widehat M)^{\circ}$ \cite[Lemma 10.3.1]{kot84}, 
we have $S_{\vp}(\widehat{G})=\cS_{\vp}(\widehat{G})$ and
$\widehat Z_{\vp, \scn}(\widehat{G}_{\scn}) = Z(\widehat G_{\scn}).$
Thus the exact sequence \eqref{central ext} is equal to \eqref{exact isogeny}.

\subsection{Further notation} \label{further notation}
Let $m,$ $n,$ and $d$ be positive integers. For a central division algebra $D_d$ of dimension $d^2$ over $F,$ we let  $\GL_{m}(D_d)$ 
denote the group of all invertible elements of $m \times m$ matrices over $D_d.$  
Let $\SL_m(D_d)$ be the subgroup of elements of reduced norm 1 in $\GL_m(D_d).$
Note that $\GL_{m}(D_d)$ is the group of $F$-points of an algebraic group over $F$ which is an $F$-inner form of $\GL_n$ (see \cite[Sections 2.2 \& 2.3]{pr94} for details). 
By abuse of notation, we shall write $\GL_m(D_d)$ for both the $F$-inner form and the group of its $F$-points.
The same is applied to $\SL_m(D_d).$ 

For $i \in \NN,$ we denote by $H^i(F, G) := H^i(\Gal (\bar{F} / F), G(\bar{F}))$
the Galois cohomology of $G.$
Given $\pi \in \Pi(G),$ we denote by $\omega_{\pi}$ its central character. 
The cardinality of a finite set $S$ is denoted by $|S|.$ We denote by $( \; \cdot \; )^D$ the Pontryagin dual, i.e., $\Hom(\; \cdot \; , \CC^{1}),$ where $\CC^{1}$ is the unit circle group in $\CC^{\times}.$ 
We denote by $\mathbbm{1}$ the trivial character.
For any positive integer $n,$ we denote by $\mu_n$ the algebraic group, so that $\mu_n(R):= \{ r \in R : r^n = 1 \}$ with any $F$-algebra $R.$
We write $A \sqcup B$ for the disjoint union of two sets $A$ and $B.$
\section{Group structures} \label{gp structure}
In this section, we describe the structure of algebraic groups under consideration in the paper and discuss their mutual relations.
We mainly follow notation in \cite{gt, gtan12}.
Of our interest are the $F$-split groups $\GSp_4,$ $\Sp_4,$ $\GSO_{2,2},$ $\SO_{2,2},$ $\GSO_{3,3},$ $\SO_{3,3},$ and their non quasi-split $F$-inner forms.
We refer the reader to \cite[p.119]{sa71} for admissible diagrams of those $F$-inner forms.
\subsection{Symplectic cases}
We write $\GSp_{1,1}$ and $\Sp_{1,1}$ for the non quasi-split $F$-inner forms of the symplectic similitude group $\GSp_4$ and the symplectic group $\Sp_4,$ respectively.
The group $\GSp_{1,1}$ is isomorphic to $\GU(2,D)$ which is the similitude group of the unique $2$-dimensional Hermitian vector space over the quaternion division algebra $D$ over $F.$
For more details, we refer the reader to \cite[Section 2.1]{gtan12}.
Note that $\GSp_{1,1}$ is the only (up to $F$-isomorphism) non quasi-split $F$-inner form of $\GSp_4,$ 
since the set $H^1(F, \PSp_4),$ which parameterizes $F$-inner forms of $\GSp_4,$ is in bijection with $\mu_2(\CC)^D$ by the Kottwitz isomorphism \cite[Theorem 1.2]{kot86}.
Likewise, the same argument is true for $\Sp_4.$ We further note that
\[
{\Sp}_4 = ({{\GSp}_4})_{\der} \subset {\GSp}_4 \quad \text{and} \quad 
{\Sp}_{1,1} = ({{\GSp}_{1,1}})_{\der} \subset {\GSp}_{1,1},
\]
where the subscript $\der$ stands for the derived group. 
\subsection{Orthogonal cases} \label{othogonal case}
From \cite{gt, gtan12}, we recall the following isomorphisms of algebraic groups:
\begin{align*}
{\GSO}_{2,2} & \s
 ({\GL}_2 \times {\GL}_2)  / \{ (z, z^{-1}) : z \in {\GL}_1   \}, 
\\
 {\GSO}_{4,0} & \s
 ({\GL}_1(D) \times {\GL}_1(D))  / \{ (z, z^{-1}) : z \in {\GL}_1  \},
 \\
{\GSO}^*_{1,1} & \s
 ({\GL}_1(D) \times {\GL}_2)  / \{ (z, z^{-1}) : z \in {\GL}_1   \}, 
\\
{\GSO}_{3,3} & \s
 ({\GL}_4 \times {\GL}_1)  / \{ (z, z^{-2}) : z \in {\GL}_1  \},
 \\
{\GSO}^*_{3,0} & \s
 ({\GL}_1(D_4) \times {\GL}_1)  / \{ (z, z^{-2}) : z \in {\GL}_1  \},
 \\
{\GSO}(V_D) & \s
 ({\GL}_2(D) \times {\GL}_1)  / \{ (z, z^{-2}) : z \in {\GL}_1  \},
\end{align*}
where $D=D_2$ is the quaternion division algebra over $F,$ and $D_4$ is a division algebra of dimension $16$ over $F.$
As mentioned in \cite[Section 1]{gtan12}, there are only two (up to isomorphism) division algebras of dimension 16, $D_4$ and its opposite $D_4^{op}$  (their Hasse invariants in $\QQ/\ZZ$ are $1/4$ and $-1/4,$ respectively), which have canonically isomorphic multiplicative groups $D_4^{\times}$ and $(D_4^{op})^{\times}$ under the inverse map $x \mapsto x^{-1}$ from $D_4^{\times}$ to $(D_4^{op})^{\times}.$

Note that there are only two (up to isomorphism) non quasi-split $F$-inner forms of the split group ${\GSO}_{2,2},$ which are ${\GSO}_{4,0}$ and ${\GSO}^*_{1,1}.$ 
Further, there are also only two (up to isomorphism) non quasi-split $F$-inner forms of the split group ${\GSO}_{3,3},$ which are ${\GSO}^*_{3,0}$ and ${\GSO}(V_D).$ 

Since $H^1(F, \GL_m(D_d))=1$ for any central division algebra $D_d$ of dimension $d$ over $F$ with any positive integer $m$ (see \cite[Lemma 2.8]{pr94}), 
one can easily verify that the groups of $F$-points are described as follows: 
\begin{align*}
{\GSO}_{2,2}(F) & \s
 ({\GL}_2(F) \times {\GL}_2(F))  / \{ (z, z^{-1}) : z \in F^{\times}   \}, 
\\
 {\GSO}_{4,0}(F) & \s
 ({\GL}_1(D) \times {\GL}_1(D))  / \{ (z, z^{-1}) : z \in F^{\times}  \},
 \\
{\GSO}^*_{1,1}(F) & \s
 ({\GL}_1(D) \times {\GL}_2(F))  / \{ (z, z^{-1}) : z \in F^{\times}  \}, 
\\
{\GSO}_{3,3}(F) & \s
( {\GL}_4(F) \times F^{\times} ) / \{ (z, z^{-2}) : z \in F^{\times}  \},
 \\
{\GSO}^*_{3,0}(F) & \s
 ({\GL}_1(D_4) \times F^{\times}) / \{ (z, z^{-2}) : z \in F^{\times}  \},
  \\
{\GSO}(V_D)(F) & \s
 ({\GL}_2(D) \times F^{\times}) / \{ (z, z^{-2}) : z \in F^{\times}  \}.
\end{align*}

We turn to the split groups ${\SO}_{2,2},$ ${\SO}_{3,3},$ and their non quasi-split $F$-inner forms ${\SO}_{4,0},$ $\SO^*_{1,1},$ ${\SO}^*_{3,0},$ and ${\SO}(V_D).$
We have the following isomorphisms of algebraic groups:
\begin{align*}
{\SO}_{2,2} & \s
 ({\SL}_2 \times {\SL}_2 ) / \Delta \mu_2,  
\\
 {\SO}_{4,0} & \s
 ({\SL}_1(D) \times {\SL}_1(D))  / \Delta \mu_2, 
 \\
{\SO}^*_{1,1} & \s
 ({\SL}_1(D) \times {\SL}_2)  / \Delta \mu_2,  
\\
{\SO}_{3,3} & \s
 {\SL}_4 / \mu_2,
 \\
{\SO}^*_{3,0} & \s
 {\SL}_1(D_4)  / \mu_2,
  \\
{\SO}(V_D) & \s
 {\SL}_2(D)  / \mu_2,
\end{align*}
where $\Delta \mu_2$ means $\{(1,1), (-1,-1) \}.$
We note that
\[
{\SO}_{2,2} = ({\GSO}_{2,2})_{\der} \subset {\GSO}_{2,2},
\]
and the same is true for all the other groups ${\SO}_{4,0},$ ${\SO}^*_{1,1},$ ${\SO}_{3,3},$ ${\SO}^*_{3,0},$ and ${\SO}(V_D).$

Using the fact that $H^1(F, G)=1$ for any simply connected semi-simple algebraic group $G$ over $F$ \cite[Theorem 6.4]{pr94}, we further have the following exact sequences of the groups of $F$-points:
\begin{align*}
1 \longrightarrow 
({\SL}_2(F) \times {\SL}_2(F)) / \{(1,1), (-1,-1) \} 
\longrightarrow
& {\SO}_{2,2}(F)
\longrightarrow 
F^{\times} / (F^{\times})^2
\longrightarrow 1,
\\
1 \longrightarrow 
({\SL}_1(D) \times {\SL}_1(D)) / \{(1,1), (-1,-1) \} 
\longrightarrow
& {\SO}_{4,0}(F)
\longrightarrow 
F^{\times} / (F^{\times})^2
\longrightarrow 1,
\\
1 \longrightarrow 
({\SL}_1(D) \times {\SL}_2(F)) / \Delta \mu_2(F) 
\longrightarrow
& {\SO}^*_{1,1}(F)
\longrightarrow 
F^{\times} / (F^{\times})^2
\longrightarrow 1,
\\
1 \longrightarrow 
{\SL}_4(F) / \{ \pm 1 \}
\longrightarrow
& {\SO}_{3,3}(F)
\longrightarrow 
F^{\times} / (F^{\times})^2
\longrightarrow 1,
\\
1 \longrightarrow 
{\SL}_1(D_4)  / \{ \pm 1 \}
\longrightarrow
& {\SO}^*_{3,0}(F)
\longrightarrow 
F^{\times} / (F^{\times})^2
\longrightarrow 1,
\\
1 \longrightarrow 
{\SL}_2(D)  / \{ \pm 1 \}
\longrightarrow
& {\SO}(V_D)(F)
\longrightarrow 
F^{\times} / (F^{\times})^2
\longrightarrow 1.
\end{align*}
Note that $F^{\times} / (F^{\times})^2$ comes from the isomorphism $H^1(F, \mu_2) \s F^{\times} / (F^{\times})^2 \s H^1(F, \Delta \mu_2).$
Specially, for the case of ${\SO}^*_{1,1},$ considering the kernel of the similitude character
\[
{\simi}^*_{1,1} : {\GSO}^*_{1,1} \rightarrow {\GL}_1,
\]
where $\simi^*_{1,1}(\alpha, \beta) = \Nrd(\alpha \beta)$ and $\Nrd$ is the reduced norm on $\GL_1(D),$
we have
\[
1 \longrightarrow {\SO}^*_{1,1}(F)
\longrightarrow {\GSO}^*_{1,1}(F)
\overset{\simi^*_{1,1}}{\longrightarrow} F^{\times} \longrightarrow H^1(F,{\SO}^*_{1,1}) \longrightarrow \cdots.
\]
Likewise, we have the following exact sequences:
\begin{align*}
1 \longrightarrow {\SO}_{2,2}(F)
\longrightarrow 
& {\GSO}_{2,2}(F)
\overset{\simi_{2,2}}{\longrightarrow} F^{\times} \longrightarrow H^1(F,{\SO}_{2,2}) \longrightarrow \cdots,
\\
1 \longrightarrow {\SO}_{4,0}(F)
\longrightarrow 
& {\GSO}_{4,0}(F)
\overset{\simi_{4,0}}{\longrightarrow} F^{\times} \longrightarrow H^1(F,{\SO}_{4,0}) \longrightarrow \cdots,
\\
1 \longrightarrow {\SO}_{3,3}(F)
\longrightarrow 
& {\GSO}_{3,3}(F)
\overset{\simi_{3,3}}{\longrightarrow} F^{\times} \longrightarrow H^1(F,{\SO}_{3,3}) \longrightarrow \cdots,
\\
1 \longrightarrow {\SO}^*_{3,0}(F)
\longrightarrow 
& {\GSO}^*_{3,0}(F)
\overset{\simi^*_{3,0}}{\longrightarrow} F^{\times} \longrightarrow H^1(F,{\SO}^*_{3,0}) \longrightarrow \cdots,
\\
1 \longrightarrow {\SO}(V_D)(F)
\longrightarrow 
& {\GSO}(V_D)(F)
\overset{\simi_{V_D}}{\longrightarrow} F^{\times} \longrightarrow H^1(F,{\SO}(V_D)) \longrightarrow \cdots.
\end{align*}
\subsection{$L$-groups} \label{L-groups}
We recall the following descriptions of dual groups from \cite[Sections 3 and 4]{gt}:
\begin{align*}
{^L{\GSp}_{1,1}} = \widehat{{\GSp}_{1,1}}  &= {\GSp}_4(\CC) \s {\GSpin}_5(\CC),
\\
{^L{\GSO}^*_{1,1}} = \widehat{{\GSO}^*_{1,1}}  & \s {\GSpin}_4(\CC) 
\\
& \s
( {\GL}_2(\CC) \times {\GL}_2(\CC) )^\circ 
= \{ (g_1, g_2) \in {\GL}_2(\CC) \times {\GL}_2(\CC) : \det g_1 = \det g_2   \} \\
 & \s ({\Spin}_4(\CC) \times \CC^{\times})/\{ (1,1), (-1,-1)\} , \\
{^L{\GSO}^*_{3,0}} = \widehat{{\GSO}^*_{3,0}}  & = {\GSpin}_6(\CC)
\s
 \{ (g_1, g_2) \in {\GL}_4(\CC) \times {\GL}_1(\CC) : \det g_1 = (g_2)^2   \} \\
 & \s ({\Spin}_6(\CC) \times \CC^{\times})/\{ (1,1), (-1,-1)\} .
\end{align*}
The argument on $L$-groups for inner forms in Section \ref{inner forms} implies that:
\begin{align*}
{^L{\GSp}_{1,1}} & = {^L{\GSp}_4},
\\
{^L{\GSO}^*_{1,1}} & = {^L{\GSO}_{2,2}} = {^L{\GSO}_{4,0}},
\\
{^L{\GSO}^*_{3,0}} & = {^L{\GSO}(V_D)} = {^L{\GSO}_{3,3}}.
\end{align*}
We also have the following:
\begin{align*}
{^L{\Sp}_{1,1}} = {^L{\Sp}_4} = \widehat{{\Sp}_{1,1}}   &= {\SO}_5(\CC) \s {\Spin}_5(\CC),
\\
{^L{\SO}^*_{1,1}} = {^L{\SO}_{2,2}} = {^L{\SO}_{4,0}} = \widehat{{\SO}^*_{1,1}}  & = {\SO}_4(\CC) 
\s
({\SL}_2(\CC) \times {\SL}_2(\CC)) / \{ (1,1), (-1, -1) \}, 
\\
{^L{\SO}^*_{3,0}} = {^L{\SO}(V_D)} ={^L{\SO}_{3,3}} = \widehat{{\SO}^*_{3,0}}  & = {\SO}_6(\CC)
\s
{\SL}_4(\CC)/ \mu_2(\CC).
\end{align*}
We consider the following maps $\std_{1,1}$ (this was denoted by $\std$ in \cite{gtsp10}), $\std^*_{1,1},$ and $\std^*_{3,0} :$ 
\begin{align} \label{std}
{\std}_{1,1}: \widehat{{\GSp}_{1,1}} & {\longrightarrow}  \widehat{{\Sp}_{1,1}} = {\SO}_5(\CC), 
\\
{\std}^*_{1,1}: \widehat{{\GSO}^*_{1,1}} & {\longrightarrow} \widehat{{\SO}^*_{1,1}} = {\SO}_4(\CC) \s ({\SL}_2(\CC)\times {\SL}_2(\CC))/\{(1,1), (-1,-1) \},
\\
{\std}^*_{3,0}: \widehat{{\GSO}^*_{3,0}} & {\longrightarrow} \widehat{{\SO}^*_{3,0}} = {\SO}_6(\CC) \s {\SL}_4(\CC)/\mu_2(\CC),
\end{align}
which are respectively induced from the canonical inclusions:
\[
{\Sp}_{1,1}  \hookrightarrow {\GSp}_{1,1}
, \quad
{\SO}^*_{1,1}  \hookrightarrow {\GSO}^*_{1,1}
, \quad
{\SO}^*_{3,0}  \hookrightarrow {\GSO}^*_{3,0}.
\]
Note that these inclusions come from: 
\begin{align*}
1 \longrightarrow {\Sp}_{1,1}
\longrightarrow 
& {\GSp}_{1,1}
\overset{\simi_{1,1}}{\longrightarrow} {\GL}_1
\longrightarrow 1, 
\\
1 \longrightarrow {\SO}^*_{1,1}
\longrightarrow 
& {\GSO}^*_{1,1}
\overset{\simi^*_{1,1}}{\longrightarrow} {\GL}_1
\longrightarrow 1, 
\\
1 \longrightarrow {\SO}^*_{3,0}
\longrightarrow 
& {\GSO}^*_{3,0}
\overset{\simi^*_{3,0}}{\longrightarrow} {\GL}_1
\longrightarrow 1. 
\end{align*}
Therefore, we have the following commutative diagram of dual groups
\[
    \xymatrix{
        \widehat{{\GSO}^*_{1,1}} \ar@{^{(}->}[r] \ar[d]_{\std^*_{1,1}} 
        & \widehat{{\GSp}_{1,1}}  \ar@{^{(}->}[r] \ar[d]_{\std_{1,1}} & \widehat{{\GSO}^*_{3,0}} \ar@{^{(}->}[r] \ar[d]_{\std^*_{3,0}} & {\GL}_4(\CC) \times {\GL}_1(\CC)  \\
        \widehat{{\SO}^*_{1,1}} \ar@{^{(}->}[r] 
        & \widehat{{\Sp}_{1,1}}  \ar@{^{(}->}[r] & \widehat{{\SO}^*_{3,0}} \ar@{^{(}->}[r] & {\GSO}_6(\CC),
        }
\]
where all standard maps $\std^*_{1,1},$ $\std_{1,1},$ $\std^*_{3,0}$ are surjective. 
Further, we have the following exact sequences:
\begin{equation} \label{surjective prs22}
1 \longrightarrow \CC^{\times}
\longrightarrow
\widehat{{\GSO}^*_{1,1}} \overset{{\std}^*_{1,1}}{\longrightarrow} \widehat{{\SO}^*_{1,1}} \longrightarrow 1,
\end{equation}
\begin{equation} \label{surjective prs11}
1 \longrightarrow \CC^{\times}
\longrightarrow
\widehat{{\GSp}_{1,1}} \overset{{\std}_{1,1}}{\longrightarrow} \widehat{{\Sp}_{1,1}} \longrightarrow 1,
\end{equation}
\begin{equation} \label{surjective prs30}
1 \longrightarrow \CC^{\times}
\longrightarrow
\widehat{{\GSO}^*_{3,0}} \overset{{\std}^*_{3,0}}{\longrightarrow} \widehat{{\SO}^*_{3,0}} \longrightarrow 1.
\end{equation}
We also use the notation $\std_{4}$ for $\GSp_4,$ $\std_{2,2}$ for ${\GSO}_{2,2},$ $\std_{4,0}$ for ${\GSO}_{4,0},$ and $\std_{3,3}$ for ${\GSO}_{3,3}.$
Again, the argument on $L$-groups for inner forms in Section \ref{inner forms} implies that:
\begin{align*}
{\std}_{1,1} & = {\std}_{4},
\\
{\std}^*_{1,1}& = {\std}_{2,2}= {\std}_{4,0},
\\
{\std}^*_{3,0} & ={\std}_{V_D}={\std}_{3,3}.
\end{align*}
\section{Restriction} \label{whole sec of rest}
In this section, we will recall and adjust some well-known results about the restriction.
\subsection{Results of  Gelbart-Knapp, Tadi\'c, and Hiraga-Saito} \label{results in rest}
For a moment, we let $G$ and $\tG$ denote connected reductive algebraic groups over $F$ satisfying the property that
\begin{equation*} 
G_{\der} = \tG_{\der} \subseteq G \subseteq \tG,
\end{equation*}
where the subscript $\der$ stands for the derived group. 
Given $\sigma \in \Pi(G),$ by \cite[Lemma 2.3]{gk82} and \cite[Proposition 2.2]{tad92}, there exists $\ts \in \Pi(\tG)$ such that 
\[
\sigma \hookrightarrow {\Res}_{G}^{\tG}(\ts),
\]
that is, $\sigma$ is an irreducible constituent in the restriction ${\Res}_{G}^{\tG}(\ts)$ of $\ts$ from $\tG(F)$ to $G(F).$
We write both $\Pi_{\sigma}(G)$ and $\Pi_{\ts}(G)$ for the set of equivalence classes of all irreducible constituents of ${\Res}_{G}^{\tG}(\ts).$ 
It follows from \cite[Lemma 2.1]{gk82} and \cite[Proposition 2.4 \& Corollary 2.5]{tad92} that $\Pi_{\sigma}(G)$ is finite and independent of the choice of the lifting $\ts \in \Pi(\tG).$ 
Further, for any irreducible constituents $\si_1$ and  $\si_2$ in ${\Res}_{G}^{\tG}(\ts),$ it is clear that $\Pi_{\si_1}(G) = \Pi_{\si_2}(G).$
\begin{rem} \label{rem in rest from to}
From \cite[Proposition 2.7]{tad92}, we note that any member in $\Pi_{\ts}(G)$ is supercuspidal if and only if $\ts$ is. The same is true for essentially square-integrable and essentially tempered representations.
\end{rem}
\begin{pro}(\cite[Lemma 2.4]{gk82}; \cite[Corollary 2.5]{tad92}) \label{pro for lifting}
Given $\ts_1$, $\ts_2 \in \Pi(\tG),$ the following statements are equivalent:
\begin{itemize}
  \item[1)]  There exists a character $\chi \in (\tG / G)^D$ such that $\ts_1 \s \ts_2 \otimes \chi;$
  \item[2)]  $\Pi_{\ts_1}(G) \cap \Pi_{\ts_2}(G) \neq \emptyset;$
  \item[3)]  $\Pi_{\ts_1}(G) = \Pi_{\ts_2}(G).$
\end{itemize}
\qed
\end{pro}
The restriction ${\Res}_{G}^{\tG}(\ts)$ is completely reducible by \cite[Lemma 2.1]{gk82} and \cite[Lemma 2.1]{tad92}, we have the following decomposition 
\begin{equation*} 
{\Res}_{G}^{\tG}(\ts) = m \bigoplus _{\tau \in \Pi_{\ts}(G)}  \tau.
\end{equation*}
Here, the positive integer $m$ denotes the common multiplicity over $\tau \in \Pi_{\sigma}(G)$ (see \cite[Lemma 2.1(b)]{gk82}). 
Given $\ts \in \Pi(\tG),$ we define 
\begin{equation} \label{def of I sigma}
I(\ts) := \{ \chi \in (\tG(F)/G(F))^D : \ts \otimes \chi \s \ts  \}.
\end{equation}
We later use $I^G(\ts)$ to emphasize groups (see Section \ref{pf of 1-1 for Sp11}).
Considering the dimension of the $\CC$-vector space $\End_G({\Res}_{G}^{\tG}(\ts)),$ we have the following equality (cf. \cite[Proposition 3.2]{choiymulti})
\begin{equation} \label{multi and characters}
m^2 \cdot |\Pi_{\sigma}(G)| = |I(\ts)|.
\end{equation}
Let $\chi \in I(\ts)$ be given. 
Based on \cite[Chapter 2]{hs11}, since $\ts \s \ts \otimes \chi,$ 
we have a non-zero endomorphism $I_{\chi} \in {\Aut}_{\CC}(V_{\ts})$ such that
\[
I_{\chi} \circ (\ts \otimes \chi)  = \ts \circ I_{\chi}.
\]
For each $z \in \CC^\times,$ we denote by $z \cdot \id_{V_{\ts}}$ the scalar endomorphism $\tilde v \mapsto z \cdot \tilde v$ for $v \in V_{\ts}.$ 
So, we identify $\CC^\times$ and the subgroup of ${\Aut}_{\CC}(V_{\ts})$ consisting of $z \cdot \id_{V_{\ts}}.$
We now define $\mcA(\ts)$ as the subgroup of ${\Aut}_{\CC}(V_{\ts})$ generated by $\{I_{\chi} : \chi \in I(\ts) \}$ and $\CC^\times.$ Then the map $I_{\chi} \mapsto \chi$ induces the following exact sequence
\begin{equation*} 
1 \longrightarrow \CC^\times \longrightarrow \mcA(\ts) \longrightarrow I(\ts) \longrightarrow 1.
\end{equation*}
We equip $\mcA(\ts)$ with the topology 
such that the induced topology on $\CC^\times$ is the induced topology by the usual topology on $\CC$ and 
such that the projection $\mcA(\ts) \rightarrow I(\ts)$ is continuous with respect to the discrete topology on $I(\ts).$  
We denote by $\Irr(\mcA(\ts), \id)$ the set of isomorphism classes of irreducible smooth representations 
of the group $\mcA(\ts)$ such that $z\cdot \id_{V_{\ts}} \in \CC^\times$ acts as the scalar $z.$ 
By \cite[Corollary 2.10]{hs11}, we then have an isomorphism
\begin{equation}  \label{useful decomp}
V_{\ts} ~ ~ \s \bigoplus_{\xi \in \Irr(\mcA(\ts), \id)} \xi \boxtimes  \si_{\xi}
\end{equation}
as representations of the semi-direct product $\mcA(\ts) \rtimes G(F).$ It follows that there is a bijection
\begin{equation}  \label{bij 333}
\Irr(\mcA(\ts), \id) \overset{\s}{\longrightarrow} \Pi_{\ts}(G),
\end{equation}
sending $\xi \mapsto  \si_{\xi}.$ 
We denote by $\xi_{\si}$ the inverse of $\sigma$ via the correspondence \eqref{bij 333}.
\begin{rem} \label{dim1=dim2}
Given $\ts \in \Pi_{\disc}(\tG),$ since the multiplicity $m$ is common, the isomorphism \eqref{useful decomp} implies that 
\[
m={\dimi}\xi_{\si_1} = {\dimi}\xi_{\si_2}
\]
for any $\si_1, \si_2 \in \Pi_{\ts}(G).$
\end{rem}
\subsection{Useful arguments} \label{thm by Labesse}
We discuss a few arguments which will be used in Sections \ref{LLC so22 and its inner forms}, \ref{LLC so33 and its inner forms}, and \ref{LLC for Sp_{1,1}}. 
We first recall a theorem of Labesse in \cite{la85} which verifies the existence of a lifting of a given $L$-parameter in the following setting. 
Let $G$ and $\tG$ be connected reductive algebraic groups over $F$ with an exact sequence of connected components of $L$-groups
\[
1 \longrightarrow \widehat{S} \longrightarrow \widehat{\tG} \overset{pr}{\longrightarrow}\widehat{G} \longrightarrow 1,
\]
where $\widehat{S}$ is a central torus in $\widehat{G},$ and the surjective homomorphism $pr$ is compatible with $\Gamma$-actions on $\widehat{\tG}$ and $\widehat{G}.$
\begin{thm} (\cite[Th\'{e}or\`{e}m 8.1]{la85}) \label{thm by Labesse}
For any $\vp \in \Pi(G),$ there exists $\tvp \in \Pi(\tG)$ such that
\[
\vp = \tvp \circ pr.
\] \qed
\end{thm} 
We note that this result has been also discussed in \cite{weil74, he80, gtsp10} for the case of $G=SL_n$ and $\tG=GL_n.$ 

Second, we recall a lemma of Chao and Li in \cite{chaoli}.
\begin{lm} (\cite[Lemma 5.3.4]{chaoli}) \label{lemma of Chaoli}
With the notation in Section \ref{conj str of L-packets}, given $\vp \in \Pi(G)$ and $\tvp \in \Pi(\tG)$ with $\vp = \tvp \circ pr$ as in Theorem \ref{thm by Labesse},  we have an exact sequence of finite groups
\[
\cS_{\tvp}(\widehat{\tG}) \longrightarrow \cS_{\vp}(\widehat{G}) \longrightarrow X(\tvp) \longrightarrow 1,
\]
where $X(\tvp) := \{ a \in H^1(W_F, \widehat{S}) : a \tvp \s \tvp \mbox{ in } \widehat{\tG} \}.$ \qed
\end{lm}
Along with the definition $I(\ts)$ in \eqref{def of I sigma}, given $\tvp \in \Phi(\tG),$ we let
\begin{equation} \label{def of I vp}
I(\tvp) := \{
\chi \in (F^{\times})^D : \tvp \chi \s \tvp \mbox{ in } \widehat{\tG} 
\},
\end{equation}
where $\chi$ is considered as a $L$-parameter in $\Phi(\tG)$ via the local class field theory.
We later use $X^G(\tvp)$ and $I^G(\tvp)$ to emphasize groups (see Section \ref{pf of 1-1 for Sp11}).
\begin{lm} \label{the case of singleton}
Suppose that an $L$-packet for $\tvp \in \Phi(\tG)$ is constructed as a singleton $\{ \ts \}$ and further suppose that $\widehat{S} \s \CC^{\times}.$
 Then we have
\[
X(\tvp) \s I(\tvp).
\]
\end{lm}
\begin{proof}
This is immediate from the LLC for $GL_1$ which asserts
$(F^{\times})^D \s H^1(W_F, \widehat{S}).$ 
\end{proof}

We end this subsection by making an argument on the group of connected components. 
\begin{lm} \label{same connected components}
Let $A$ and $B$ be algebraic groups over $F$ such that $A$ is a normal subgroup of $B$ of finite index.
Then the connected components $A^{\circ}$ and $B^{\circ}$ are identical. 
Further, $\pi_0(A)$ is again a subgroup of $\pi_0(B).$
\end{lm}
\begin{proof}
It is well known that $A^{\circ}$ and $B^{\circ}$ are open, closed, normal subgroups of $A$ and $B,$ respectively.
Let $b\in B$ be given. 
Since $bA^{\circ}b^{-1}$ is an open and connected subgroup containing the identity, we have $bA^{\circ}b^{-1} \subset B^{\circ}.$
Note that the index $[A^{\circ}:B^{\circ}]$ is finite.
If $bA^{\circ}b^{-1}$ is a proper subgroup of $ B^{\circ},$ then $B^{\circ}$ is disconnected into finite connected open cosets of $bA^{\circ}b^{-1},$ which is impossible. 
Thus, we must have $bA^{\circ}b^{-1} = B^{\circ},$ which implies that $A^{\circ} = B^{\circ}.$ 
Further, since the index $[A:B]$ is finite and $A^{\circ} = B^{\circ},$ it follows that $\pi_0(A)$ is a subgroup of $\pi_0(B).$
\end{proof}
\subsection{Hiraga-Saito's work on $L$-packets for inner forms of $SL_n$} \label{for sl(m,D)}
We recall a result in \cite[Chapter 12]{hs11} about the internal structure of $L$-packets for an inner form $G'=\SL_m(D_d)$ of $G=\SL_n$ with $n=md.$
Note that $^L G=\widehat{G} = {^L} G' = \widehat{G'} = \PGL_n(\CC),$ since $\Gamma$ acts trivially. 
We further have
\[
{Z}(\widehat{G}_{\scn}) = \bold{\mu_n}(\CC) \quad \text{and} \quad Z(\widehat{G})^{\Gamma} = 1.
\]
Given $\vp \in \Phi(G'),$ we have the following exact sequence
\begin{equation*} 
1 \longrightarrow \widehat Z_{\vp, \scn}(G')  \longrightarrow \cS_{\vp, \scn}(\widehat{G'}) \longrightarrow \cS_{\vp}(\widehat{G'}) \longrightarrow 1.
\end{equation*}
Note that $\widehat Z_{\vp, \scn}(G') = \bold{\mu_n}(\CC) / (\bold{\mu_n}(\CC)  \cap S_{\vp, \scn}(\widehat{G'})^{\circ})$ by definition.
We fix a character $\zeta_{G'}$ of $\bold{\mu_n}(\CC)$ which corresponds to the inner form $G'$ of $G$ via the Kottwitz isomorphism \cite[Theorem 1.2]{kot86}.
Note that when $d=1,$ $G'=G$ and $\zeta_{G} = \mathbbm{1}.$
Set $\tG = \GL_n$ and $\tG' = \GL_m(D_d).$
We consider the following exact sequence
\[
1 \longrightarrow \CC^{\times} \longrightarrow \widehat{\tG'} = {\GL}_n(\CC) \overset{pr}{\longrightarrow} \widehat{G'}={\PGL}_n(\CC) \longrightarrow 1.
\]
By Theorem \ref{thm by Labesse}, we have an $L$-parameter $\widetilde{\vp} \in \Phi(\tG')$ 
\[
\widetilde{\vp} : {\WD}_F \rightarrow {\GL}_n(\CC)
\]
such that ${pr} \circ \widetilde{\vp}=\vp$ (see also \cite{weil74, he80, chgo12}).
By the local Langlands correspondence for $\tG'$ \cite[Chapter 11]{hs11}, 
we have a unique irreducible representation $\ts \in \Pi(\tG')$ associated to the $L$-parameter $\widetilde{\vp}.$ 
The $L$-packet $\Pi_{\vp}(G')$ thus equals the set $\Pi_{\ts}(G')$ (see Section \ref{results in rest}).
\begin{lm} (\cite[Lemma 12.5]{hs11}) 
There is a homomorphism 
$\Lambda_{SL_n} : \cS_{\vp, \scn}(\widehat{G'}) \rightarrow \mcA(\ts)$ 
(unique up to 1-dimensional character of $\cS_{\varphi}(G')$) with the following commutative diagram
\begin{equation} \label{a diagram}
\begin{CD}
1 @>>> \widehat Z_{\vp, \scn}(G')  @>>> \cS_{\vp, \scn}(\widehat{G'}) @>>> \cS_{\vp}(\widehat{G'})  @>>> 1 \\
@. @VV{\zeta_{G}}V @VV{\Lambda_{SL_n}}V @VV{\s}V @.\\
1 @>>> \CC^\times @>>> \mcA(\ts) @>>> I(\ts) @>>> 1.
\end{CD} 
\end{equation} \qed
\end{lm}
Combining \eqref{useful decomp} and \eqref{a diagram}, \cite[Lemma 12.6]{hs11} states that there is a bijection
\begin{equation} \label{bij sl}
\Pi_{\vp}(G') \overset{1-1}{\longleftrightarrow} \Irr(\cS_{\vp, \scn}(\widehat{G'}), \zeta_{G'}),
\end{equation}
such that we have an isomorphism
\begin{equation*} 
V_{\ts} ~ ~ \s 
\bigoplus_{\rho \in \Irr(\cS_{\vp, \scn}(\widehat{G'}), \zeta_{G'})} \rho \boxtimes  \si_{\rho}
\end{equation*}
as representations of $\cS_{\vp, \scn}(\widehat{G'}) \rtimes G'(F),$
where $\si_{\rho}$ denotes the image of $\rho$ via the bijection \eqref{bij sl}.
It thus follows from \cite[p.5]{hs11} that 
\begin{equation} \label{dim=dim=multi}
\dimi \xi_{\si} = \dimi \rho_{\si},
\end{equation}
where $\rho_{\si}$ is the image of $\sigma$ via the bijection \eqref{bij sl},
which implies that ${\dimi}\rho_{\si_1} = {\dimi}\rho_{\si_2}$ for any $\si_1, \si_2 \in \Pi_{\vp}(G').$ 
\begin{rem} \label{multi for  SL}
From \eqref{a diagram} and \eqref{dim=dim=multi}, the multiplicity in the restriction from $\GL_n(D)$ to $\SL_n(D)$ is controlled by the following two factors: the character $\zeta_{G},$ uniquely determined by a given inner form $G,$ and the group $S_{\vp},$ determined by a given $L-$parameter $\vp.$
\end{rem}
\begin{rem} \label{rem general HS}
All above arguments can be obviously applicable to the case of $G^* = \SL_{n_1} \times \cdots \times \SL_{n_r}.$
\end{rem}
\subsection{A bijection via theta correspondence} \label{sec bij in res}
We recall a bijection between two sets of irreducible constituents in two restrictions via the theta correspondence.
For a moment, we employ the notation $\GU(V_{2n}),$ $\U(V_{2n}),$ $\GU(W_{m}),$ and $\U(W_{m})$ in \cite[Section 2]{gtan12},  
where  $V_{2n}$ and $W_{m}$ respectively denote a quaternionic Hermitian  and skew-Hermitian space over a quaternion $F$-algebra with some positive integers $n$ and $m.$ 
These represent all non qausi-split inner forms described in Section \ref{gp structure}, 
We fix a non-trivial additive character $\psi$ of $F.$ As in \cite[Section 3]{gtan12}, we consider the Weil representation of $\U(V_{2n})\times \U(W_{m})$ and its extension to $R=\GU(V_{2n})\times \GU(W_{m}),$ which respectively give theta correspondences between $\U(V_{2n})$ and $\U(W_{m})$ and between $\GU(V_{2n})$ and $\GU(W_{m}).$ 
\begin{pro}(\cite[Proposition 3.3]{gtan12}) \label{bij bw consti}
Let $\pi \in \Pi(\GU(V_{2n}))$ be given. Set \[
{\Res}^{{\GU}(V_{2n})}_{{\U}(V_{2n})} (\pi) = k \cdot \bigoplus_i \tau_i
\]
for some positive integer $k$ (the common multiplicity) and $\tau_i \in \Pi(\U(V_{2n})).$ Suppose the big theta lift $\Theta(\pi)$ of $\pi$ is nonzero. Then we have the following.

(i) There is an isomorphism
\[
\theta(\pi) \s k \cdot \bigoplus_i \theta_{\psi}(\tau_i)
\]
as representations of $\U(W_{m}).$ Moreover, $\Theta(\pi) = \theta(\pi)$ is semisimple if $\Theta_{\psi}(\tau_i) = \theta_{\psi}(\tau_i)$ is semisimple for all $i.$

(ii) There is a (uniquely determined) bijection
\[
{\Res}^{{\GU}(V_{2n})}_{{\U}(V_{2n})} (\pi) \overset{f}{\longrightarrow} 
{\Res}^{{\GU}(W_{m})}_{{\U}(W_{m})} (\theta(\pi)),
\]
sending $\tau \mapsto \theta_{\psi}(\tau) =: f(\tau),$ which immediately implies the following bijection
\[
\Pi_{\pi}({\U}(V_{2n})) \overset{f}{\longrightarrow}  \Pi_{\theta(\pi)}({\U}(W_{m})).
\]

(iii) The above statements (i) and (ii) are true for $\GU(W_{m})$ and $\U(W_{m})$ when $\GU(V_{2n})$ and $\U(V_{2n})$ are respectively replaced by $\GU(W_{m})$ and $\U(W_{m}).$
\qed
\end{pro}
\section{LLC for $\SO_{2,2}$ and its inner forms} \label{LLC so22 and its inner forms} 
We present the LLC for $\GSO_{2,2}$ and its $F$-inner forms, and establish the LLC for $\SO_{2,2}$ and its $F$-inner forms.
\subsection{The cases of $\GSO_{2,2}$ and its inner forms} \label{LLC GSO22}
From the isomorphism in Section \ref{othogonal case}
\[
{\GSO}_{2,2}(F) \s
 ({\GL}_2(F) \times {\GL}_2(F)) / \{ (z, z^{-1}) : z \in F^{\times}   \},
\]
one can notice that any irreducible admissible representation of ${\GSO}_{2,2}(F)$ is of the form $\ttau_1 \boxtimes \ttau_2,$ where $\ttau_1$ and $\ttau_2$ are in $\Pi(GL_2)$ with the same central character (cf. \cite[Section 1]{gtrep}).
This implies that
\[
\Pi({\GSO}_{2,2})=
\{
\ttau_1 \boxtimes \ttau_2 : \ttau_1, \ttau_2 \in \Pi({\GL}_2) ~~\text{with}~~ \omega_{\ttau_1}=\omega_{\ttau_2}
\}.
\]
Further, due to the form of $L$-group $\widehat{{\GSO}_{2,2}}$ in Section \ref{othogonal case}, we note that
\begin{equation*} 
\Phi({\GSO}_{2,2})=
\{
\tvp_1 \oplus \tvp_2 : \tvp_1, \tvp_2 \in \Phi({\GL}_2) ~~\text{with}~~ \det \tvp_1=\det \tvp_2
\}.
\end{equation*}
Thus, by the LLC for $\GL_n$ \cite{ht01, he00, scholze13}, there is a surjective, one-to-one map 
\[
L_{2,2} : {\Pi}({\GSO}_{2,2}) \longrightarrow \Phi({\GSO}_{2,2}).
\]
For non quasi-split $F$-inner forms of ${\GSO}_{2,2},$ we again recall the isomorphism in Section \ref{othogonal case}:
\begin{align*}
 {\GSO}_{4,0}(F) & \s
 ({\GL}_1(D) \times {\GL}_1(D))  / \{ (z, z^{-1}) : z \in F^{\times}  \},
 \\
{\GSO}^*_{1,1}(F) & \s
 ({\GL}_1(D) \times {\GL}_2(F))  / \{ (z, z^{-1}) : z \in F^{\times}   \}. 
\end{align*}
Similarly, we have:
\begin{align*}
\Pi({\GSO}_{4,0}) & =
\{
\ttau_1 \boxtimes \ttau_2  : \ttau_1, \ttau_2 \in \Pi({\GL}_1(D)) ~~\text{with}~~ \omega_{\ttau_1}=\omega_{\ttau_2}
\},
\\
\Pi({\GSO}^*_{1,1}) & =
\{
\ttau_1 \boxtimes \ttau_2  : \ttau_1 \in \Pi({\GL}_1(D)), \ttau_2 \in \Pi({\GL}_2) ~~\text{with}~~ \omega_{\ttau_1}=\omega_{\ttau_2}
\}.
\end{align*}
Further, due to the form of $L$-groups in Section \ref{othogonal case}, we note that:
\begin{align*}
\Phi({\GSO}_{4,0}) & =
\{
\tvp_1 \oplus \tvp_2 : \tvp_1, \tvp_2 \in \Phi({\GL}_1(D)) ~~\text{with}~~ \det \tvp_1=\det \tvp_2
\},
\\
\Phi({\GSO}^*_{1,1}) & =
\{
\tvp_1 \oplus \tvp_2 : \tvp_1 \in \Phi({\GL}_1(D)), \tvp_2 \in \Phi({\GL}_2) ~~\text{with}~~ \det \tvp_1=\det \tvp_2
\}.
\end{align*}
Thus, by the LLC for $\GL_n$ \cite{ht01, he00, scholze13} and for $\GL_m(D)$ \cite{hs11}, there are surjective, one-to-one maps: 
\begin{align*}
L_{4,0} & : {\Pi}({\GSO}_{4,0}) \longrightarrow \Phi({\GSO}_{4,0}),
\\
L^*_{1,1} & : {\Pi}({\GSO}^*_{1,1}) \longrightarrow \Phi({\GSO}^*_{1,1}).
\end{align*}
Since three maps $L_{2,2},$ $L_{4,0},$ and $L^*_{1,1}$ are one-to-one, each fiber gives rise to $L$-packets for $\GSO_{2,2},$ $\GSO_{4,0},$ and $\GSO^*_{1,1},$ which are all singletons.
For simplicity of notation, we write ${\GSO}_\dagger$ for ${\GSO}_{2,2},$ ${\GSO}_{4,0},$ and ${\GSO}^*_{1,1}.$ 
Further, recalling the notation in Section \ref{conj str of L-packets}, we note that:
\begin{align*}
S_{\tvp}(\widehat{{\GSO}_\dagger}) & \subset (\widehat{{\GSO}_\dagger})_{\ad} \s {\PSO}_4(\CC)\s {\PSL}_2(\CC)\times {\PSL}_2(\CC), \\
S_{\tvp, \scn}(\widehat{{\GSO}_\dagger}) & \subset (\widehat{{\GSO}_\dagger})_{\scn} \s  {\Spin}_4(\CC) \s {\SL}_2(\CC)\times {\SL}_2(\CC).
\end{align*}

\subsection{Construction of $L$-packets for $\SO_{2,2}$ and its inner forms} \label{const of L-packet for SO22}
Given $\sigma \in \Pi({\SO}_{2,2}),$ from the arguments in Section \ref{results in rest}, there is a lifting $\ts \in \Pi({\GSO}_{2,2})$ such that 
\[
\sigma \hookrightarrow {\Res}_{{\SO}_{2,2}}^{{\GSO}_{2,2}}(\ts).
\]
We define a map 
\begin{equation*} 
{\L}_{2,2} : {\Pi}({\SO}_{2,2}) \longrightarrow \Phi({\SO}_{2,2})
\end{equation*}
by $\L_{2,2}(\sigma) := \std_{2,2}(L(\ts)).$ 
Note that $\L_{2,2}$ is not depending on the choice of the lifting $\ts,$ since another lifting must be of the form $\ts \otimes \chi$ for some quasi-character $\chi$ of $F^{\times}$ by Proposition \ref{pro for lifting} and $L_{2,2}(\ts \otimes \chi) = L_{2,2}(\ts) \otimes \chi$ for any quasi-character $\chi$ of $F^{\times}$  \cite{ht01, he00}.
Thus, the map $\L_{2,2}$ is well-defined. 
This is an analogue of the LLC for ${\SL}_n$ \cite{gk82}.

Furthermore, $\L_{2,2}$ is a surjective, 
since any $\vp \in \Phi({\SO}_{2,2})$ can be lifted to some $\tvp \in \Phi({\GSO}_{2,2})$ by \eqref{surjective prs22} and Theorem \ref{thm by Labesse}. 
For each $\vp \in \Phi({\SO}_{2,2}),$ the fiber is given by
\begin{equation*} 
\Pi_{\vp}({\SO}_{2,2}) = \Pi_{\ts}({\SO}_{2,2}),
\end{equation*}
where $\ts$ is the unique member in $\Pi_{\tvp}({\GSO}_{2,2})$ and $\tvp$ lies in $\Phi({\GSO}_{2,2})$ such that $\std_{2,2} \circ \tvp =\vp.$ 
Due to \cite{ht01, he00} and Proposition \ref{pro for lifting}, the fiber does not depend on the choice of $\tvp.$
This forms an $L$-packet for ${\SO}_{2,2}.$

Similarly, given $\sigma_{4,0} \in \Pi({\SO}_{4,0}),$ there is a lifting $\ts_{4,0} \in \Pi({\GSO}_{4,0})$ such that 
\[
\sigma_{4,0} \hookrightarrow {\Res}_{{\SO}_{4,0}}^{{\GSO}_{4,0}}(\ts_{4,0}).
\]
We define a map  
\begin{equation*} 
{\L}_{4,0} : {\Pi}({\SO}_{4,0}) \longrightarrow \Phi({\SO}_{4,0})
\end{equation*}
by $\L_{4,0}(\sigma_{4,0}) := \std_{4,0}(L(\ts_{4,0})).$
In the same way with $\L_{2,2},$ it turns out that $\L_{4,0}$ is a well-defined, surjective, and finite-to-one map.
Likewise, we have a surjective, finite-to-one map
\begin{equation*} 
{\L}^*_{1,1} : {\Pi}({\SO}^*_{1,1}) \longrightarrow \Phi({\SO}^*_{1,1})
\end{equation*}
by $\L^*_{1,1}(\sigma^*_{1,1}) := \std^*_{1,1}(L(\ts^*_{1,1})),$
where $\sigma^*_{1,1}$ and $\ts^*_{1,1}$ are corresponding representations for ${\SO}^*_{1,1}$ and ${\GSO}^*_{1,1},$ respectively. 
For each $\vp \in \Phi({\SO}_{4,0}),$ the $L$-packet is given by
\begin{equation*} 
\Pi_{\vp}({\SO}_{4,0}) = \Pi_{\ts_{4,0}}({\SO}_{4,0}).
\end{equation*}
Likewise, for each $\vp \in \Phi({\SO}^*_{1,1}),$ we have the $L$-packet
\begin{equation*} 
\Pi_{\vp}({\SO}^*_{1,1}) = \Pi_{\ts^*_{1,1}}({\SO}^*_{1,1}).
\end{equation*}
Again, due to \cite{ht01, he00, hs11} and Proposition \ref{pro for lifting}, each $L$-packet does not depend on the choice of $\tvp.$
\subsection{Internal structure of $L$-packets for $\SO_{2,2}$ and its inner forms} \label{para for SO22}
We continue with the notation in Section \ref{conj str of L-packets}. 
For simplicity of notation, we shall write ${\SO}_\dagger$ for 
${\SO}_{2,2},$ ${\SO}_{4,0},$ and ${\SO}^*_{1,1}.$ 
Recall from Section \ref{L-groups} that
\[
\widehat{{\SO}_\dagger} \s {\SO}_4(\CC) \s ({\SL}_2(\CC)\times {\SL}_2(\CC) ) /\{(1,1), (-1,-1) \}. 
\]
Note that: 
\[
(\widehat{{\SO}_\dagger})_{\ad} ={\PSO}_4(\CC), ~~~ 
(\widehat{{\SO}_\dagger})_{\scn} ={\Spin}_4(\CC), ~~~
Z((\widehat{{\SO}_\dagger})_{\scn}) = Z((\widehat{{\SO}_\dagger})_{\scn})^{\Gamma} \s \mu_2(\CC) \times \mu_2(\CC).
\]
Let $\vp \in \Phi({\SO}_\dagger)$ be given.
We fix a lifting $\tvp \in \Phi({\GSO}_\dagger)$ via the surjective map $\widehat{{\GSO}_\dagger}  \twoheadrightarrow \widehat{{\SO}_\dagger}$ (see Theorem \ref{thm by Labesse}). 
We note that:
\begin{align*}
S_{\vp}(\widehat{{\SO}_\dagger}) &\subset {\PSO}_4(\CC) \s {\PSL}_2(\CC)\times {\PSL}_2(\CC),  \\
\cS_{\vp, \scn}(\widehat{{\SO}_\dagger}) &\subset {\Spin}_4(\CC) \s {\SL}_2(\CC)\times {\SL}_2(\CC).
\end{align*}
One can then have a central extension 
\begin{equation*} 
1 \longrightarrow \widehat Z_{\vp, \scn}({\SO}_\dagger)  \longrightarrow \cS_{\vp, \scn}(\widehat{{\SO}_\dagger}) \longrightarrow \cS_{\vp}(\widehat{{\SO}_\dagger}) \longrightarrow 1.
\end{equation*}
Let $\zeta_{2,2},$ $\zeta_{4,0},$ and $\zeta^*_{1,1}$ be characters on $Z((\widehat{{\SO}_\dagger})_{\scn})$ which respectively correspond to $\SO_{2,2},$ $\SO_{4,0},$ and $\SO^*_{1,1}$ via the Kottwitz isomorphism \cite[Theorem 1.2]{kot86}. 

\begin{thm} \label{1-1 for SO4}
Given an $L$-parameter $\vp \in \Phi(\SO_{\dagger}),$ 
we fix a lifting $\tvp \in \Phi(\GSO_{\dagger})$ of $\vp.$ 
Let $\ts$ be the unique member in $\Pi_{\tvp}(\GSO_{\dagger})$ via the LLC for $\GSO_{\dagger}$ in Section \ref{LLC GSO22}.
Then, there is a one-one bijection 
\[
\Pi_{\vp}({\SO}_{\dagger}) \overset{1-1}{\longleftrightarrow} \Pi(\cS_{\vp, \scn}(\widehat{{\SO}_{\dagger}}), \zeta_{\dagger}),
\]
sending $\sigma \mapsto \rho_{\sigma},$ 
such that we have an isomorphism
\[
V_{\ts} ~ ~ \s \bigoplus_{\sigma \in \Pi_{\vp}({\SO}_{\dagger})} \rho_{\sigma} \boxtimes  \si
\]
as representations of $\cS_{\vp, \scn}(\widehat{{\SO}_{\dagger}}) \rtimes {\SO}_{\dagger}(F),$
where the character $\zeta_{\dagger}$ runs through $\zeta_{2,2},$ $\zeta_{4,0},$ $\zeta^*_{1,1},$ according to $\SO_{\dagger}.$ 
\end{thm}
\begin{rem} \label{some +_ SO4} 
Given $\vp \in \Phi(\SO_{2,2}),$ by Theorem \ref{1-1 for SO4}, we have a one-to-one correspondence
\[
\Pi_{\vp}({\SO}_{2,2}) \cup \Pi_{\vp}({\SO}_{4,0}) \cup \Pi_{\vp}({{\SO}^*_{1,1}}^{-+}) \cup \Pi_{\vp}({{\SO}^*_{1,1}}^{+-}) \overset{1-1}{\longleftrightarrow} {\Irr}(\cS_{\vp, \scn}(\widehat{{\SO}_{2,2}})),
\] 
where ${{\SO}^*_{1,1}}^{-+} \s (\SL_1(D) \times \SL_2)/\Delta \mu_2$ and  ${{\SO}^*_{1,1}}^{+-} \s (\SL_2 \times \SL_1(D))/\Delta \mu_2,$ both of which are isomorphic to ${\SO}^*_{1,1}.$
\end{rem}
\subsection{Proof of Theorem \ref{1-1 for SO4}} \label{pf of 1-1 for SO4}
We follow the idea in \cite[Lemma 12.6]{hs11}.
We first deal with the case of $\SO_{\dagger} = \SO^*_{1,1}.$ 
Then, the proofs for the other two cases of ${\SO}_{2,2}$ and ${\SO}_{4,0}$ are the same after replacing $\SL_1(D) \times \SL_2$ by $\SL_2 \times \SL_2$ and $\SL_1(D) \times \SL_1(D),$ respectively.

Let an $L$-parameter $\vp \in \Phi(\SO^*_{1,1})$ be given.
As described in Section \ref{const of L-packet for SO22}, 
there is an $L$-parameter $\tvp \in \Phi(\GSO^*_{1,1})$ such that 
$\std^*_{1,1} \circ \tvp=\vp.$
The description in Section \ref{LLC GSO33} implies that
$\tvp$ is of the form $\tvp_1 \oplus \tvp_2 ,$ 
where 
$\tvp_1 \in \Phi({\GL}_1(D))$ and $\tvp_2 \in \Phi({\GL}_2)$ with $\det \tvp_1 = \det \tvp_2.$

Now, we denote by $\vp_0$ the image in $\PSL_2(\CC) \times \PSL_2(\CC)$ of $\tvp$ via the composite of maps
\[
\widehat{{\GSO}^*_{1,1}} 
\overset{{\std}^*_{1,1}}{\longrightarrow} \widehat{{\SO}^*_{1,1}} = {\SO}_4(\CC) \s ({\SL}_2(\CC)\times {\SL}_2(\CC))/\{(1,1), (-1,-1) \}
\overset{pr^*_{1,1}}{\longrightarrow}
{\PSL}_2(\CC) \times {\PSL}_2(\CC).
\]  
It then follows that $\vp_0 \in \Phi(\SL_1(D) \times \SL_2)$ and $\vp_0 = \vp_1 \oplus \vp_2$ with $\vp_1 \in \Phi({\SL}_1(D))$ and $\vp_2 \in \Phi(\SL_2).$
Note that $\vp_0 = pr \circ \tvp_0,$ where $pr: \GL_2(\CC) \times \GL_2(\CC) \twoheadrightarrow \PSL_2(\CC) \times \PSL_2(\CC)$ is the usual projection map.

Due to Section \ref{const of L-packet for SO22}, we have $\sigma \in \Pi_{\vp}({\SO}^*_{1,1})$ and $\ts \in \Pi_{\tvp}({\GSO}^*_{1,1}).$ 
Note from Section \ref{LLC GSO22}  that 
$\ts$ is of the form $\ttau_1 \oplus \ttau_2$ with $\omega_{\ttau_1}=\omega_{\ttau_1},$ where $\ttau_1 \in \Pi({\GL}_1(D))$ and $\ttau_2 \in \Pi({\GL}_2)$  corresponding to $\tvp_1$ and $\tvp_2$ via the LLC for $\GL_2$ \cite{he00, ht01, scholze13} and for $\GL_1(D)$ \cite{hs11}, respectively.
\begin{lm} \label{subgroup lemma SO4}
With the notation above, $S_{\vp}(\widehat{{\SO}^*_{1,1}})$ is a normal subgroup of finite index in $S_{\vp_0}(\widehat{{\SL}_{1}(D)\times\SL_2}).$
\end{lm}
\begin{proof}
The centralizer $C_{\vp_0}(\widehat{{\SL}_{1}(D)\times\SL_2})$ is equal to the image of the disjoint union
\[
\bigsqcup_{\nu \in {\Hom}(W_F, \{\pm 1\})} 
\{ h \in {\SO}_4(\CC) : h \vp(w)h^{-1} \vp(w)^{-1} = \nu(w) \}
\]
via the map $\std^*_{1,1}.$ Further, we note that  
\[
\{ h \in {\SO}_4(\CC) : h \vp(w)h^{-1} \vp(w)^{-1} = 1 \} = C_{\vp}(\widehat{{\SO}^*_{1,1}}) 
\]
and $S_{\vp}(\widehat{{\SO}^*_{1,1}}) = {\std}^*_{1,1} (C_{\vp}(\widehat{{\SO}^*_{1,1}})).$
It is elementary to check that
\[
g \cdot {\std}^*_{1,1} (C_{\vp}(\widehat{{\SO}^*_{1,1}})) \cdot g^{-1} = {\std}^*_{1,1} (C_{\vp}(\widehat{{\SO}^*_{1,1}}))
\]
for any $g \in S_{\vp_0}(\widehat{{\SL}_{1}(D)\times\SL_2}).$
Since ${\Hom}(W_F, \{\pm 1\})$ is finite,
the lemma is proved.
\end{proof}
\begin{lm}  \label{lemma 3 SO4} 
With the notation above, we have the following commutative diagram
\[
\begin{CD}
1 @>>>  \widehat Z_{\vp, \scn}({\SO}^*_{1,1})  @>>> \cS_{\vp, \scn}(\widehat{{\SO}^*_{1,1}}) @>>> \cS_{\varphi}(\widehat{{\SO}^*_{1,1}})  @>>> 1 \\
@. @| @VV{\cap}V @VV{\cap}V @.\\
1 @>>>  \widehat Z_{\vp_0, \scn}({\SL}_{1}(D)\times\SL_2)  @>>> \cS_{\vp_0, \scn}(\widehat{{\SL}_{1}(D)\times\SL_2}) @>>> \cS_{\varphi_0}(\widehat{{\SL}_{1}(D)\times\SL_2})  @>>> 1.
\end{CD} 
\]
\end{lm}
\begin{proof}
From Lemmas \ref{same connected components} and \ref{subgroup lemma SO4}, it follows that $\cS_{\varphi}(\widehat{{\SO}^*_{1,1}}) $ is a subgroup of $\cS_{\varphi_0}(\widehat{{\SL}_{1}(D)\times\SL_2}).$
Note that $S_{\vp, \scn}(\widehat{{\SO}^*_{1,1}})$ and $S_{\vp_0, \scn}(\widehat{{\SL}_{1}(D)\times\SL_2})$ are respectively both central extensions of $S_{\vp}(\widehat{{\SO}^*_{1,1}})$ and $S_{\vp_0}(\widehat{{\SL}_{1}(D)\times\SL_2})$ by $\mu_2(\CC) \times \mu_2(\CC).$ 
Using the same arguments,  $S_{\vp, \scn}(\widehat{{\SO}^*_{1,1}})^{\circ}$ and $S_{\vp_0, \scn}(\widehat{{\SL}_{1}(D)\times\SL_2})^\circ$ are identical. 
It follows that $\widehat Z_{\vp, \scn}({\SO}^*_{1,1}) = \widehat Z_{\vp_0, \scn}({\SL}_{1}(D)\times\SL_2)$ and $\cS_{\vp, \scn}(\widehat{{\SO}^*_{1,1}}) \subset \cS_{\vp_0, \scn}(\widehat{{\SL}_{1}(D)\times\SL_2}).$ 
Thus, the proof of the lemma is complete.
\end{proof}
\begin{lm} \label{lemma 1 SO4}
With the  notation above, $I(\ts)$ is a subgroup of $I(\ts_0).$
\end{lm}
\begin{proof}
Note that $\Pi_{\tvp}(\GSO_{\dagger})$ is a singleton from Section \ref{LLC GSO22} and the LLC for $\GL_n$ and $\GL_m(D)$ is compatible with twisting by characters \cite{ht01, he00, hs11}. So, we have $I(\tvp) = I(\ts).$
Likewise, we have $I(\tvp_0) = I(\ts_0).$
From the fact that $\tvp$ is of the form $\tvp_0 \oplus \mu,$ it then follows that $I(\ts) \subset I(\tvp_0).$
\end{proof}
\begin{rem} \label{useful iso from gk}
From \cite[Theorem 4.3]{gk82}, one can notice that
$\cS_{\varphi_0}(\widehat{{\SL}_{1}(D)\times\SL_2}) \s I(\tvp_0).$
\end{rem}
\begin{lm}  \label{lem 2 SO4}
With the  notation above, we have
\[
\cS_{\vp} \s I(\ts).
\]
\end{lm}
\begin{proof}
From \eqref{surjective prs22} and Lemma \ref{the case of singleton}, it is enough to show that $\cS_{\vp} \s X(\tvp).$
This is immediate from Lemma \ref{lemma of Chaoli},
since the centralizer $C_{\tvp}(\widehat{{\GSO}_{\dagger}})$ is connected so that $\cS_{\tvp}({\GSO}_{\dagger})$ is trivial.
\end{proof}
\begin{pro}  \label{pro 4 SO4}
There is a homomorphism
$\Lambda_{\dagger} : \cS_{\vp, \scn}(\widehat{{\SO}^*_{1,1}}) \rightarrow \mcA(\ts)$ 
(unique up to 1-dimensional character of $\cS_{\varphi}(\widehat{{\SO}^*_{1,1}})$) with the following commutative diagram
\begin{equation*} 
\begin{CD}
1 @>>>  \widehat Z_{\vp, \scn}({\SO}^*_{1,1})  @>>> \cS_{\vp, \scn}(\widehat{{\SO}^*_{1,1}}) @>>> \cS_{\varphi}(\widehat{{\SO}^*_{1,1}})  @>>> 1 \\
@. @VV{\zeta_{\dagger}}V @VV{\Lambda^*_{1,1}}V @VV{\s}V @.\\
1 @>>> \CC^\times @>>> \mcA(\ts) @>>> I(\ts) @>>> 1.
\end{CD} 
\end{equation*}
\end{pro}
\begin{proof}
Using Hiraga-Saito's homomorphism $\Lambda_{{\SL}_2 \times {\SL}_2}$ in \eqref{a diagram} (cf. Remark \ref{rem general HS}) and the fact that $\cS_{\varphi}(\widehat{{\SO}^*_{1,1}}) $ is a subgroup of $\cS_{\varphi_0}(\widehat{{\SL}_{1}(D)\times\SL_2})$ (see Lemma \ref{lemma 3 SO4}),
we define a map $\Lambda^*_{1,1} :  \cS_{\vp, \scn}(\widehat{{\SO}^*_{1,1}}) \rightarrow \mcA(\ts)$ as the restriction 
\[
\Lambda_{{\SL}_2\times {\SL}_2}|_{\cS_{\vp, \scn}(\widehat{{\SO}^*_{1,1}})}
\]
of $\Lambda_{{\SL}_2\times {\SL}_2}$ to $\cS_{\vp, \scn}(\widehat{{\SO}^*_{1,1}}).$
Due to Lemmas \ref{lemma 3 SO4}, \ref{lemma 1 SO4}, and \ref{lem 2 SO4}, and by the definition of $\Lambda_{\SL_2\times\SL_2}$ in \eqref{a diagram}, $\Lambda^*_{1,1}$ is well-defined.
Since $I(\ts)$ is contained in $I(\Pi)$ and $\cS_{\varphi}(\widehat{{\SO}^*_{1,1}}) $ is a subgroup of $\cS_{\varphi_0}(\widehat{{\SL}_{1}(D)\times\SL_2}),$  it follows that the image $\Lambda^*_{1,1}$ is in $\mcA(\ts).$
Thus, the proof of the proposition is complete. 
\end{proof}

We now finish the Proof of Theorem \ref{1-1 for SO4}. 
Since $\Pi_{\vp}(\SO^*_{1,1})$ is in bijection with $\Irr(\mcA(\ts), \id)$ due to \eqref{useful decomp} and \eqref{bij 333}, and since $\Irr(\mcA(\ts), \id)$ is again in bijection with $\Pi(\cS_{\vp, \scn}(\widehat{{\SO}^*_{1,1}}), \zeta^*_{1,1})$ due to Proposition \ref{pro 4 SO4}, the proof of Theorem \ref{1-1 for SO4} is complete.
\subsection{Properties of $\L$-maps for $\SO_{2,2}$ and its inner forms} \label{properties for so22}
The $\L$-maps defined in Section \ref{const of L-packet for SO22} satisfy the following property.
We continue to use ${\SO}_\dagger$ for $\SO_{2,2},$ $\SO_{4,0},$ and ${\SO}^*_{1,1},$
so that $\L_{\dagger},$ $\sigma_{\dagger},$ and so on will make sense accordingly.
\begin{pro} \label{discreteness for SO22}
A given $\sigma_{\dagger} \in \Pi({\SO}_\dagger)$ is an essentially square-integrable representation if and only if its $L$-parameter $\vp_{\sigma_{\dagger}} := \L_{\dagger}(\sigma_{\dagger})$ does not factor through any proper Levi subgroup of ${\SO}_4(\CC).$
\end{pro}
\begin{proof}
By the definition of $\L_{\dagger}$ in Section \ref{const of L-packet for SO22}, $\sigma_{\dagger}$ is an irreducible constituent of the restriction $\ts_{\dagger}|_{{\SO}_\dagger}$ for some $\ts_{\dagger} \in \Pi({\GSO}_\dagger).$
From Remark \ref{rem in rest from to} and \cite{ht01, he00, scholze13, hs11}, $\sigma_{\dagger}$ is an essentially square-integrable representation if and only if $\ts_{\dagger}$ is if and only if $\tvp_{\sigma_{\dagger}}:= L_{\dagger}(\ts_{\dagger})$ does not factor through any proper Levi subgroup of ${\GSO}_\dagger(\CC)$ if and only if $\vp_{\sigma_{\dagger}}$ does not.
\end{proof}
\begin{rem}
In the same way with the proof of Proposition \ref{discreteness for SO22}, we have that a given $\sigma_{\dagger} \in \Pi({\SO}_\dagger)$ is tempered if and only if the image of its $L$-parameter $\vp_{\sigma_{\dagger}} := \L_{\dagger}(\sigma_{\dagger})$ in ${\SO}_4(\CC)$ is bounded. 
\end{rem}
\section{LLC for $\SO_{3,3}$ and its inner forms} \label{LLC so33 and its inner forms}  
Following the idea in Section \ref{LLC so22 and its inner forms}, we present the LLC for $\GSO_{3,3}$ and its $F$-inner forms and establish the LLC for $\SO_{3,3}$ and its $F$-inner forms.
\subsection{The cases of $\GSO_{3,3}$ and its inner forms} \label{LLC GSO33}
From the isomorphism in Section \ref{othogonal case}
\[
{\GSO}_{3,3}(F) \s
 ({\GL}_4(F) \times F^{\times}) / \{ (z, z^{-2}) : z \in F^{\times}   \},
\]
one can notice that any irreducible admissible representation of ${\GSO}_{3,3}(F)$ has the form $\Pi \boxtimes \mu$ for $\Pi \in \Pi(GL_4)$ and $\mu \in \Pi(GL_1)$ with $\omega_{\Pi}=\mu^2$ (cf. \cite[Section 1]{gtrep}).
This implies that
\[
\Pi({\GSO}_{3,3})=
\{
\Pi \boxtimes \mu : \Pi \in \Pi({\GL}_4), \mu \in \Pi({\GL}_1) ~~ \text{with}~~ \omega_{\Pi}=\mu^2
\}.
\]
Further, due to the form of $L$-group $\widehat{{\GSO}_{3,3}}$ in Section \ref{othogonal case}, we note that
\begin{equation} \label{desc of L-para for GSO33}
\Phi({\GSO}_{3,3})=
\{
\tvp_0 \oplus \mu : \tvp_0 \in \Phi({\GL}_4), \mu \in \Phi({\GL}_1) ~~\text{with}~~ \det \tvp_0=\mu^2
\}.
\end{equation}
Thus, by the LLC for $\GL_n$ \cite{ht01, he00, scholze13}, there is a surjective, one-to-one map 
\[
L_{3,3} : {\Pi}({\GSO}_{3,3}) \longrightarrow \Phi({\GSO}_{3,3}).
\]
For a non quasi-split $F$-inner form  ${\GSO}^*_{3,0}$ of ${\GSO}_{3,3},$ we again recall the isomorphism in Section \ref{othogonal case}
\[
 {\GSO}^*_{3,0}(F) \s
 ({\GL}_1(D_4) \times F^{\times})  / \{ (z, z^{-2}) : z \in F^{\times}  \}.
\]
Similarly, we have
\[
\Pi({\GSO}^*_{3,0}) =
\{
\Pi \boxtimes \mu : \Pi \in \Pi({\GL}_1(D_4)), \mu \in \Pi({\GL}_1) ~~\text{with}~~ \omega_{\Pi}=\mu^2.
\}.
\]
Further, due to the form of $L$-groups in Section \ref{othogonal case}, we note that
\begin{equation*} 
\Phi({\GSO}^*_{3,0}) =
\{
\tvp_0 \oplus \mu : \tvp_0 \in \Phi({\GL}_1(D_4)), \mu \in \Phi({\GL}_1) ~~\text{with}~~ \det \tvp_0=\mu^2
\}.
\end{equation*}
Thus, by the LLC for $\GL_m(D)$ \cite{hs11}, there is a surjective, one-to-one maps 
\[
L^*_{3,0} : {\Pi}({\GSO}^*_{3,0}) \longrightarrow \Phi({\GSO}^*_{3,0}).
\]
Since two maps $L_{3,3}$ and $L^*_{3,0}$ are one-to-one, each fiber gives rise to $L$-packets for $\GSO_{3,3}$ and $\GSO^*_{3,0},$ which are all singletons.
Likewise, we define a surjective, one-to-one map
\[
L_{V_D} : {\Pi}({\GSO}_{V_D}) \longrightarrow \Phi({\GSO}_{V_D})
\]
and construct $L$-packets for ${\GSO}_{V_D}.$
For simplicity of notation, we write ${\GSO}_\flat$ for ${\GSO}_{3,3},$ ${\GSO}^*_{3,0},$ and ${\GSO}_{V_D}.$ 
Further, recalling the notation in Section \ref{conj str of L-packets}, we note that:
\begin{align*}
S_{\tvp}(\widehat{{\GSO}_\flat}) & \subset (\widehat{{\GSO}_\flat})_{\ad} \s {\PSO}_6(\CC)  \s {\PSL}_4(\CC), \\
S_{\tvp, \scn}(\widehat{{\GSO}_\flat}) & \subset (\widehat{{\GSO}_\flat})_{\scn} \s  {\Spin}_6(\CC) \s {\SL}_4(\CC).
\end{align*}
\subsection{Construction of $L$-packets for $\SO_{3,3}$ and its inner forms} \label{const of L-packet for SO33}
Given $\sigma \in \Pi({\SO}_{3,3}),$ there is a lifting $\ts \in \Pi({\GSO}_{3,3})$ such that 
\[
\sigma \hookrightarrow {\Res}_{{\SO}_{3,3}}^{{\GSO}_{3,3}}(\ts).
\]
We define a map 
\begin{equation*} 
{\L}_{3,3} : {\Pi}({\SO}_{3,3}) \longrightarrow \Phi({\SO}_{3,3})
\end{equation*}
by $\L_{3,3}(\sigma) := \std_{3,3}(L(\ts)).$ 
Note that $\L_{3,3}$ is not depending on the choice of the lifting $\ts,$ since another lifting must be of the form $\ts \otimes \chi$ for some quasi-character $\chi$ of $F^{\times}$ by Proposition \ref{pro for lifting} and $L_{3,3}(\ts \otimes \chi) = L_{3,3}(\ts) \otimes \chi$ for any quasi-character $\chi$ of $F^{\times}$  \cite{ht01, he00}.
Thus, the map $\L_{3,3}$ is well-defined. 

Furthermore, $\L_{3,3}$ is a surjective, 
since any $\vp \in \Phi({\SO}_{3,3})$ can be lifted to some $\tvp \in \Phi({\GSO}_{3,3})$ by \eqref{surjective prs30} and Theorem \ref{thm by Labesse}. 
For each $\vp \in \Phi({\SO}_{3,3}),$ the fiber is given by
\begin{equation*} 
\Pi_{\vp}({\SO}_{3,3}) = \Pi_{\ts}({\SO}_{3,3}),
\end{equation*}
where $\ts$ is the unique member in $\Pi_{\tvp}({\GSO}_{3,3})$ and $\tvp$ lies in $\Phi({\GSO}_{3,3})$ such that $\std_{3,3} \circ \tvp=\vp.$ 
Due to \cite{ht01, he00} and Proposition \ref{pro for lifting}, the fiber does not depend on the choice of $\tvp.$
This forms an $L$-packet for ${\SO}_{3,3}.$

Similarly, given $\sigma^*_{3,0} \in \Pi({\SO}^*_{3,0}),$ there is a lifting $\ts^*_{3,0} \in \Pi({\GSO}^*_{3,0})$ such that 
\[
\sigma^*_{3,0} \hookrightarrow {\Res}_{{\SO}^*_{3,0}}^{{\GSO}^*_{3,0}}(\ts^*_{3,0}).
\]
We define a map  
\begin{equation*} 
{\L}^*_{3,0} : {\Pi}({\SO}^*_{3,0}) \longrightarrow \Phi({\SO}^*_{3,0})
\end{equation*}
by $\L^*_{3,0}(\sigma^*_{3,0}) := \std^*_{3,0}(L(\ts^*_{3,0})).$
In the same way with $\L_{3,3},$ $\L^*_{3,0}$ turns out to be a well-defined, surjective, and finite-to-one map.
For each $\vp \in \Phi({\SO}^*_{3,0}),$ the $L$-packet is given by
\begin{equation*} 
\Pi_{\vp}({\SO}^*_{3,0}) = \Pi_{\ts^*_{3,0}}({\SO}^*_{3,0}).
\end{equation*}
Due to \cite{hs11} and Proposition \ref{pro for lifting}, each $L$-packet does not depend on the choice of $\tvp.$

Likewise, we define a surjective, finite-to-one map  
\begin{equation*} 
{\L}_{V_D} : {\Pi}({\SO}(V_D) \longrightarrow \Phi({\SO}(V_D))
\end{equation*}
by $\L_{V_D}(\sigma_{V_D}) := \std_{V_D}(L(\ts_{V_D}))$  
For each $\vp \in \Phi({\SO}(V_D)),$ the $L$-packet is given by
\begin{equation*} 
\Pi_{\vp}({\SO}(V_D)) = \Pi_{\ts_{V_D}}({\SO}(V_D)).
\end{equation*}
As in Section \ref{const of L-packet for SO22}, each $L$-packet does not depend on the choice of $\tvp.$
\subsection{Internal structure of $L$-packets for $\SO_{3,3}$ and its inner forms} \label{para for SO33}
We continue with the notation in Section \ref{conj str of L-packets}. 
For simplicity of notation, we shall write ${\SO}_\flat$ for 
${\SO}_{3,3},$ ${\SO}^*_{3,0},$ and $\SO(V_D).$ 
Recall from Section \ref{L-groups} that
\[
\widehat{{\SO}_\flat} \s {\SO}_6(\CC) \s {\SL}_4(\CC)/\mu_2(\CC). 
\]
Note that 
\[
(\widehat{{\SO}_\flat})_{\ad} ={\PSO}_6(\CC), ~~~ 
(\widehat{{\SO}_\flat})_{\scn} ={\Spin}_6(\CC), ~~~ 
Z((\widehat{{\SO}_\flat})_{\scn}) = Z((\widehat{{\SO}_\flat})_{\scn})^{\Gamma} \s \mu_4(\CC).
\]
Let $\vp \in \Phi({\SO}_\flat)$ be given.
We fix a lifting $\tvp \in \Phi({\GSO}_\flat)$ via the surjective map $\widehat{{\GSO}_\flat}  \twoheadrightarrow \widehat{{\SO}_\flat}$ (see Theorem \ref{thm by Labesse}). 
We note that:
\begin{align*}
S_{\vp}(\widehat{{\SO}_\flat}) &\subset {\PSO}_6(\CC) \s {\PSL}_4(\CC), \\
S_{\vp, \scn}(\widehat{{\SO}_\flat}) &\subset {\Spin}_6(\CC) \s {\SL}_4(\CC).
\end{align*}
One can then have a central extension 
\begin{equation} 
1 \longrightarrow \widehat Z_{\vp, \scn}({\SO}_\flat)  \longrightarrow \cS_{\vp, \scn}(\widehat{{\SO}_\flat}) \longrightarrow \cS_{\vp}(\widehat{{\SO}_\flat}) \longrightarrow 1.
\end{equation}
Let $\zeta_{3,3},$ $\zeta^*_{3,0},$ and $\zeta_{V_D}$ be characters on $Z((\widehat{{\SO}_\flat})_{\scn})$ which correspond to $\SO_{3,3},$ $\SO^*_{3,0},$ and $\SO(V_D)$ via the Kottwitz isomorphism \cite[Theorem 1.2]{kot86}. 
Note that the inverse $(\zeta^*_{3,0})^{-1}$ corresponds to another form $\SL_1(D_4^{op})/\mu_2,$ which is isomorphic to $\SO^*_{3,0},$ via the canonical isomorphism between two multiplicative groups $D_4^{\times}$ and $(D_4^{op})^{\times}$ (see Section \ref{othogonal case}). 

\begin{thm} \label{1-1 for SO6}
Given an $L$-parameter $\vp \in \Phi(\SO_{\flat}),$ we fix a lifting $\tvp \in \Phi(\GSO_{\flat})$ of $\vp.$ 
Let $\ts$ be the unique member in $\Pi_{\tvp}(\GSO_{\flat})$ via the LLC for $\GSO_{\flat}$ in Section \ref{LLC GSO33}.
Then, there is a one-one bijection 
\[
\Pi_{\vp}({\SO}_{\flat}) \overset{1-1}{\longleftrightarrow} \Pi(\cS_{\vp, \scn}(\widehat{{\SO}_{\flat}}), \zeta_{\flat}),
\]
sending $\sigma \mapsto \rho_{\sigma},$ 
such that we have  an isomorphism
\[
V_{\ts} ~ ~ \s \bigoplus_{\sigma \in \Pi_{\vp}({\SO}_{\flat})} \rho_{\sigma} \boxtimes  \si
\]
as representations of $\cS_{\vp, \scn}(\widehat{{\SO}_{\flat}}) \rtimes {\SO}_{\flat}(F),$
where the character $\zeta_{\flat}$ runs through $\zeta_{3,3},$ $\zeta_{V_D},$ $\zeta^*_{3,0},$ according to ${\SO}_{\flat}.$
\end{thm}
\begin{rem}
Given $\vp \in \Phi(\SO_{3,3}),$ by Theorem \ref{1-1 for SO6}, we have a one-to-one correspondence
\[
\Pi_{\vp}({\SO}_{3,3}) \cup \Pi_{\vp}({\SO}(V_D)) \cup \Pi_{\vp}^{1/4}({\SO}^*_{3,0}) \cup \Pi_{\vp}^{3/4}({\SO}^*_{3,0}) \overset{1-1}{\longleftrightarrow} {\Irr}(\cS_{\vp, \scn}(\widehat{{\SO}_{3,3}})).
\]
Just only for here, we distinguish $\Pi_{\vp}^{1/4}({\SO^*_{3,0}})$ and $\Pi_{\vp}^{3/4}({\SO^*_{3,0}})$ in the sense that 
$\Pi_{\vp}^{1/4}({\SO^*_{3,0}})$ denotes the $L$-packet for ${\SO^*_{3,0}}$ with $D_4$ and $\Pi_{\vp}^{3/4}({\SO^*_{3,0}})$ denotes the $L$-packet for ${\SO^*_{3,0}}$ with $D_4^{op}$
(see Section \ref{othogonal case}).
\end{rem}
\subsection{Proof of Theorem \ref{1-1 for SO6}} \label{pf of 1-1 for SO6}
We follow the idea in \cite[Lemma 12.6]{hs11}.
We first deal with the case of $\SO_{\flat} = \SO^*_{3,0}.$ 
Then, the proofs for the other two cases of ${\SO}_{3,3}$ and ${\SO}(V_D)$ are the same after replacing $\SL_1(D_4)$ by $\SL_4$ and $\SL_2(D),$ respectively. 

Let an $L$-parameter $\vp \in \Phi(\SO^*_{3,0})$ be given.
As described in Section \ref{const of L-packet for SO33}, 
there is an $L$-parameter $\tvp \in \Phi(\GSO^*_{3,0})$ such that 
$\std^*_{3,0} \circ \tvp=\vp.$
The description \eqref{desc of L-para for GSO33} implies that
$\tvp$ is of the form $\tvp_0 \oplus \mu,$ 
where 
$\tvp_0 \in \Phi({\GL}_1(D_4))$ and $\mu \in \Phi({\GL}_1).$

Now, we denote by $\vp_0$ the image in $\PSL_4(\CC)$ of $\tvp$ via the composite of maps
\[
\widehat{{\GSO}^*_{3,0}} 
\overset{\std^*_{3,0}}{\longrightarrow} \widehat{{\SO}^*_{3,0}} = {\SO}_6(\CC) \s SL_4(\CC)/\mu_2(\CC)
\overset{pr^*_{3,0}}{\longrightarrow}
{\PSL}_4(\CC).
\]  
It then follows that $\vp_0 \in \Phi(\SL_1(D_4)).$
Note that $\vp_0 = pr \circ \tvp_0,$ where $pr: \GL_4(\CC) \twoheadrightarrow \PSL_4(\CC)$ is the usual projection map.

Due to Section \ref{const of L-packet for SO33}, we have $\sigma \in \Pi_{\vp}({\SO}^*_{3,0})$ and $\ts \in \Pi_{\tvp}({\GSO}^*_{3,0}).$ 
Note from Section \ref{LLC GSO33}  that 
$\ts$ is of the form $\Pi \boxtimes \mu$ with $\omega_{\Pi}=\mu^2,$ where $\Pi \in \Pi({\GL}_1(D_4))$ corresponding to $\tvp_0$ via the LLC for $GL_1(D_4)$ \cite{hs11} and $\mu \in \Pi({\GL}_1).$  
\begin{lm} \label{subgroup lemma SO6}
With the notation above, $S_{\vp}(\widehat{{\SO}^*_{3,0}})$ is a normal subgroup of finite index in $S_{\vp_0}(\widehat{{\SL}_{1}(D_4)}).$
\end{lm}
\begin{proof}
The centralizer $C_{\vp_0}(\widehat{{\SL}_{1}(D_4)})$ is equal to the image of the disjoint union
\[
\bigsqcup_{\nu \in {\Hom}(W_F, \{\pm 1\})} 
\{ h \in {\SO}_6(\CC) : h \vp(w)h^{-1} \vp(w)^{-1} = \nu(w) \}
\]
via the map $\std^*_{3,0}.$ Further, we note that:  
\[
\{ h \in {\SO}_6(\CC) : h \vp(w)h^{-1} \vp(w)^{-1} = 1 \} = C_{\vp}(\widehat{{\SO}^*_{3,0}}), 
\]
\[
S_{\vp}(\widehat{{\SO}^*_{3,0}}) = {\std}^*_{3,0} (C_{\vp}(\widehat{{\SO}^*_{3,0}})).
\]
It is elementary to check that
\[
g \cdot {\std}^*_{3,0} (C_{\vp}(\widehat{{\SO}^*_{3,0}})) \cdot g^{-1} = {\std}^*_{3,0} (C_{\vp}(\widehat{{\SO}^*_{3,0}}))
\]
for any $g \in S_{\vp_0}(\widehat{{\SL}_{1}(D_4)}).$
Since ${\Hom}(W_F, \{\pm 1\})$ is finite,
the lemma is proved.
\end{proof}
\begin{lm}  \label{lemma 3 SO6} 
With the notation above, we have the following commutative diagram
\[
\begin{CD}
1 @>>>  \widehat Z_{\vp, \scn}({\SO}^*_{3,0})  @>>> \cS_{\vp, \scn}(\widehat{{\SO}^*_{3,0}}) @>>> \cS_{\varphi}(\widehat{{\SO}^*_{3,0}})  @>>> 1 \\
@. @| @VV{\cap}V @VV{\cap}V @.\\
1 @>>>  \widehat Z_{\vp_0, \scn}({\SL}_{1}(D_4))  @>>> \cS_{\vp_0, \scn}(\widehat{{\SL}_{1}(D_4)}) @>>> \cS_{\varphi_0}(\widehat{{\SL}_{1}(D_4)})  @>>> 1.
\end{CD} 
\]
\end{lm}
\begin{proof}
From Lemmas \ref{same connected components} and \ref{subgroup lemma SO6}, it follows that $\cS_{\varphi}(\widehat{{\SO}^*_{3,0}}) $ is a subgroup of $\cS_{\varphi_0}(\widehat{{\SL}_{1}(D_4)}).$
Note that $S_{\vp, \scn}(\widehat{{\SO}^*_{3,0}})$ and $S_{\vp_0, \scn}(\widehat{{\SL}_{1}(D_4)})$ are respectively both central extensions of $S_{\vp}(\widehat{{\SO}^*_{3,0}})$ and $S_{\vp_0}(\widehat{{\SL}_{1}(D_4)})$ by $\mu_4(\CC).$ 
Using the same arguments, we have $S_{\vp, \scn}(\widehat{{\SO}^*_{3,0}})^{\circ}$ and $S_{\vp_0, \scn}(\widehat{{\SL}_{1}(D_4)})^\circ$ are identical. 
It follows that $\widehat Z_{\vp, \scn}({\SO}^*_{3,0}) = \widehat Z_{\vp_0, \scn}({\SL}_{1}(D_4))$ and $\cS_{\vp, \scn}(\widehat{{\SO}^*_{3,0}}) \subset \cS_{\vp_0, \scn}(\widehat{{\SL}_{1}(D_4)}).$ 
Thus, the proof is complete.
\end{proof}
\begin{lm} \label{lemma 1 SO6}
With the notation above, $I(\ts)$ is a subgroup of $I(\ts_0).$
\end{lm}
\begin{proof}
Note that $\Pi_{\tvp}(\GSO_{\dagger})$ is a singleton from Section \ref{LLC GSO22} and the LLC for $\GL_n$ and $\GL_m(D)$ is compatible with twisting by characters. So, we have $I(\tvp) = I(\ts).$
Likewise, we have $I(\tvp_0) = I(\ts_0).$
From the fact that $\tvp$ is of the form $\tvp_0 \oplus \mu,$ it then follows that $I(\ts) \subset I(\tvp_0).$
\end{proof}
\begin{rem} \label{useful iso from gk}
From \cite[Theorem 4.3]{gk82}, we have
$\cS_{\varphi_0}(\widehat{{\SL}_{1}(D)\times\SL_2}) {\s} I(\tvp_0).$
\end{rem}
\begin{lm}  \label{lem 2 SO6}
With the  notation above, we have
\[
\cS_{\vp} \s I(\ts).
\]
\end{lm}
\begin{proof}
From \eqref{surjective prs30} and Lemma \ref{the case of singleton}, it is enough to show that $\cS_{\vp} \s X(\tvp).$
This is immediate from Lemma \ref{lemma of Chaoli},
since the centralizer $C_{\tvp}(\widehat{{\GSO}_{\flat}})$ is connected so that $\cS_{\tvp}({\GSO}_{\flat})$ is trivial.
\end{proof}

\begin{pro}  \label{pro 4 SO6}
There is a homomorphism
$\Lambda_{\flat} : \cS_{\vp, \scn}(\widehat{{\SO}^*_{3,0}}) \rightarrow \mcA(\ts)$ 
(unique up to 1-dimensional character of $\cS_{\varphi}(\widehat{{\SO}^*_{3,0}})$) with the following commutative diagram
\begin{equation*} 
\begin{CD}
1 @>>>  \widehat Z_{\vp, \scn}({\SO}^*_{3,0})  @>>> \cS_{\vp, \scn}(\widehat{{\SO}^*_{3,0}}) @>>> \cS_{\varphi}(\widehat{{\SO}^*_{3,0}})  @>>> 1 \\
@. @VV{\zeta_{\flat}}V @VV{\Lambda^*_{3,0}}V @VV{\s}V @.\\
1 @>>> \CC^\times @>>> \mcA(\ts) @>>> I(\ts) @>>> 1.
\end{CD} 
\end{equation*}
\end{pro}
\begin{proof}
Using Hiraga and Saito's homomorphism $\Lambda_{{\SL}_4}$ in \eqref{a diagram} and the fact that $\cS_{\varphi}(\widehat{{\SO}^*_{3,0}}) $ is a subgroup of $\cS_{\varphi_0}(\widehat{{\SL}_{1}(D_4)})$ (see Lemma \ref{lemma 3 SO6}), we define a map $\Lambda^*_{3,0} :  \cS_{\vp, \scn}(\widehat{{\SO}^*_{3,0}}) \rightarrow \mcA(\ts)$ as the restriction 
\[
\Lambda_{{\SL}_4}|_{\cS_{\vp, \scn}(\widehat{{\SO}^*_{3,0}})}
\]
of $\Lambda_{{\SL}_4}$ to $\cS_{\vp, \scn}(\widehat{{\SO}^*_{3,0}}).$
Due to Lemmas \ref{lemma 3 SO6}, \ref{lemma 1 SO6}, and \ref{lem 2 SO6}, and by the definition of $\Lambda_{\SL_4}$ in \eqref{a diagram}, $\Lambda^*_{3,0}$ is well-defined.
Since $I(\ts)$ is contained in $I(\Pi)$ and $\cS_{\varphi}(\widehat{{\SO}^*_{3,0}}) $ is a subgroup of $\cS_{\varphi_0}(\widehat{{\SL}_{1}(D_4)}),$  it follows that the image $\Lambda^*_{3,0}$ is in $\mcA(\ts).$
Thus, the proof is complete. 
\end{proof}

We now finish the Proof of Theorem \ref{1-1 for SO6}. 
Since $\Pi_{\vp}(\SO^*_{3,0})$ is in bijection with $\Irr(\mcA(\ts), \id)$ due to \eqref{useful decomp} and \eqref{bij 333}, and since $\Irr(\mcA(\ts), \id)$ is again in bijection with $\Pi(\cS_{\vp, \scn}(\widehat{{\SO}^*_{3,0}}), \zeta^*_{3,0})$ due to Proposition \ref{pro 4 SO6}, the proof of Theorem \ref{1-1 for SO6} is complete.
\subsection{Properties of $\L$-maps for $\SO_{3,3}$ and its inner forms} \label{properties for so33}
The $\L$-maps defined in Section \ref{const of L-packet for SO33} satisfy the following property.
We continue to use ${\SO}_\flat$ for ${\SO}_{3,3},$ $\SO(V_D),$ and ${\SO}^*_{3,0},$
so that $\L_{\flat},$ $\sigma_{\flat},$ and so on will be used accordingly.
\begin{pro} \label{discreteness for SO33}
A given $\sigma_{\flat} \in \Pi({\SO}_\flat)$ is an essentially square-integrable representation if and only if its $L$-parameter $\vp_{\sigma_{\flat}} := \L_{\flat}(\sigma_{\flat})$ does not factor through any proper Levi subgroup of ${\SO}_6(\CC).$
\end{pro}
\begin{proof}
By the definition of $\L_{\flat}$ in Section \ref{const of L-packet for SO33}, $\sigma_{\flat}$ is an irreducible constituent of the restriction $\ts_{\flat}|_{{\SO}_\flat}$ for some $\ts_{\flat} \in \Pi({\GSO}_\flat).$
From Remark \ref{rem in rest from to} and \cite{ht01, he00, scholze13, hs11}, $\sigma_{\flat}$ is an essentially square-integrable representation if and only if $\ts_{\flat}$ is if and only if $\tvp_{\sigma_{\flat}}:= L_{\flat}(\ts_{\flat})$ does not factor through any proper Levi subgroup of ${\GSO}_\flat(\CC)$ if and only if $\vp_{\sigma_{\flat}}$ does not.
\end{proof}
\begin{rem}
In the same way with the proof of Proposition \ref{discreteness for SO33}, we have that a given $\sigma_{\flat} \in \Pi({\SO}_\flat)$ is tempered if and only if the image of its $L$-parameter $\vp_{\sigma_{\flat}} := \L_{\flat}(\sigma_{\flat})$ in ${\SO}_6(\CC)$ is bounded. 
\end{rem}
\section{LLC for $\Sp_{1,1}$} \label{LLC for Sp_{1,1}}
In this section, we state and prove the local Langlands conjecture for $\Sp_{1,1}.$  
Furthermore, we classify all cases of the central extension \eqref{central ext} for $\Sp_{1,1},$ describe all sizes of $L$-packets of $\Sp_{1,1},$ illustrate multiplicities in restriction from $\GSp_{1,1},$ and give an explicit example in which an interesting phenomenon appears.
\subsection{Revisiting the LLC for $\GSp_{1,1}$} \label{LLC GSP11}
We recall the LLC for $\GSp_{1,1},$ which was established by Gan and Tantono in \cite{gtan12}, and utilize it to construct the LLC for $\Sp_{1,1}$ in Section \ref{const of L-packet for Sp11}.
Consider ${\GSO^*_{3,0}}$ and ${\GSO^*_{1,1}}$ which participate in $L$-packets for $\GSp_{1,1}.$ 
The relations between dual groups in Section \ref{L-groups} can be combined with \cite[Section 7]{gtan12} to have the following inclusions $\iota^*_{3,0}, \iota^*_{1,1}$ on $L$-parameters:
\begin{equation*} 
\iota^*_{3,0}: 
\{ \text{irreducible 4-dimensional} ~ \tvp \in \Phi({\GSp}_{1,1}) \}
~~ \hookrightarrow ~~
 \Phi({\GL}_1(D_4)) \times \Phi({\GL}_1)
\end{equation*}
defined by $ \iota^*_{3,0}(\tvp) = (\tvp, \simi \tvp),$ and
\begin{equation*} 
\iota^*_{1,1}: 
\{ (\tvp_1, \tvp_2) \in \Phi({\GSO}^*_{1,1}) : \tvp_1 \neq \tvp_2, ~~ \det \tvp_1 = \det \tvp_2 \} / {\Out}({\SO}_4) 
~~ \hookrightarrow ~~
 \Phi({\GSp}_{1,1})  
\end{equation*} 
defined by $ \iota^*_{1,1}(\tvp_1, \tvp_2)=\tvp_1\oplus\tvp_2 = \tvp,$
where 
the action of ${\Out}({\SO}_4)$ on $\Phi({\GSO}^*_{1,1})$ is given by $(\tvp_1, \tvp_2) \mapsto (\tvp_2, \tvp_1).$
\begin{rem} \label{L-parameters of GSp11}
We note from \cite[Section 7]{gtan12} that $\tvp \in \Phi({\GSp}_{1,1})$ is either an irreducible 4-dimensional representation or the image of $\iota^*_{1,1}.$
Moreover, since $\tvp_1 \in \Phi(\GL_1(D))$ and $\tvp_2 \in \Phi(\GL_2),$ the action of of ${\Out}({\SO}_4)$ is non-trivial if and only if both $\tvp_1$ and $\tvp_2$ are elliptic $L$-parameters of ${\GL}_2.$
\end{rem}
The LLC for $\GSp_{1,1}$ states that there is a surjective, two-to-one map 
\[
L_{1,1} : {\Pi}({\GSp}_{1,1}) \longrightarrow \Phi({\GSp}_{1,1}),
\]
satisfying several natural conditions which determine the map uniquely (see \cite[p.763]{gtan12} for details).

\subsection{Construction of $L$-packets for $\Sp_{1,1}$} \label{const of L-packet for Sp11}
We define a map 
\begin{equation*} 
{\L}_{1,1} : {\Pi}({\Sp}_{1,1}) \longrightarrow \Phi({\Sp}_{1,1})
\end{equation*}
by ${\L}_{1,1}(\sigma) = \std_{1,1}(L_{1,1}(\ts))$ with $\ts \in \Pi({\GSp}_{1,1})$ such that 
\[
\sigma \hookrightarrow {\Res}_{{\Sp}_{1,1}}^{{\GSp}_{1,1}}(\ts).
\]
This is an analogue of the local Langlands correspondence for ${\Sp}_4$ which was established by Gan and Takeda in \cite{gtsp10}. 
Note that $L_{1,1}(\ts \otimes \chi) = L_{1,1}(\ts) \otimes \chi$ for any quasi-character $\chi$ of $F^{\times}$ \cite[(iv) p.2]{gtan12} and ${\L}_{1,1}$ is not depending on the choice of the lifting $\ts$ by Proposition \ref{pro for lifting}.
Thus, the map ${\L}_{1,1}$ is well-defined. Furthermore, since any $\vp \in \Phi({\Sp}_{1,1})$ can be lifted to some $\tvp \in \Phi({\GSp}_{1,1})$ \cite[Proposition 2.8]{gtsp10}, 
${\L}_{1,1}$ is a surjective. For each $\vp \in \Phi({\Sp}_{1,1}),$ the fiber is given by
\begin{equation} \label{def of L-packet}
\Pi_{\vp}({\Sp}_{1,1}) = \bigcup_{\ts \in \Pi_{\tvp}({\GSp}_{1,1})} \Pi_{\ts}({\Sp}_{1,1}),
\end{equation}
where $\tvp$ lies in $\Phi({\GSp}_{1,1})$ such that $\std_{1,1} \circ \tvp=\vp$ (see Theorem \ref{thm by Labesse}). 
Due to \cite[(iv) p.2]{gtan12} and Proposition \ref{pro for lifting}, the fiber does not depend on the choice of $\tvp.$
This forms an $L$-packet for ${\Sp}_{1,1}.$
\begin{rem}
Unlike the case of $\Sp_4,$ it is possible that the union in \eqref{def of L-packet} is not disjoint. 
This occurs only when $\tvp \in \Phi({\GSp}_{1,1})$ is of the form $\iota^*_{1,1}(\tvp_1,  \tvp_2)=\tvp_1 \oplus \tvp_2$  for some $(\tvp_1,  \tvp_2) \in \Phi({\GSO^*_{1,1}})$ such that $\tvp_1 \s \tvp_2 \chi$ for some quadratic character $\chi$ of $F^{\times}$ (see \cite[Proposition 6.8.(iii)(b)]{gtsp10}). 
Later, we will analyze this case in Section \ref{pf of 1-1 for Sp11} and give its explicit example in Section \ref{sec of example}).
\end{rem}
\subsection{Internal structure of $L$-packets for $\Sp_{1,1}$} \label{para}
We parameterize each $L$-packet $\Pi_{\vp}({\Sp}_{1,1})$ for ${\Sp}_{1,1}$ in terms of so-called $S$-groups, as described in Section \ref{conj str of L-packets}.

We narrow down notation in Sections \ref{conj str of L-packets} and \ref{whole sec of rest} to the case of $\Sp_{1,1}.$
Recall from Section \ref{L-groups} that
\[
\widehat{{\Sp}_{1,1}} = \widehat{{\Sp}_{4}} = {\PSp}_4(\CC) \s {\SO}_5(\CC). 
\]
Note that 
\[
(\widehat{{\Sp}_{1,1}})_{\ad} ={\PSp}_4(\CC), ~~~ 
(\widehat{{\Sp}_{1,1}})_{\scn} ={\Sp}_4(\CC), ~~~ 
Z((\widehat{{\Sp}_{1,1}})_{\scn}) = Z((\widehat{{\Sp}_{1,1}})_{\scn})^{\Gamma} \s \mu_2(\CC).
\]
Let $\vp \in \Phi({\Sp}_{1,1})$ be given.
We fix a lifting $\tvp \in \Phi({\GSp}_{1,1})$ via the surjective map $\widehat{{\GSp}_{1,1}}  \longrightarrow \widehat{{\Sp}_{1,1}}$ (see Theorem \ref{thm by Labesse}). 
With the notation in Section \ref{conj str of L-packets}, we have:
\begin{align*}
S_{\vp}(\widehat{{\Sp}_{4}}) &=S_{\vp}(\widehat{{\Sp}_{1,1}}) \subset {\PSO}_5(\CC),\\
S_{\tvp}(\widehat{{\GSp}_{4}})&=S_{\tvp}(\widehat{{\GSp}_{1,1}}) \subset {\PSO}_5(\CC),\\
S_{\vp, \scn}(\widehat{{\Sp}_{4}})& = S_{\vp, \scn}(\widehat{{\Sp}_{1,1}}) \subset {\Sp}_4(\CC),\\
S_{\tvp, \scn}(\widehat{{\GSp}_{4}}) &= S_{\tvp, \scn}(\widehat{{\GSp}_{1,1}}) \subset {\Sp}_4(\CC).
\end{align*}
We then have a central extension 
\begin{equation*} %
1 \longrightarrow \widehat Z_{\vp, \scn}({\Sp}_{1,1})  \longrightarrow \cS_{\vp, \scn}(\widehat{{\Sp}_{1,1}}) \longrightarrow \cS_{\vp}(\widehat{{\Sp}_{1,1}}) \longrightarrow 1.
\end{equation*}
We denote by $\mathbbm{1}$ the trivial character and $\sgn$  the non-trivial characters on $\ZZ/2\ZZ \s \mu_2(\CC).$
Considering the isomorphism $ Z(({\Sp}_{1,1})_{\scn}) \s \mu_2(\CC),$ $\mathbbm{1}$ maps to $\Sp_4$ and $\sgn$ to $\Sp_{1,1},$ via the Kottwitz isomorphism \cite[Theorem 1.2]{kot86}.

To state Theorem \ref{1-1 for Sp11} below, we need to recall three mutually exclusive possibilities of $\tvp \in \Phi(\GSp_{1,1})$ from \cite[Section 7]{gtan12} as follows.
\begin{itemize}
\item \textbf{Case I}: $\tvp$ is of the form $\tvp_1 \oplus \tvp_2,$ where $\tvp_i \in \Phi_{\el}({\GL}_2),$ $\tvp_1 \not\s \tvp_2,$ and $\det \tvp_1 = \det \tvp_2.$ 
Since $\Phi_{\el}({\GL}_2)=\Phi_{\el}({\GL}_1(D)),$ we thus note that $\tvp \in \Phi({\GSO}^*_{1,1}).$ 
Based on the classification in \cite[Proposition 6.8(iii)]{gtsp10}, we further subcategorize this case as follows:
\begin{itemize}
\item \textbf{(a)} $\tvp_1 \not\s \tvp_2 \otimes \chi$ for any character $\chi$ on $F^{\times},$ 
\item \textbf{(b)} $\tvp_1 \s \tvp_2 \otimes \chi$ with $\chi$ necessarily quadratic.
\end{itemize}
\item \textbf{Case II}: $\tvp$ is of the form $\chi (\tvp_0 \oplus (\omega_0 \oplus \mathbbm{1})),$ where
$\chi$ is a quasi-character on $F^{\times},$ $\tvp_0$ lies in $\Phi_{\el}({\GL}_1(D)),$ and $\omega_0$ denotes the central character of the essentially square-integrable representation corresponding to $\tvp_0$ via the local Langlands correspondence for $\GL_1(D)$ \cite[Chapter 11]{hs11}.
We note that $\tvp \in \Phi({\GSO}^*_{1,1}).$

\item \textbf{Case III}: $\tvp$ sits in $\Phi_{\el}({\GL}_1(D_4)),$ which in fact coincides with $\Phi({\GL}_1(D_4)).$
\end{itemize}
Next, we recall the $L$-packets $\Pi_{\tvp}(\GSp_{1,1})$ for each case, which were established in \cite{gtan12}.

\textbf{Case I:} $\Pi_{\tvp}(\GSp_{1,1})=\{ 
\ts_1=:\theta(JL(\tau_1) \boxtimes \tau_2),  \ts_2:=\theta(JL(\tau_2) \boxtimes \tau_2)
\},$ where $\theta$ stands for theta correspondence from $\GSO^*_{1,1}$ to $\GSp_{1,1},$ $JL$ denotes the local Jacquet-Langlands lift from $\GL_2(F)$ to $\GL_1(D),$ 
and $\tau_i \in \Pi_{\ess, \disc}(\GL_2)$ is corresponding to $\tvp_i$ via the local Langlands correspondence for $\GL_2$ \cite{ht01, he00}. 
Note that $\Pi_{\tvp}(\GSp_{1,1})$ consists of essentially square-integrable representations.

\textbf{Case II:} $\Pi_{\tvp}(\GSp_{1,1})=\{ \ts:=J_P(\rho, \chi) \},$ 
where $J_P(\rho, \chi)$ denotes the Langlands quotient of the standard module, 
$P \s (\GL_1(D)\times \GL_1) \cdot N$ (see \cite[Section 5.3]{gtan12}) is an $F$-parabolic subgroup (which is the Siegel parabolic subgroup) of $\GSp_{1,1},$ 
and $\rho$ is the essentially square-integrable representation corresponding to $\tvp_0$ via the local Langlands correspondence for $\GL_1(D).$
Note that $J_P(\rho, \chi)$ is not essentially square-integrable.

\textbf{Case III:} $\Pi_{\tvp}(\GSp_{1,1})=\{ \ts:=\pi \},$ where $\pi$ is the essentially square-integrable representation of $\GSp_{1,1}(F)$ 
whose theta lift $\theta(\pi)$ to $\GSO^*_{3,0}$ is $\Pi \boxtimes \mu \in \Pi(\GSO^*_{3,0}).$
Note that $\omega_{\Pi}=\mu^2$ and $\mu= \simi(\tvp).$ 
\begin{rem} \label{rem for L-packets GSp11}
The $L$-packets of Case I and Case III exhaust the set $\Pi_{\el, \disc}(\GSp_{1,1}),$ and the $L$-packets of Case II exhaust the set $\Pi(\GSp_{1,1}) \smallsetminus \Pi_{\el, \disc}(\GSp_{1,1}).$ 
\end{rem}
\begin{thm} \label{1-1 for Sp11}
With the notation above, given an $L$-parameter $\vp \in \Phi(\Sp_{1,1}),$ we fix its lifting $\tvp \in \Phi(\GSp_{1,1}).$
Then, there is a one-one bijection 
\begin{equation} \label{the 1-1}
\Pi_{\vp}({\Sp}_{1,1}) \overset{1-1}{\longleftrightarrow} \Irr(\cS_{\vp, \scn}(\widehat{{\Sp}_{1,1}}), {\sgn}),
\end{equation}
sending $\sigma \mapsto \rho_{\sigma},$ 
such that we have isomorphisms:
\begin{align*} \label{decompositions for sp11}
V_{\ts_i} ~ ~  & \s \bigoplus_{\sigma \in \Pi_{\ts_i}({\Sp}_{1,1})} \rho_{\sigma} \boxtimes  \si ~ (i=1, 2), ~~ \mbox{ for Case I-(a)}, \\ 
V_{\ts} ~ ~ &\s \bigoplus_{\sigma \in \Pi_{\ts}({\Sp}_{1,1})} \rho_{\sigma} \boxtimes  \si,  ~~  \mbox{ for Cases II and III},
\end{align*}
as representations of the semi-direct product $\cS_{\vp, \scn}(\widehat{{\Sp}_{1,1}}) \rtimes \Sp_{1,1}(F),$
and for Case I-(b), we have:
\[
\Pi_{\ts_1}({\Sp}_{1,1}) = \Pi_{\ts_2}({\Sp}_{1,1}),
\]
\[
\mbox{the multiplicity in }{\Res}^{{\GSp}_{1,1}}_{\Sp_{1,1}}(\ts_i) = \frac{\dimi \rho_{\sigma}}{2}, ~ (i=1,2).
\]
Here, $\Pi_{\sharp}({\Sp}_{1,1})$ denotes the set of equivalence classes of all irreducible constituents of ${\Res}_{{\Sp}_{1,1}}^{{\GSp}_{1,1}}(\sharp)$ with $\sharp \in \{ \ts_1, ~ \ts_2, ~ \ts\},$ as defined in Section \ref{results in rest}.
\end{thm}
\begin{rem}
The bijection in Theorem \ref{1-1 for Sp11} is uniquely determined via the theta correspondence in Proposition \ref{bij bw consti}. 
Nevertheless, since our proof in Section \ref{pf of 1-1 for Sp11} relies on that of $\SL_n'$ by Hiraga and Saito in \cite{hs11}, 
the bijection in Theorem \ref{1-1 for Sp11} depends on the choice of a certain homomorphism $\Lambda_{\SL_n}$ described in Section \ref{for sl(m,D)}.
Moreover, since there is no Whittaker model for the non quasi-split group $\Sp_{1,1},$ each $L$-packet $\Pi_{\vp}(\Sp_{1,1})$ has no base point (cf. \cite[p.3003]{gtsp10}).
\end{rem}

\subsection{Proof of Theorem \ref{1-1 for Sp11}}
\label{pf of 1-1 for Sp11}
We follow the idea in \cite[Lemma 12.6]{hs11} and utilize the results in Sections \ref{sec bij in res}, \ref{para for SO22}, and \ref{para for SO33}. 
We begin with the following lemma.
\begin{lm} \label{a lemma for sp11 and sp4}
With the notation in Section \ref{conj str of L-packets}, we have
\[
\widehat Z_{\tvp, \scn}({\GSp}_{1,1}) = \widehat Z_{\tvp, \scn}({\Sp}_{1,1}) \s \mu_2(\CC).
\]
\end{lm}
\begin{proof}
Since $Z({\Sp}_4(\CC)) \s \mu_2(\CC),$ it suffices to show that $\widehat Z_{\tvp, \scn}({\GSp}_{1,1}) = \mu_2(\CC)/(\mu_2(\CC) \cap S_{\tvp}(\widehat{{\GSp}_{1,1}})^{\circ} )= Z(\widehat{{\Sp}_{1,1}}).$ 
Using the fact that the non quasi-split inner form ${\GSp}_{1,1}$ corresponds to the unique non-trivial character $\sgn$ via the Kottwitz isomorphism \cite[Theorem 1.2]{kot86}, \cite[Lemma 9.1]{hs11} yields
\[
\mu_2(\CC) \cap S_{\tvp}(\widehat{{\GSp}_{1,1}})^{\circ} \subset \ker ({\sgn}) = \{ 1 \},
\]
which completes the proof of the lemma.
\end{proof}
Lemma \ref{a lemma for sp11 and sp4} and the exact sequence \eqref{central ext}  give the following two exact sequences:
\begin{equation} \label{central ext for GSp11}
1 \longrightarrow \mu_2(\CC)  \longrightarrow \cS_{\tvp, \scn}(\widehat{{\GSp}_{1,1}}) \longrightarrow \cS_{\tvp}(\widehat{{\GSp}_{1,1}}) \longrightarrow 1,
\end{equation}
\begin{equation} \label{central ext for Sp11}
1 \longrightarrow \mu_2(\CC)  \longrightarrow \cS_{\vp, \scn}(\widehat{{\Sp}_{1,1}}) \longrightarrow \cS_{\vp}(\widehat{{\Sp}_{1,1}}) \longrightarrow 1.
\end{equation}
From \cite[Section 7]{gtan12}, we note that:  
\begin{align*}
\cS_{\tvp, \scn}(\widehat{{\GSp}_{1,1}}) \s \ZZ/2\ZZ \quad &\text{if} \quad  \cS_{\tvp}(\widehat{{\GSp}_{1,1}}) \s \{1\}, \\
\cS_{\tvp, \scn}(\widehat{{\GSp}_{1,1}}) \s \ZZ/2\ZZ \times \ZZ/2\ZZ \quad &\text{if} \quad \cS_{\tvp}(\widehat{{\GSp}_{1,1}}) \s \ZZ/2\ZZ.
\end{align*}
From \cite[Proposition 2.9]{gtsp10}, we recall an exact sequence 
\begin{equation} \label{central ext for Sp4}
1 \longrightarrow \cS_{\tvp}(\widehat{{\GSp}_{4}}) \longrightarrow \cS_{\vp}(\widehat{{\Sp}_{4}}) \longrightarrow I(\tvp) \longrightarrow 1.
\end{equation}  
Combining \eqref{central ext for GSp11}, \eqref{central ext for Sp11}, and \eqref{central ext for Sp4}, we have the following commutative exact sequences
\begin{equation} \label{commutative diagrams for Sp11}
\begin{CD}
@. @. @. 1 \\
@. @. @. @VVV @.\\
1 @>>> \mu_2(\CC) @>>> \cS_{\tvp, \scn}(\widehat{{\GSp}_{1,1}}) @>>> \cS_{\tvp} (\widehat{{\GSp}_{1,1}}) @>>> 1 \\
@. @| @VV{\cap}V @VV{\cap}V @.\\
1 @>>> \mu_2(\CC) @>>> \cS_{\vp, \scn}(\widehat{{\Sp}_{1,1}}) @>>> \cS_{\vp}(\widehat{{\Sp}_{1,1}}) @>>> 1 \\
@. @. @. @VVV @. \\
@. @. @. I(\tvp) \\
@. @. @. @VVV @.\\
@. @. @. 1 
\end{CD}
\end{equation}
By the snake Lemma, we obtain an exact sequence (the middle vertical exact sequence)
\begin{equation} \label{important exact}
1 \longrightarrow \cS_{\tvp, \scn}(\widehat{{\GSp}_{1,1}}) \longrightarrow \cS_{\vp, \scn}(\widehat{{\Sp}_{1,1}}) \longrightarrow I(\tvp) \longrightarrow 1.
\end{equation}

Now, we verify the theorem for each case in Section \ref{para}.

\textbf{Case I-(a)}:
Note that $\tvp$ is elliptic (hence, $\vp$ is elliptic). 
It then follows that $\widehat Z_{\vp, \scn}({\SO}^*_{1,1}) = \mu_2(\CC) \times \mu_2(\CC)$ (cf. Section \ref{conj str of L-packets}), the connected component group $\cS$   equals $S$ itself, and $\cS_{\tvp, \scn}(\widehat{{\GSp}_{1,1}}) \s \ZZ/2\ZZ \times \ZZ/2\ZZ.$ 

To emphasize groups in the definitions \eqref{def of I sigma} and \eqref{def of I vp}, we recall the notation $I^G$ with $G \in \{{\Sp}_{1,1}, ~ {\SO}^*_{1,1}\}.$
Since $\tvp_1 \not\s \tvp_2 \otimes \chi$ for any character $\chi$ on $F^{\times},$ 
we have
\begin{equation} \label{iso toward A-1(a)}
I^{{\Sp}_{1,1}}(\ts_1) \s I^{{\Sp}_{1,1}}(\ts_2) \s I^{{\Sp}_{1,1}}(\tvp) \s I^{{\SO}^*_{1,1}}(\tvp) \s I^{{\SO}^*_{1,1}}(JL(\tau_1) \boxtimes \tau_2) \s I^{{\SO}^*_{1,1}}(JL(\tau_2) \boxtimes \tau_1).
\end{equation}
From Lemma \ref{lem 2 SO4}, \eqref{central ext for Sp4}, \eqref{important exact}, and \eqref{iso toward A-1(a)},  we then have the following exact sequences
\begin{equation} \label{case I-(a)}
\begin{CD}
@. @. @. S_{\tvp}(\widehat{{\GSp}_{4}}) \\
@. @. @. @VV{\cap}V @.\\
 1 @>>> \mu_2(\CC) @>>> S_{\vp, \scn}(\widehat{{\Sp}_{1,1}}) @>>> S_{\vp}(\widehat{{\Sp}_{1,1}}) @>>> 1 \\
 @. @VV{\cap}V @| @VV{\text{surjective}}V \\
  1 @>>> S_{\tvp, \scn}(\widehat{{\GSp}_{1,1}}) @>>> S_{\vp, \scn}(\widehat{{\Sp}_{1,1}}) @>>> I^{\Sp_{1,1}}(\tvp) @>>> 1 \\
@. @AA{\s}A @AA{\s}A @AA{\s}A \\
1 @>>> \mu_2(\CC) \times \mu_2(\CC) @>>> S_{\vp, \scn}(\widehat{{\SO}^*_{1,1}}) @>>> S_{\vp} (\widehat{{\SO}^*_{1,1}}) @>>> 1 .
\end{CD}
\end{equation}
Here, we take the embedding $\mu_2(\CC) \hookrightarrow \mu_2(\CC) \times \mu_2(\CC)$ as $a \mapsto (a, a).$ 
Since the character $\sgn$ on $\mu_2(\CC)$ is lifted to the two characters $\sgn \times \mathbbm{1} (= \zeta^*_{1,1})$ and $\mathbbm{1} \times \sgn$ on $\mu_2(\CC) \times \mu_2(\CC)$ via the embedding,  
we have the following bijection 
\begin{equation} \label{1-1 for A-1}
\Irr(S_{\vp, \scn}(\widehat{{\Sp}_{1,1}}), \sgn) \overset{1-1}{\longleftrightarrow} 
\Irr(S_{\vp, \scn}(\widehat{{\SO}^*_{1,1}}), \zeta^*_{1,1}) \sqcup \Irr(S_{\vp, \scn}(\widehat{{\SO}^*_{1,1}}), \mathbbm{1}\times \sgn).
\end{equation}
We note that, via the Kottwitz isomorphism \cite[Theorem 1.2]{kot86}, the characters $\zeta^*_{1,1}$ and $\mathbbm{1} \times \sgn$
respectively correspond to 
\[
({\SL}_1(D) \times {\SL}_2)/\Delta\mu_2 = {\SO}^*_{1,1} ~~ \text{and} ~~ ({\SL}_2 \times {\SL}_1(D))/\Delta\mu_2 = {{\SO}^*_{1,1}}^{+-},
\]
which are non quasi-split $F$-inner forms of $\SO_4$ (see Remark \ref{some +_ SO4}).
Considering the characters $\zeta^*_{1,1}$ and $\mathbbm{1} \times \sgn$ as characters on $S_{\tvp, \scn}(\widehat{{\GSp}_{1,1}}),$  due to \eqref{case I-(a)}, we have the following bijections:
\begin{equation*} 
\Irr(S_{\vp, \scn}(\widehat{{\Sp}_{1,1}}), \zeta^*_{1,1}) \overset{1-1}{\longleftrightarrow} 
\Irr(S_{\vp, \scn}(\widehat{{\SO}^*_{1,1}}), \zeta^*_{1,1}) 
\overset{1-1}{\longleftrightarrow} \Pi_{JL(\tau_1) \boxtimes \tau_2}({\SO}^*_{1,1}),
\end{equation*}
\begin{equation*} 
\Irr(S_{\vp, \scn}(\widehat{{\Sp}_{1,1}}), \mathbbm{1} \times \sgn) \overset{1-1}{\longleftrightarrow} 
\Irr(S_{\vp, \scn}(\widehat{{\SO}^*_{1,1}}), \mathbbm{1} \times \sgn) 
\overset{1-1}{\longleftrightarrow} \Pi_{\tau_1 \boxtimes JL(\tau_2)}({{\SO}^*_{1,1}}^{+-})
\overset{1-1}{\longleftrightarrow} \Pi_{JL(\tau_2) \boxtimes \tau_1}({\SO}^*_{1,1}).
\end{equation*}
Since the character $\zeta^*_{1,1}$ corresponds to $\ts_1$ and the other $\mathbbm{1} \times \sgn$ corresponds to $\ts_2$ (see \cite[Section 7.2]{gtan12}),
Proposition \ref{bij bw consti} and Theorem \ref{1-1 for SO4} give rise to the following bijections: 
\begin{equation*} 
\Irr(S_{\vp, \scn}(\widehat{{\SO}^*_{1,1}}), \zeta^*_{1,1}) 
\overset{1-1}{\longleftrightarrow} \Pi_{JL(\tau_1) \boxtimes \tau_2}({\SO}^*_{1,1})
\overset{1-1}{\longleftrightarrow} 
\Pi_{\ts_1}({\Sp}_{1,1}),
\end{equation*}
\begin{equation*} 
\Irr(S_{\vp, \scn}(\widehat{{\SO}^*_{1,1}}), \mathbbm{1} \times \sgn) 
\overset{1-1}{\longleftrightarrow} \Pi_{\tau_1 \boxtimes JL(\tau_2)}({{\SO}^*_{1,1}}^{+-})
\overset{1-1}{\longleftrightarrow} \Pi_{JL(\tau_2) \boxtimes \tau_1}({\SO}^*_{1,1})
\overset{1-1}{\longleftrightarrow} 
\Pi_{\ts_2}({\Sp}_{1,1}).
\end{equation*}
Using Proposition \ref{bij 333}, Theorem \ref{1-1 for SO4}, and the isomorphism $S_{\vp, \scn}(\widehat{{\Sp}_{1,1}}) \s S_{\vp, \scn}(\widehat{{\SO}^*_{1,1}})$ in \eqref{case I-(a)},
we thus have the following isomorphism
\[
V_{\ts_i} ~ ~  \s \bigoplus_{\sigma \in \Pi_{\ts_i}({\Sp}_{1,1})} \rho_{\sigma} \boxtimes  \si ~ (i=1, 2),
\]
as representations of the semi-direct product $S_{\vp, \scn}(\widehat{{\Sp}_{1,1}}) \rtimes \Sp_{1,1}(F).$
This completes the proof of Theorem \ref{1-1 for Sp11} for {Case I-(a)}.

\textbf{Case I-(b)}: Since $\tvp_1 \s \tvp_2 \otimes \chi$ with $\chi$ necessarily quadratic, 
we have 
\begin{equation} \label{iso toward A-1(b)}
I^{{\Sp}_{1,1}}(\ts_1) \s I^{{\Sp}_{1,1}}(\ts_2) \s I^{{\SO}^*_{1,1}}(\tvp) \s I^{{\SO}^*_{1,1}}(JL(\tau_1) \boxtimes \tau_2) \s I^{{\SO}^*_{1,1}}(JL(\tau_2) \boxtimes \tau_1) \overset{\not\s}{\hookrightarrow} I^{{\Sp}_{1,1}}(\tvp).
\end{equation}
One can notice that \eqref{iso toward A-1(b)} is slightly different from \eqref{iso toward A-1(a)}.
From Lemma \ref{lem 2 SO4}, \eqref{central ext for Sp4}, \eqref{important exact}, and \eqref{iso toward A-1(b)},  we then have the following exact sequences
\begin{equation} \label{case I-(b)}
\begin{CD}
@. @. @. S_{\tvp}(\widehat{{\GSp}_{4}}) \\
@. @. @. @VV{\cap}V @.\\
 1 @>>> \mu_2(\CC) @>>> S_{\vp, \scn}(\widehat{{\Sp}_{1,1}}) @>>> S_{\vp}(\widehat{{\Sp}_{1,1}}) @>>> 1 \\
 @. @VV{\cap}V @| @VV{\text{surjective}}V \\
  1 @>>> S_{\tvp, \scn}(\widehat{{\GSp}_{1,1}}) @>>> S_{\vp, \scn}(\widehat{{\Sp}_{1,1}}) @>>> I^{\Sp_{1,1}}(\tvp) @>>> 1 \\
@. @AA{\s}A @AA{\cup}A @AA{\cup}A \\
1 @>>> \mu_2(\CC) \times \mu_2(\CC) @>>> S_{\vp, \scn}(\widehat{{\SO}^*_{1,1}}) @>>> S_{\vp} (\widehat{{\SO}^*_{1,1}}) @>>> 1 . 
\end{CD}
\end{equation}
To see the bijection \eqref{the 1-1} for Case I-(b), we use \cite[Lemma 9.2.2]{art12} and obtain the following bijection with the property \cite[(9.2.15)]{art12} 
\begin{equation} \label{bijection 1 for A-1(b)}
\Irr(S_{\vp, \scn}(\widehat{{\Sp}_{1,1}}), \sgn) 
\overset{1-1}{\longleftrightarrow} 
\Irr(Z(S_{\vp, \scn}(\widehat{{\Sp}_{1,1}})), \sgn), 
\end{equation}
where $Z(S_{\vp, \scn}(\widehat{{\Sp}_{1,1}}))$ denotes the center of the group $S_{\vp, \scn}(\widehat{{\Sp}_{1,1}}).$ 
Note that, in \cite[Lemma 9.2.2]{art12}, our notation $Z(S_{\vp, \scn}(\widehat{{\Sp}_{1,1}}))$ is $Z_{\psi}$ and our character $\sgn$ is $\widehat{\zeta}_{\psi}.$ 
Moreover, the number $N$ therein equals 5, so that $|I_o|$ (see Section \ref{all pos} for the definition) equals 3 or 5 depending on partitions of 5.
Using \cite[(9.2.10)]{art12} and \eqref{case I-(b)}, we then get:
\begin{equation} \label{inclusions of centers 1}
\mu_2(\CC)={Z}((\widehat{{\Sp}_{1,1}})_{\scn}) \leq Z(S_{\vp, \scn}(\widehat{{\Sp}_{1,1}})) \lneqq Z(S_{\vp, \scn}(\widehat{{\SO}^*_{1,1}})),
\end{equation}
\begin{equation} \label{inclusions of centers 2}
\mu_2(\CC)={Z}((\widehat{{\Sp}_{1,1}})_{\scn}) < \mu_2(\CC) \times \mu_2(\CC) ={Z}((\widehat{{\SO}^*_{1,1}})_{\scn}) \lneqq Z(S_{\vp, \scn}(\widehat{{\SO}^*_{1,1}})). 
\end{equation}
Furthermore, we note that 
\begin{equation*} 
[Z(S_{\vp, \scn}(\widehat{{\SO}^*_{1,1}})): Z(S_{\vp, \scn}(\widehat{{\Sp}_{1,1}}))] = 2, 
\end{equation*}
which implies that 
\[
[Z(S_{\vp, \scn}(\widehat{{\Sp}_{1,1}})): {Z}((\widehat{{\Sp}_{1,1}})_{\scn})] = [Z(S_{\vp, \scn}(\widehat{{\SO}^*_{1,1}})):{Z}((\widehat{{\SO}^*_{1,1}})_{\scn}) ].
\]
Restricting characters via the inclusions \eqref{inclusions of centers 1} and \eqref{inclusions of centers 2}, by \cite[Lemma 9.2.2]{art12}, we have the following bijections
\begin{equation*} 
\Irr(Z(S_{\vp, \scn}(\widehat{{\Sp}_{1,1}})), \sgn) 
\overset{1-1}{\longleftrightarrow} 
\Irr(Z(S_{\vp, \scn}(\widehat{{\SO}^*_{1,1}})), \sgn \times \mathbbm{1})
\overset{1-1}{\longleftrightarrow}
\Irr(S_{\vp, \scn}(\widehat{{\SO}^*_{1,1}}), \sgn \times \mathbbm{1}). 
\end{equation*}
Then, Proposition \ref{bij bw consti} and Theorem \ref{1-1 for SO4} yield
\begin{equation} \label{bijection 4 for A-1(b)}
\Irr(S_{\vp, \scn}(\widehat{{\SO}^*_{1,1}}), \sgn \times \mathbbm{1})
\overset{1-1}{\longleftrightarrow}
\Pi_{\ts_1}({\Sp}_{1,1}) = \Pi_{\ts_2}({\Sp}_{1,1}). 
\end{equation}
Proposition \ref{pro for lifting} implies that
\begin{equation} \label{two res}
{\Res}^{\GSp_{1,1}}_{\Sp_{1,1}}(\ts_1) \s  {\Res}^{\GSp_{1,1}}_{\Sp_{1,1}}(\ts_2),
\end{equation}
which gives the last equality in \eqref{bijection 4 for A-1(b)}.
Let $\sigma \in \Pi_{\vp}(\Sp_{1,1})$ be given. 
We write $\rho^*_{1,1}(\sigma)$ for the image of $\rho_{\sigma}$ via the composite of the bijections \eqref{bijection 1 for A-1(b)} -- \eqref{bijection 4 for A-1(b)}.
From \cite[Lemma 9.2.2]{art12}, we then have
\begin{equation} \label{difference 2}
\dimi \rho_{\sigma} = 2 \cdot \dimi \rho^*_{1,1}(\sigma).
\end{equation}
Note from Section \ref{results in rest} that $\dimi \rho^*_{1,1}(\sigma)$ equals the mutiplicity in the restrictions \eqref{two res}.
This completes the proof of Theorem \ref{1-1 for Sp11} for {Case I-(b)}.
\begin{rem} \label{rem special for a-1b}
We make the following remarks on the proof for {Case I-(b)} above.
\begin{itemize}
\item[1.] It follows from the idea in \cite[Section 9.2]{art12} that
$[S_{\vp, \scn}(\widehat{{\Sp}_{1,1}}):S_{\vp, \scn}(\widehat{{\SO}^*_{1,1}})] = 2.$
This index leads to the difference \eqref{difference 2} in dimensions (cf. \eqref{dim=dim=multi} and Remark \ref{multi for  SL}).
\item[2.] Unlike Case I-(a), we observe that $Z(S_{\vp, \scn}(\widehat{{\Sp}_{1,1}}))$ no longer contains $\mu_2(\CC) \times \mu_2(\CC) \s S_{\tvp, \scn}(\widehat{{\GSp}_{1,1}})$ (see Section \ref{all pos} for details).
This leads to the fact that only one between $\Irr(S_{\vp, \scn}(\widehat{{\Sp}_{1,1}}), \zeta^*_{1,1})$ and $\Irr(S_{\vp, \scn}(\widehat{{\Sp}_{1,1}}), \mathbbm{1}\times \sgn)$ is non-empty. 
Thus, we do not have the bijection \eqref{1-1 for A-1}.
\end{itemize}
\end{rem}

\textbf{Case II}:
We then have the following commutative exact sequences
\begin{equation} \label{case II}
\begin{CD}
 1 @>>> \mu_2(\CC) @>>> \cS_{\vp, \scn}(\widehat{{\Sp}_{1,1}}) @>>> \cS_{\vp}(\widehat{{\Sp}_{1,1}}) @>>> 1 \\
@. @VVV @VVV @VV{\s}V \\
1 @>>> \widehat Z_{\vp, \scn}({\SO}^*_{1,1}) @>>> \cS_{\vp, \scn}(\widehat{{\SO}^*_{1,1}}) @>>> \cS_{\vp} (\widehat{{\SO}^*_{1,1}}) @>>> 1 . 
\end{CD}
\end{equation}
Here, the very right isomorphism comes from the fact that $\# \cS_{\tvp}(\widehat{{\GSp}_{1,1}}) = 1$ and  
$
\cS_{\vp} (\widehat{{\SO}^*_{1,1}}) \s I(\tvp) \s \cS_{\vp}(\widehat{{\Sp}_{1,1}}).
$
Further, since the character ${\zeta}^*_{1,1}$ on $\mu_2(\CC) \times \mu_2(\CC)$ equals $\sgn \times \mathbbm{1},$ using the same idea in the proof of Lemma \ref{a lemma for sp11 and sp4}, $\widehat Z_{\vp, \scn}({\SO}^*_{1,1})$ is either $\mu_2(\CC) \times \{1\}$ or $\mu_2(\CC) \times \mu_2(\CC).$
We claim that 
\begin{equation} \label{Zsc}
\widehat Z_{\vp, \scn}({\SO}^*_{1,1}) \s \mu_2(\CC) \times \{1\} \s \mu_2(\CC).
\end{equation}
Indeed, it follows from \cite[p.535]{art12} that
\[
S_{\vp, \scn}(\widehat{G})^{\circ} = (Z(\widehat M_{\scn})^{\Gamma})^\circ,
\]
where $\Gamma$ acts trivially, $\widehat M_{\scn}$ is the preimage of $\widehat M$ in $\widehat G_{\scn},$ and $M$ is a Levi subgroup of $G$ with respect to which $\vp$ is elliptic. 
Due to the proof of \cite[Proposition 7.1 (iii)]{gtan12} for {Case II}, $M$ is the Siegel maximal Levi subgroup of $\SO^*_{1,1},$ and $\widehat M_{\scn}$ equals 
$
{\GL}_1(\CC) \times {\SL}_2(\CC),
$
which is a Levi subgroup of $\Spin_{4}(\CC) \s \SL_2(\CC) \times \SL_2(\CC) = \widehat G_{\scn}$ (see \cite[Section 4]{sh88} and \cite[2.3.1]{kim05}).
It thus follows that
\[
S_{\vp, \scn}(\widehat{G})^{\circ} = {\GL}_1(\CC), 
\]
which implies
\[
Z(\widehat G_{\scn} \cap S_{\vp, \scn}(\widehat{G})^{\circ}) \s \{1\} \times \mu_2(\CC).
\]
From the definition 
\[
\widehat Z_{\vp, \scn}({\SO}^*_{1,1}) := Z(\widehat G_{\scn}) / (Z(\widehat G_{\scn}) \cap S_{\vp, \scn}(\widehat{G})^{\circ}),
\]
the claim \eqref{Zsc} has been verified.
We then have
\begin{equation} \label{iso S Sp11 and SO11} 
\cS_{\vp, \scn}(\widehat{{\SO}^*_{1,1}}) \s \cS_{\vp, \scn}( \widehat{{\Sp}_{1,1}}),
\end{equation}
which implies that 
\[
\Irr(\cS_{\vp, \scn}(\widehat{{\SO}^*_{1,1}}), \zeta^*_{1,1}) \overset{1-1}{\longleftrightarrow} 
\Irr(\cS_{\vp, \scn}(\widehat{{\Sp}_{1,1}}), \sgn).
\]
From Proposition \ref{bij 333} and Theorem \ref{1-1 for SO4}, we thus have
\[
V_{\ts} ~ ~  \s \bigoplus_{\sigma \in \Pi_{\ts}({\Sp}_{1,1})} \rho_{\sigma} \boxtimes  \si,
\]
as representations of the semi-direct product $S_{\vp, \scn}(\widehat{{\Sp}_{1,1}}) \rtimes \Sp_{1,1}(F).$
This completes the proof of Theorem \ref{1-1 for Sp11} for {Case II}.

\textbf{Case III}: 
Note that $\tvp$ is elliptic (hence, $\vp$ is). 
It is clear that $\widehat Z_{\vp, \scn}({\SO}^*_{3,0}) = \mu_4(\CC).$
Further, the group connected component $\cS$  equals $S$ itself. 
We have the following commutative exact sequences
\begin{equation} \label{case III}
\begin{CD}
 1 @>>> \mu_2(\CC) @>>> S_{\vp, \scn}(\widehat{{\Sp}_{1,1}}) @>>> S_{\vp}(\widehat{{\Sp}_{1,1}}) @>>> 1 \\
@. @VV{\cap}V @VV{\cap}V @VV{\s}V \\
1 @>>> \mu_4(\CC) @>>> S_{\vp, \scn}(\widehat{{\SO}^*_{3,0}}) @>>> S_{\vp} (\widehat{{\SO}^*_{3,0}}) @>>> 1 . 
\end{CD}
\end{equation}
Indeed, the last isomorphism comes from 
\begin{equation} \label{case III below}
S_{\vp} (\widehat{{\SO}^*_{3,0}}) \s I(\tvp) \s S_{\vp}(\widehat{{\Sp}_{1,1}}),
\end{equation}
since $S_{\tvp}(\widehat{{\GSp}_{1,1}})$ is a singleton.
Further, since 
$S_{\vp}(\widehat{{\Sp}_{1,1}}) = S_{\vp}(\widehat{{\Sp}_{1,1}}) \subset \SO_{5}(\CC)$ ($\vp$ being elliptic), it follows that $S_{\vp, \scn}(\widehat{{\Sp}_{1,1}})$ equals the inverse image ${pr_{1,1}}^{-1}(S_{\vp}(\widehat{{\Sp}_{1,1}})),$
where we recall that $pr_{1,1}$ is the usual projection from $\Sp_{4}(\CC)$ onto $\SO_{5}(\CC).$
Likewise, 
since
$S_{\vp}(\widehat{{\SO}^*_{3,0}})\subset \SO_{6}(\CC),$ it follows that $ S_{\vp, \scn}(\widehat{{\SO}^*_{3,0}}) = {pr^*_{3,0}}^{-1}(S_{\vp}(\widehat{{\SO}^*_{3,0}})),$
where we recall that $pr^*_{3,0}$ is the usual projection from $\SL_{4}(\CC)$ onto $\SO_{6}(\CC).$
Using the commutative diagram
\[
\begin{CD}
 1 @>>> \mu_2(\CC) @>>> {\Sp}_4(\CC) @>{pr_{1,1}}>> {\SO}_5(\CC) @>>> 1 \\
@. @VV{\cap}V @VV{\cap}V @VV{\cap}V \\
1 @>>> \mu_4(\CC) @>>> {\SL}_4(\CC) @>pr^*_{3,0}>> {\SO}_6(\CC) @>>> 1,
\end{CD}
\]
we thus have the inclusion $S_{\vp, \scn}(\widehat{{\Sp}_{1,1}}) \subset S_{\vp, \scn}(\widehat{{\SO}^*_{3,0}}).$
\begin{rem}
This inclusion can be also obtained from the following commutative diagram
\[
\begin{CD}
 1 @>>> \mu_2(\CC) @>>> S_{\vp, \scn}(\widehat{{\Sp}_{1,1}}) @>>> S_{\vp}(\widehat{{\Sp}_{1,1}}) @>>> 1 \\
@. @| @VV{\cap}V @VV{\cap}V \\
1 @>>> \mu_2(\CC) @>>> S_{\vp, \scn}(\widehat{{\SO}^*_{3,0}}) @>>> C_{\vp} (\widehat{{\SO}^*_{3,0}}) @>>> 1 .
\end{CD}
\]
\end{rem}
To complete the proof of Theorem \ref{1-1 for Sp11} for {Case III}, from Proposition \ref{bij bw consti} and Theorem \ref{1-1 for SO6}, it remains to show that 
\[
\Irr(S_{\vp, \scn}(\widehat{{\SO}^*_{3,0}}), \zeta^*_{3,0}) \overset{1-1}{\longleftrightarrow} 
\Irr(S_{\vp, \scn}(\widehat{{\Sp}_{1,1}}), \sgn).
\]
Equivalently, due to \eqref{case III}, we claim that the restriction 
$
\rho|_{S_{\vp, \scn}(\widehat{{\Sp}_{1,1}})}
$
is irreducible and $\rho|_{\mu_2(\CC)} = \sgn,$  for any $\rho \in \Irr(S_{\vp, \scn}(\widehat{{\SO}^*_{3,0}}), \zeta^*_{3,0}).$
To verify this argument, we set $S_{\vp, \scn}(\widehat{{\Sp}_{1,1}})=A$ and $S_{\vp, \scn}(\widehat{{\SO}^*_{3,0}})=B$ for simplicity.
Using the Frobenius reciprocity, we have
\[
\langle \rho|_A, \rho|_A \rangle_A = \langle {\Ind}_A^B(\rho|_A), \rho \rangle_B,
\] 
where $\langle \rho_1, \rho_2 \rangle_H = \dimi_{\CC} \Hom_H(\rho_1, \rho_2)$ for any representation $\rho_i$ of a finite group $H.$
Since $B/A \s \ZZ/2\ZZ \s \mu_4(\CC)/\mu_2(\CC),$ we have
\[
{\Ind}_A^B(\rho|_A) \s \rho \oplus (\rho \otimes \chi),
\]
where $\chi$ is a character on $\mu_4(\CC)$ but trivial on $\mu_2(\CC).$ 
But, since $(\rho \otimes \chi)|_{\mu_4(\CC)} \neq \zeta^*_{3,0},$ we have $\rho \not\s  \rho \otimes \chi.$ 
Thus, it follows that
\[
\langle \rho|_A, \rho|_A \rangle_A = 1,
\]
which implies that $\rho|_A$ is irreducible.
Lastly, it is immediate that $\rho|_{\mu_2(\CC)} = {\zeta^*_{3,0}}|_{\mu_2(\CC)} = \sgn.$
This completes the proof of Theorem \ref{1-1 for Sp11} for {Case III}.
Therefore, the proof of Theorem \ref{1-1 for Sp11} is complete. 
\subsection{Properties of ${\L}_{1,1}$-map for $\Sp_{1,1}$} \label{properties for sp11}
The ${\L}_{1,1}$-map defined in Section \ref{const of L-packet for Sp11} satisfies the following properties.
\begin{pro} \label{discreteness for Sp11}
A given $\sigma \in \Pi({\Sp}_{1,1})$ is an essentially square-integrable representation if and only if its $L$-parameter $\vp_{\sigma} := \L_{1,1}(\sigma)$ does not factor through any proper Levi subgroup of ${\SO}_5(\CC).$
\end{pro}
\begin{proof}
By the definition of $\L_{1,1}$ in Section \ref{const of L-packet for Sp11}, $\sigma$ is an irreducible constituent of the restriction $\ts|_{{\Sp}_{1,1}}$ for some $\ts \in \Pi({\GSp}_{1,1}).$
From Remark \ref{rem in rest from to} and \cite[Theorem 9.1(b)]{gtan12},  $\sigma$ is an essentially square-integrable representation if and only if $\ts$ is if and only if $\tvp_{\sigma}:= L_{1,1}(\ts)$ does not factor through any proper Levi subgroup of ${\GSp}_4(\CC)$ if and only if $\vp_{\sigma}$ does not.
\end{proof}
\begin{rem}
In the same way with the proof of Proposition \ref{discreteness for Sp11}, we have that a given $\sigma_{1,1} \in \Irr({\Sp}_{1,1})$ is tempered if and only if the image of its $L$-parameter $\vp_{\sigma_{1,1}} := \L_{1,1}(\sigma_{1,1})$ in ${\SO}_5(\CC)$ is bounded. 
\end{rem}
\begin{pro}
Let $\vp \in \Phi_{\disc}(\Sp_{1,1})$ and $\sigma_1, \sigma_2 \in \Pi_{\vp}(\Sp_{1,1})$ be given.
Let $M$ be an $F$-Levi subgroup of an $F$-inner form of $\Sp_{2n},$ which is the product of $\Sp_{1,1}$ and copies of $F$-inner forms of $\GL_{m_i}$ with $n=2+\sum m_i.$ 
For any $\tau \boxtimes \si_1, \tau \boxtimes \si_2 \in \Pi_{\disc}(M),$ $\nu \in \mathfrak{a}^{*}_{M, \CC},$ and $w \in W_M$ with $^wM = M,$
we have 
\[
\mu_M(\nu, \tau \boxtimes \si_1, w) = \mu_M(\nu, \tau \boxtimes \si_2, w).
\]
\end{pro}
\begin{proof}
This follows from \cite[Section 8]{gtan12},  our construction of $L$-packets in Section \ref{const of L-packet for Sp11}, and the fact that the restriction of representations preserves the intertwining operator and the Plancherel measure (cf. \cite[Section 2.2]{choiy1}).
\end{proof}
\subsection{More studies on $L$-packets for ${\Sp}_{1,1}$} \label{all pos}
Based on \cite[Sections 5 and 6]{gtsp10} and \cite[Section 9.2]{art12}, we classify the central extension \eqref{central ext} for all $\vp \in \Phi(\Sp_{1,1})$ and illuminate all sizes of $L$-packets of $\Sp_{1,1}$ as well as all multiplicities in restriction from $\GSp_{1,1}.$
Let $\vp \in \Phi({\Sp}_{1,1})$ be given.
Using Theorem \ref{thm by Labesse}, we fix a lifting $\tvp  \in \Phi({\GSp}_{1,1}),$
such that $\vp = \std_{1,1} \circ \tvp,$ where $\std_{1,1}$ is the surjective map from $\widehat{{\GSp}_{1,1}}$ to $\widehat{{\Sp}_{1,1}}$ as in \eqref{std}.

From Section \ref{para}, we recall three mutually exclusive cases: {Case I}, {Case II}, {Case III}. 
For each case, we will give a description of $\cS_{\vp} = \cS_{\vp}(\widehat{\Sp_{1,1}})$ and its central lifting $\cS_{\vp, \scn}= \cS_{\vp, \scn}(\widehat{\Sp_{1,1}}),$ which fit into the following exact sequence in \eqref{commutative diagrams for Sp11}
\begin{equation*} 
1 \longrightarrow \mu_2(\CC) \longrightarrow \cS_{\vp, \scn} \longrightarrow \cS_{\vp} \longrightarrow 1.
\end{equation*} 
We claim that $\cS_{\vp, \scn}$ is isomorphic to one of the following groups
\[
\ZZ/2\ZZ \times \ZZ/4\ZZ,~ (\ZZ/2\ZZ)^2, ~ D_8, ~ \text{the Pauli group}, ~ D_8 \ast Q_8, ~ \ZZ/4\ZZ, ~ \ZZ/2\ZZ.
\]
To this end, we first recall some arguments from \cite[Section 9.2]{art12}, which will be applied to the cases of elliptic parameters: {Case I}, {Case III}.
Let $\vp$ be elliptic. 
Then, we have $\cS_{\vp} = S_{\vp}$ and $\cS_{\vp, \scn} = S_{\vp, \scn}.$
Following \cite[Section 9.2, p.531]{art12}, we set
\[
\vp = \vp_1 \oplus \vp_2 \oplus \cdots \oplus \vp_r,
\]
where $\vp_i \in \Phi_2(\GL(N_i))$ and $N_1 + N_2 + \cdots + N_r = 5,$
and
we have a decomposition of $\{1, \dots, 5\}$ into two disjoint subsets 
\[
\{1, \dots, 5\} = I_e \sqcup I_o,
\] 
consisting of those indices $k$ whose associate degrees $N_k$ are either even or odd.
Applying arguments in \cite[pp. 531-534]{art12} to  the case of $\Sp_{1,1},$ we have
\[
\delta_{\vp} = 1, ~~ \varepsilon_{\vp}=0, ~~ S_{\vp} \s (\ZZ/2\ZZ)^{r-1}.
\] 
Moreover, since $S_{\vp, \scn}$ is abelian if and only if $|I_o|\leq 2$ (see \cite[p.533]{art12}), $S_{\vp, \scn}$ is non-abelian if and only if the partition of $N$ is $1+1+1+1+1,$ $2+1+1+1,$ or $3+1+1.$
For these three cases, the derived group of $S_{\vp, \scn}$ equals $\{ \pm 1 \}$ and the center $Z(S_{\vp, \scn})$ has order $2^{|I_e|+1}.$


\textbf{Case I}:
Since $\tvp=\tvp_1 \oplus \tvp_2,$ we have $\vp = \std_{1,1} \circ \tvp = \mathbbm{1} \oplus (\tvp_1^{\vee} \otimes \tvp_2)$ (cf. \cite[p. 3008]{gtsp10}), where $\tvp_1^{\vee} $ is the contragredient of $\tvp_1.$ 
Further, since $S_{\tvp, \scn}(\widehat{{\GSp}_{1,1}}) \s \ZZ/2\ZZ \times \ZZ/2\ZZ,$ from \eqref{commutative diagrams for Sp11}, we have 
\begin{equation*} 
\ZZ/2\ZZ \times \ZZ/2\ZZ \hookrightarrow S_{\vp, \scn}.
\end{equation*}
As in Section \ref{para}, based on the classification in \cite[Proposition 6.8(iii)]{gtsp10}, we proceed with two following subcases.

\textbf{Case I-(a)}:
From the proof of \cite[Proposition 6.8(iii)]{gtsp10}, the partition of $N=5$ is either $1+2+2$ or $1+4.$
Note that $\cS_{\vp, \scn}$ is abelian for both cases.

When $5=1+2+2,$ $S_{\vp, \scn}$ is isomorphic to $ (\ZZ/2\ZZ)^3,$ $2\ZZ \times \ZZ/4\ZZ,$ or $\ZZ/8\ZZ,$ since $S_{\vp} \s (\ZZ/2\ZZ)^2.$
Using arguments in \cite[(9.2.7) and p. 531]{art12} on the order of each element of $S_{\vp, \scn},$ one can notice that $S_{\vp, \scn}$ consists of four elements of order 4, three elements of order 2, and the identity.
Note that our group $S_{\vp, \scn}$ is denoted by a subgroup $B_{\psi}$ of  $B(N)$ in \cite[p. 531]{art12}. 
We thus have  
\[
1 \longrightarrow \mu_2(\CC) \longrightarrow S_{\vp, \scn} \s \ZZ/2\ZZ \times \ZZ/4\ZZ  \longrightarrow S_{\vp} \s (\ZZ/2\ZZ)^2 \longrightarrow 1
\]
and $|\Pi_{\vp}(\Sp_{1,1})|=4.$ 
By Remark \ref{dim1=dim2}, \eqref{dim=dim=multi}, and Theorem \ref{1-1 for Sp11}, the multiplicity in ${\Res}^{\GSp_{1,1}}_{\Sp_{1,1}}(\ts)$ is 1.

When $5=1+4,$ $S_{\vp, \scn}$  is isomorphic to $(\ZZ/2\ZZ)^2$ or $\ZZ/4\ZZ,$ since $S_{\vp} \s \ZZ/2\ZZ.$ 
As in the case of $5=1+2+2$ in Case I-(a), one can notice that  $S_{\vp, \scn}$ consists of three elements of order 2 and the identity. 
we thus have 
\[
1 \longrightarrow \mu_2(\CC) \longrightarrow S_{\vp, \scn} \s (\ZZ/2\ZZ)^2 \longrightarrow S_{\vp} \s \ZZ/2\ZZ \longrightarrow 1
\]
and
$|\Pi_{\vp}(\Sp_{1,1})|=2.$
By Remark \ref{dim1=dim2}, \eqref{dim=dim=multi}, and Theorem \ref{1-1 for Sp11}, the multiplicity in ${\Res}^{\GSp_{1,1}}_{\Sp_{1,1}}(\ts)$ is 1.

\textbf{Case I-(b)}: 
From \cite[Proposition 6.8(iii)(b)]{gtsp10}, we notice that
\[
\vp = {\std}_{1,1} \circ \tvp = \mathbbm{1} \oplus \chi \oplus Ad(\tvp_1)\chi.
\]
This implies that the partition of $N=5$ is $1+1+3,$ $1+1+1+2,$ or $1+1+1+1+1.$
Note that $S_{\vp, \scn}$ is non-abelian for all three cases.

When $5=1+1+3,$ 
we have $S_{\vp} \s (\ZZ/2\ZZ)^2$ and $S_{\vp, \scn}$  is non-abelian of order 8, 
which is isomorphic to either the dihedral group $\mathcal{D}_8$ of order 8 or the finite quaternion group $Q_8$ of order 8.
As in the case of $5=1+2+2$ in Case I-(a), one can notice that $S_{\vp, \scn}$ consists of two elements of order 4, five elements of order 2, and the identity.
We thus have 
\[
1 \longrightarrow \mu_2(\CC) \longrightarrow S_{\vp, \scn} \s \mathcal{D}_8  \longrightarrow S_{\vp} \s (\ZZ/2\ZZ)^2 \longrightarrow 1.
\]
Further, by \cite[Lemma 9.2.2]{art12} (or directly from the character table of $\mathcal{D}_8$), ${\Irr}(S_{\vp, \scn})$ consists of one 2-dimensional representation and four 1-dimensional characters. 
Thus, $|\Pi_{\vp}(\Sp_{1,1})|=1.$ 
By Theorem \ref{1-1 for Sp11}, the multiplicity in ${\Res}^{\GSp_{1,1}}_{\Sp_{1,1}}(\ts)$ is 1.
We will give an explicit example for this case in Section \ref{sec of example}.

When $5=1+1+1+2,$ 
we have $S_{\vp} \s (\ZZ/2\ZZ)^3$ and  $S_{\vp, \scn}$   is non-abelian of order 16.
As in the case of $5=1+2+2$ in Case I-(a), one can notice that $S_{\vp, \scn}$ consists of eight elements of order 4, seven elements of order 2, and the identity.
Moreover, from \cite[(9.2.10)]{art12} we note that the center $Z(S_{\vp, \scn})$ has order 4 and contains an element of order 4 (denoted by $b_{\{4\}}$ in \cite[(9.2.10)]{art12}).
It follows that $Z(S_{\vp, \scn}) \s \ZZ/4\ZZ.$
Thus, $S_{\vp, \scn}$ has a group presentation
\[
\{
a^\alpha b^\beta c^\gamma ~:~ a^4=b^2=c^2=1, ~ ba=ab, ~ ca=ac, ~ cb=a^2bc 
\},
\]
which equals $G_{10}$ in \cite[Theorem 2]{wild05}.
Another description of $S_{\vp, \scn}$ is $\{\pm I,\pm iI,\pm X,\pm iX,\pm Y,\pm iY,\pm Z,\pm iZ\},$
where $i=\sqrt{-1},$ 
\[
I=I_{2 \times 2}=\begin{pmatrix}
1  & 0\\
0 &  1
\end{pmatrix}, \quad
X=\begin{pmatrix}
0  & 1\\
1 &  0
\end{pmatrix}, \quad
Y=\begin{pmatrix}
0  & -i\\
i &  0
\end{pmatrix}, ~~ \text{and}~~
Z=\begin{pmatrix}
1  & 0\\
0 &  -1
\end{pmatrix},
\]
which is called the Pauli group (see \cite{nc00}).
Therefore, we have
\[
1 \longrightarrow \mu_2(\CC) \longrightarrow S_{\vp, \scn} \s \text{the Pauli group} \s G_{10} \text{ in \cite[Theorem 2]{wild05}}  \longrightarrow S_{\vp} \s (\ZZ/2\ZZ)^3 \longrightarrow 1.
\]
Further, by \cite[Lemma 9.2.2]{art12}, ${\Irr}(S_{\vp, \scn})$ consists of two 2-dimensional representations and eight 1-dimensional characters.
Thus, $|\Pi_{\vp}(\Sp_{1,1})|=2.$
By Theorem \ref{1-1 for Sp11}, the multiplicity in ${\Res}^{\GSp_{1,1}}_{\Sp_{1,1}}(\ts)$ is 1.

When $5=1+1+1+1+1,$ 
we have $S_{\vp} \s (\ZZ/2\ZZ)^4$ and  $S_{\vp, \scn}$   is non-abelian of order 32.
As in the case of $5=1+2+2$ in Case I-(a), one can notice that  $S_{\vp, \scn}$ consists of twenty elements of order 4, eleven elements of order 2, and the identity.
Moreover, we note that $Z(S_{\vp, \scn}) \s \mu_2(\CC)$ by \cite[(9.2.10)]{art12}, and 
${\Irr}(S_{\vp, \scn})$ consists of one 4-dimensional representation and sixteen 1-dimensional characters by \cite[Lemma 9.2.2]{art12}.
Thus, $S_{\vp, \scn}$ is isomorphic to the central product $\mathcal{D}_8 \ast Q_8$ of the diheldral group $\mathcal{D}_8$ and the quaternion group $Q_8,$ defined as
\[
\mathcal{D}_8 \ast Q_8 := (\mathcal{D}_8 \times Q_8)/\langle(z_1,z_2) \rangle,
\]
where $\langle z_1 \rangle$ and $\langle z_2 \rangle$ denote the unique normal subgroups of $\mathcal{D}_8$ and $Q_8,$ respectively.
Note that $\mathcal{D}_8 \ast Q_8$ is one of two extra-special groups of order 32 and the other one has twelve elements of order 4 (cf. \cite{gor80, robinson96, cath05}).
Therefore, we have 
\[
1 \longrightarrow \mu_2(\CC) \longrightarrow S_{\vp, \scn} \s \mathcal{D}_8 \ast Q_8 \longrightarrow S_{\vp} \s (\ZZ/2\ZZ)^4 \longrightarrow 1,
\]
and $|\Pi_{\vp}(\Sp_{1,1})|=1.$
By Theorem \ref{1-1 for Sp11}, the multiplicity in ${\Res}^{\GSp_{1,1}}_{\Sp_{1,1}}(\ts)$ is 2.

\textbf{Case II}: 
Based on \cite[Proposition 6.10(ii)]{gtsp10}, we have
\begin{equation} \label{need to be det}
1 \longrightarrow \mu_2(\CC) \longrightarrow \cS_{\vp, \scn} \s \ZZ/4\ZZ \longrightarrow \cS_{\vp} \s \ZZ/2\ZZ \longrightarrow 1,
\end{equation}
if $\tvp_0$ is dihedral with respect to a quadratic extension $E/F$ (by definition, $\tvp_0 \s \tvp_0 \otimes \omega_{E/F},$ where $\omega_{E/F}$ is the quadratic character corresponding to $E$ via the local class field theory) and $\det (\tvp_0)=\omega_{E/F};$ otherwise, we have
\[
1 \longrightarrow \mu_2(\CC) \overset{\s}{\longrightarrow} \cS_{\vp, \scn} \longrightarrow \cS_{\vp} = 1 \longrightarrow 1.
\]
Accordingly, we have $|\Pi_{\vp}(\Sp_{1,1})|=2,$ or $1,$ respectively.
By Remark \ref{dim1=dim2}, \eqref{dim=dim=multi}, and Theorem \ref{1-1 for Sp11}, the multiplicity in ${\Res}^{\GSp_{1,1}}_{\Sp_{1,1}}(\ts)$ is 1.
We should mention the reason why $\cS_{\vp, \scn} \not\s (\ZZ/2\ZZ)^2$ in \eqref{need to be det}.
Indeed, since $\tvp_0$ is dihedral with respect to $E/F,$ using \cite[Section 9.2]{art12} with the partition $N=3=1+2,$ we have the following exact sequence
\[
1 \longrightarrow \mu_2(\CC) \longrightarrow S_{\vp_0, \scn}(\widehat{{\SL}_2}) \s \ZZ/4\ZZ \longrightarrow S_{\vp_0}(\widehat{{\SL}_2}) \s \ZZ/2\ZZ \longrightarrow 1,
\]
where $\vp_0$ is the image of $\tvp_0$ via the projection $\GL_2(\CC) \twoheadrightarrow \PGL_2(\CC) = \SO_3(\CC).$
Since  $\tvp$ is of the form $\chi (\tvp_0 \oplus (\omega_0 \oplus \mathbbm{1})),$  we have
\[
\cS_{\vp, \scn} = S_{\vp_0, \scn}(\widehat{{\SL}_2}) \times \{ 1 \}
\]
Thus, the exact sequence \eqref{need to be det} follows.

\textbf{Case III}: Since $\vp$ is elliptic, we follow the idea of {Case I} above.
From \cite[Theorem 6.5 and Proposition 6.8 (i)\&(ii)]{gtsp10}, there are only following cases. 
\begin{itemize}
\item \textbf{(a)}: $\tvp$ is premitive. This case is the situation of 
\cite[Theorem 6.5(I)]{gtsp10}).
Simply, the partition of $N=5$ is $5$ due to \cite[Proposition 5.1(I)]{gtsp10}. 
We thus have
\[
1 \longrightarrow \mu_2(\CC) \overset{\s}{\longrightarrow} S_{\vp, \scn} \longrightarrow S_{\vp} = 1 \longrightarrow 1
\]
and $|\Pi_{\vp}(\Sp_{1,1})|=1.$ 
By Remark \ref{dim1=dim2}, \eqref{dim=dim=multi}, and Theorem \ref{1-1 for Sp11}, the multiplicity in ${\Res}^{\GSp_{1,1}}_{\Sp_{1,1}}(\ts)$ is 1.

\item \textbf{(b)}: $\tvp = \Ind_{W_E}^{W_F}\si,$ $\si^{\tau} \s \si \chi,$ $\chi^2 \neq 1,$ and $\simi(\tvp)|_{W_E}=\chi \det\si \neq \det \si,$ where $E/F$ is a quadratic extension with $\Gal(E/F)=<\tau>,$ $\si$ is a primitive representation of $W_E$ (by definition, 
$\tvp_0$ is not of the form $\Ind_{W_F}^{W_E} \rho$ for a finite extension $F/E$ and some irreducible $\rho$), and $\chi$ is a character of $W_E.$ 
This case is the situation of 
\cite[Theorem 6.5(II)]{gtsp10}).
Simply, the partition of $N=5$ is $2+3$ due to \cite[Proposition 5.1(II)]{gtsp10}. 
As before, we have $S_{\vp} \s \ZZ/2\ZZ$ and the central extension $S_{\vp, \scn}$ by $\mu_2(\CC)$ is abelian of order 4, which is isomorphic to $(\ZZ/2\ZZ)^2$ or $\ZZ/4\ZZ.$
As in the case of $5=1+2+2$ in Case I-(a), one can notice that $S_{\vp, \scn}$ consists of two elements of order 4, one elements of order 2, and the identity.
We thus have
\[
1 \longrightarrow \mu_2(\CC) {\longrightarrow} S_{\vp, \scn} \s \ZZ/4\ZZ \longrightarrow S_{\vp} \s \ZZ/2\ZZ \longrightarrow 1
\]
and $|\Pi_{\vp}(\Sp_{1,1})|=2.$
By Remark \ref{dim1=dim2}, \eqref{dim=dim=multi}, and Theorem \ref{1-1 for Sp11}, the multiplicity in ${\Res}^{\GSp_{1,1}}_{\Sp_{1,1}}(\ts)$ is 1.

\item \textbf{(c)}: 
$\tvp = \Ind_{W_E}^{W_F}\si$ and  $\simi(\tvp)|_{W_E}=\det \si,$ where $E/F$ is a quadratic extension with $\Gal(E/F)=<\tau>$ and $\si$ is an irreducible 2-dimensional representation of $W_E.$
This case is the situation of 
\cite[Theorem 6.5(III)]{gtsp10}. 
We divide into the following subcases whose partitions follow from the proofs of \cite[Proposition 5.1 and Theorem 6.5(III)]{gtsp10}. We first consider the case that $\si^{\tau} \neq \si \chi$ for any character $\chi$ of $W_E,$ which is the situation of 
\cite[Theorem 6.5(III)(a)]{gtsp10}. 
\begin{itemize}
\item \textbf{(c1)}:  $\si$ is primitive. 
Then, the partition of $N=5$ is $1+4,$
which is the same as the second case in {Case I-(a)} above.

\item \textbf{(c2)}: $\si= \Ind_{W_E}^{W_K}\rho$ with $\Gal(K/F)\s \ZZ/4\ZZ.$ 
This is the same as Case III-(c1).

\item \textbf{(c3)}: $\si= \Ind_{W_E}^{W_K}\rho$ with $\Gal(K/F)\s \ZZ/2\ZZ \times \ZZ/2\ZZ.$ The partition of $N=5$ is $1+2+2,$ which is the same as the first case in {Case I-(a)}.

\item \textbf{(c4)}: $\si= \Ind_{W_E}^{W_K}\rho$ with $K/F$ non-Galois.
The partition of $N=5$ is either $1+4$ or $1+2+2,$ which is the same as Case III-(c1) or Case III-(c3), respectively.
\end{itemize}
We then consider the case that $\si^{\tau} = \si \chi$ for some character $\chi$ of $W_E$ ($\chi$ is necessarily quadratic), which is the situation of 
\cite[Theorem 6.5(III)(b)]{gtsp10}. 
\begin{itemize}
\item \textbf{(c5)}: $\chi^{\tau} \neq \chi.$ The partition of $N=5$ is $1+2+2,$ which is the same as Case III-(c3).

\item \textbf{(c6)}: $\chi^{\tau} = \chi.$ The partition of $N=5$ is $1+1+3,$ $1+1+1+2,$ or $1+1+1+1+1,$ which is the same as Case I-(b) above.
\end{itemize}

\item \textbf{(d)}: This is the remaining elliptic parameter which is the situation of \cite[Proposition 6.8(i)\&(ii)]{gtsp10}. We divide into the following subcases whose partitions follow from the proofs of\cite[Proposition 6.8]{gtsp10}.
\begin{itemize}
\item \textbf{(d1)}: $\tvp = \mu \boxtimes S_4$ with $\mu$ a 1-dimensional character of $W_F$ and $S_4$ the 4-dimensional representation of $\SL_2(\CC).$
The partition of $N=5$ is $5,$ which is the same as Case III-(a).

\item \textbf{(d2)}: $\tvp = \si \boxtimes S_2$ with $\si$ an irreducible 2-dimensional dihedral representation of $W_F$ and $S_2$ the 2-dimensional representation of $\SL_2(\CC).$ 
The partition of $N=5$ is either $1+1+3$ or $2+3.$
When $5=1+1+3,$ $S_{\vp, \scn}$ is non-abelian of order 8, which is the same as the case of  $5=1+1+3$ in Case I-(b). 
When $5=2+3,$ this is the same as Case III-(b).
\end{itemize}
\end{itemize}

\begin{rem}
By Remark \ref{dim1=dim2}, \eqref{dim=dim=multi}, and Theorem \ref{1-1 for Sp11}, the multiplicity in ${\Res}^{\GSp_{1,1}}_{\Sp_{1,1}}(\ts)$ is 1, except when the partition is $1+1+3,$ $1+1+1+2,$ $1+1+1+1+1.$ For these three cases, the multiplicity is 2, 2, 4 in the order given.
\end{rem}

\begin{rem}
We note that the two cases (i) and (iii) of \cite[Proposition 6.10]{gtsp10} are not relevant to $\Sp_{1,1}$ \cite[Section 3]{bo79}, since $\Sp_{1,1}$ has a unique (up to conjugacy) minimal $F$-parabolic subgroup, which is the Siegel maximal parabolic and isomorphic to $D^{\times}.$ 
Further, its dual parabolic subgroup in $\widehat{\Sp_{1,1}} = \SO_5(\CC)$ is the image of the Heisenberg (or Klingen) parabolic subgroup ($\s \GL_2(\CC) \times \GL_1(\CC)$) of $\GSp_4(\CC)$ under the projection $\std_{1,1}: \GSp_4(\CC)\twoheadrightarrow \SO_5(\CC)$
(see \cite[Section 3]{bo79}, \cite[Section 3.1]{pt11}, and \cite[p.762]{gtan12}). 
This is why the two cases (i) and (iii) of \cite[Proposition 6.10]{gtsp10} deos not occur in $\GSp_{1,1}$ (see \cite[Section 7]{gtan12}). 
\end{rem}


\subsection{An example} \label{sec of example} 
We give an explicit $L$-packet of $\Sp_{1,1},$ which is considered as a new phenomenon arising in the non quasi-split inner form $\Sp_{1,1}$ differently than the split group $\Sp_4.$ 
Let $\tvp = \tvp_0 \oplus \tvp_0 \chi \in \Phi(\GSp_4)$ be given, 
where $\chi$ is a quadratic character, $\tvp_0 \in \Phi(\GL_2)$ is primitive (by definition, 
$\tvp_0$ is not of the form $\Ind_{W_E}^{W_F} \rho$ for a finite extension $E/F$ and some irreducible $\rho$), 
and $\tvp_0 \not\s \tvp_0 \chi.$ 
We then have an irreducible supercuspidal representation $\pi \in \Pi(\GL_2)$ corresponding to $\vp_0$  via the LLC for $\GL_2$ (cf. \cite[Proposition 6.3]{gtsp10}). 
Further, from the LLC for $\GSp_4$ \cite{gt}, we have an $L$-packet attached to $\vp$
\[
\Pi_{\tvp}({\GSp}_4) = \{ \ts_1, \ts_2   \},
\]
where $\ts_1$ is the theta correspondence of $\pi \boxtimes \pi\chi$ from $\GSO_{2,2}$ and $\ts_2$ is the theta correspondence of $JL(\pi) \boxtimes JL(\pi)\chi$ from $\GSO_{4,0}.$ 
Here, $JL$ is the local Jacquet-Langlands correspondence between $\GL_2$ and $D^{\times}.$ 
The projection $\vp$ of $\tvp$ onto $\widehat{\Sp_4}=\SO_5(\CC)$ is 
\begin{equation} \label{L-parameter vp}
\vp= \mathbbm{1} \oplus \chi \oplus Ad(\tvp_0)\chi \in \Phi({\Sp}_4),
\end{equation}
and \cite[Proposition 6.8(iii)(b)]{gtsp10} yields 
\begin{equation} \label{example L-packet Sp4}
\Pi_{\vp}({\Sp}_4) = \{\si_1^+, \si_1^-, \si_2^+, \si_2^- \},
\end{equation}
where
\begin{equation} \label{decomp Sp4}
{\Res}^{{\GSp}_4}_{{\Sp}_4}(\ts_1) = \{ \si_1^+, \si_1^- \}, ~~ \text{and} ~~ {\Res}^{{\GSp}_4}_{{\Sp}_4}(\ts_2) = \{ \si_2^+, \si_2^- \}.
\end{equation}
From Proposition \ref{bij bw consti}, we further note that $\si_1^{+}$ and $\si_1^{-}$ are the theta correspondences of the restriction 
\[
{\Res}^{\GSO_{2,2}}_{\SO_{2,2}}(\pi \boxtimes \pi\chi) =\{\tau_1^+, \tau_1^-  \}
\] 
to $\Sp_4,$ and $\si_2^{+}$ and $\si_2^{-}$  are the theta correspondences of the restriction 
\[
{\Res}^{\GSO_{4,0}}_{\SO_{4,0}}(JL(\pi) \boxtimes JL(\pi)\chi) =\{\tau_2^+, \tau_2^-  \}
\] 
to $\Sp_4.$
On the other hand, from \cite{gtan12}, the given $L$-parameter $\tvp$ provides an $L$-packet for $\GSp_{1,1}$ 
\[
\Pi_{\tvp}({\GSp}_{1,1}) = \{ \ts'_1, \ts'_2   \},
\]
where $\ts'_1$ and $\ts'_2$ are respectively the theta correspondences of $JL(\pi) \boxtimes \pi\chi$ and $JL(\pi\chi) \boxtimes \pi$ from $\GSO^*_{1,1}.$
We note from Proposition \ref{pro for lifting} that 
$
{\Res}^{\GSO^*_{1,1}}_{\SO^*_{1,1}}(JL(\pi) \boxtimes \pi\chi) ~~\text{and}~~ {\Res}^{\GSO^*_{1,1}}_{\SO^*_{1,1}}(JL(\pi\chi) \boxtimes \pi)
$
are identical. 
Since $\tvp_0$ is primitive (see \cite[Proposition 6.3]{gtsp10}), the set of the irreducible constituents is a singleton. 
From Proposition \ref{bij bw consti}, we have
\begin{equation} \label{decomp Sp11}
{\Res}^{\GSp_{1,1}}_{\Sp_{1,1}}(\ts'_1) =  {\Res}^{\GSp_{1,1}}_{\Sp_{1,1}}(\ts'_2)= \{ \si' \},
\end{equation}
and $\ts'_2 \s \ts'_1 \chi.$
Moreover, from the fact that $I(\tvp) \s \{ \mathbbm{1}, \chi\}$ (see \cite[Proposition 6.3(iii)(b)]{gtsp10}), it follows that 
\begin{equation} \label{sepcial multi for Sp11}
I(\ts'_1) = I(\ts'_2) = \{ \mathbbm{1} \}.
\end{equation}
Thus, the $L$-packet of $\Sp_{1,1}$ attached to the $L$-parameter $\vp$ in \eqref{L-parameter vp} is
\begin{equation} \label{example L-packet Sp11}
\Pi_{\vp}({\Sp}_{1,1}) = \{ \si' \},
\end{equation}
as constructed in Section \ref{const of L-packet for Sp11}.
Proposition \ref{bij bw consti} implies that $\si'$ is the theta correspondence of the restriction 
\[
{\Res}^{\GSO^*_{1,1}}_{\SO^*_{1,1}}(\pi \boxtimes \pi\chi) =\{\tau' \}
\] 
to $\Sp_{1,1}.$ 
We recall {Case I-(b)} in Section \ref{all pos} and compute the centralizer $C_{\vp}(\widehat{{\Sp}_{1,1}})$ with $\vp$ in \eqref{L-parameter vp}.
We then have following commutative exact sequence 
\begin{equation} \label{the last diagram}
\begin{CD}
 1 @>>> \mu_2(\CC) @>>> {\Sp}_4(\CC) @>>> {\SO}_5(\CC) @>>> 1 \\
@. @| @AA{\cup}A @AA{\cup}A \\
 1 @>>> \mu_2(\CC) @>>> \cS_{\vp, \scn}(\widehat{{\Sp}_{1,1}}) @>>> \cS_{\vp}(\widehat{{\Sp}_{1,1}}) @>>> 1 \\
@. @| @VV{\s}V @VV{\s}V \\
1 @>>> \mu_2(\CC) @>>> \mathcal{D}_8 @>>> \ZZ/2\ZZ \times \ZZ/2\ZZ @>>> 1.
\end{CD}
\end{equation}
Combining \eqref{example L-packet Sp4}, \eqref{example L-packet Sp11}, and \eqref{the last diagram}, 
we have 
the following bijections 
\[
\Pi_{\vp}({\Sp}_{4}) = \{\si_1^+, \si_1^-, \si_2^+, \si_2^- \} \overset{1-1}{\longleftrightarrow} \Irr(\cS_{\vp, \scn}(\widehat{{\Sp}_{4}}), \mathbbm{1}) \s \Irr(\ZZ/2\ZZ \times \ZZ/2\ZZ),
\]
\[
\Pi_{\vp}({\Sp}_{1,1}) = \{\si' \} \overset{1-1}{\longleftrightarrow} \Irr(\cS_{\vp, \scn}(\widehat{{\Sp}_{1,1}}), \sgn) = \Irr(\mathcal{D}_8, \sgn),
\]
with the decompositions \eqref{decomp Sp4} and \eqref{decomp Sp11}.
Therefore, this example satisfies all properties in Theorem \ref{1-1 for Sp11}.

In what follows, we discuss a new phenomenon arising in the LLC for $\Sp_{1,1}$ which has never occurred in any previous LLC.
The map $\sigma \mapsto \rho_{\sigma}$ from $\Pi_{\vp}({\Sp}_{4})$ to  $\Irr(\cS_{\vp, \scn}(\widehat{{\Sp}_{4}}), \mathbbm{1})$ provides an equality 
\[
\dimi \rho_{\sigma} =1,
\] 
which equals the multiplicity in 
\[
{\Res}_{\Sp_4}^{\GSp_4}(\ts_i)
\]
for $i=1,~2.$ 
Further, since $\Irr(\mathcal{D}_8)$ consists of four 1-dimensional characters and one 1-dimensional irreducible representation, the map $\sigma' \mapsto \rho_{\sigma'}$ from $\Pi_{\vp}({\Sp}_{1,1})$ to  $\Irr(\cS_{\vp, \scn}(\widehat{{\Sp}_{1,1}}), \sgn)$ provides an equality 
\begin{equation} \label{mutli 2}
\dimi \rho_{\sigma'} = 2.
\end{equation}
However, from \eqref{multi and characters} and \eqref{sepcial multi for Sp11}, the multiplicity in 
$
{\Res}_{\Sp_{1,1}}^{\GSp_{1,1}}(\ts'_i)
$
is $1$ for $i=1,~2.$ 
Thus, we need to consider the following quantity $m(\si')$ in the restriction
\begin{equation*} 
{\Res}^{\GSp_{1,1}}_{\Sp_{1,1}}(\ts_1' + \ts_2') = m(\si')\cdot\si'
\end{equation*}
as the multiplicity of $\ts_i'$ in the restriction, which is $2$ and equals $\dimi \rho_{\sigma'}$ in \eqref{mutli 2}.
Further, this fulfills the following character identity
\begin{equation} \label{ci}
\Theta_{\si_1^{+}}(\gamma) + \Theta_{\si_1^{-}}(\gamma) + \Theta_{\si_2^{+}}(\gamma) + \Theta_{\si_2^{-}}(\gamma) = (-1) \cdot \dimi \rho_{\sigma'} \cdot \Theta_{\si'}(\gamma')
\end{equation}
for any elliptic regular semi-simple $\gamma \in \Sp_4(F)$ and $\gamma' \in \Sp_{1,1}(F)$ having the same $\Sp_4(\bar{F})$-conjugacy class. 
Here, $\Theta_{\sharp}$ is the (Harish-Chandra) character function attached to the irreducible smooth representation $\sharp$ over the regular semi-simple set.
The character identity \eqref{ci} between $\Sp_4$ and $\Sp_{1,1}$ is obtained by restricting the character identity  between $\GSp_4$ and $\GSp_{1,1}$ established in \cite[Proposition 11.1(i)]{changan15}.

\begin{rem}
We make the following remarks.
\begin{itemize}
\item[1.] 
Unlike the above example of this section, 
none of two members in the same $L$-packet of $\GSp_4(F)$ share the same restriction to $\Sp_4(F)$ (see \cite{ap06} and \cite[Proposition 2.2]{gtsp10}).
\item[2.] Although the multiplicity in the restriction is one, the dimension of the corresponding irreducible representation of $S_{\vp, \scn}$ is two, since the restrictions of two members in the same $L$-packet of $\GSp_{1,1}$ are the same. 
\item[3.] The dimension of the corresponding irreducible representation of $S_{\vp, \scn}$ may not completely govern the multiplicity of an individual irreducible representation of a $p$-adic group in the restriction (cf. Remark \ref{multi for  SL}).
\end{itemize}
\end{rem}
\appendix

\section{Internal structure of $L$-packets for $\Sp_{4}$} \label{revisit}
The purpose of this appendix is to establish an analogue of Theorem \ref{1-1 for Sp11} for the case of $\Sp_4.$
Note that the $L$-packets for $\Sp_4$ was constructed by Gan and Takeda in \cite{gtsp10} and  their parameterization was also discussed in \cite[pp.3002-3003]{gtsp10} in another way.
We apply the same method discussed in Sections \ref{para} and \ref{pf of 1-1 for Sp11} to the study on ${\Sp}_4.$ 

From \cite{gt} and \cite{gtsp10}, we recall the LLC for $\GSp_4$ and $\Sp_4.$
Consider ${\GSO_{3,3}},$ ${\GSO_{2,2}},$ and ${\GSO_{4,0}}$ which participate in $L$-packets for $\GSp_4$ via the theta correspondence. 
The relations in Section \ref{L-groups} between dual groups can be combined with \cite[Section 6]{gt} to have two inclusions $\iota_{3,3}$ and $\iota_{2,2}$ as follows:
\begin{equation*} 
\iota_{3,3}: \{ \text{irreducible 4-dimensional} ~ \tvp \in \Phi({\GSp}_{4}) \} 
~~ \hookrightarrow ~~
 \Phi({\GSO}_{3,3}) = \Phi({\GL}_4) \times \Phi({\GL}_1)
\end{equation*}
defined by $\iota_{3,3}(\tvp)= (\tvp, \simi \tvp),$ and
\begin{equation*} 
\iota_{2,2}: \{ (\tvp_1, \tvp_2) \in \Phi({\GSO}_{2,2}) : \det \tvp_1 = \det \tvp_2 \} / {\Out}({\SO}_4) 
~~ \hookrightarrow ~~
 \Phi({\GSp}_4) 
\end{equation*} 
defined by $\iota_{2,2}(\tvp_1, \tvp_2) =\tvp_1\oplus\tvp_2 = \tvp,$
where 
the action of ${\Out}({\SO}_4)$ on $\Phi({\GSO}_{2,2})$ is given by $(\tvp_1, \tvp_2) \mapsto (\tvp_2, \tvp_1).$
Since $\Phi({\GSO}_{4,0})$ is the subset of $\Phi({\GSO}_{2,2})$ consisting of $(\tvp_1, \tvp_2)$ with elliptic $L$-parameters $\tvp_1$ and $\tvp_2,$ the restriction of $\iota_{2,2}$ to $\Phi({\GSO}_{4,0})$ is denoted by $\iota_{4,0}.$
We note from \cite[Section 7]{gt} that $\tvp \in \Phi({\GSp}_4)$ is either an irreducible 4-dimensional representation or the image of $\iota_{2,2}.$
The LLC for $\GSp_4$ states that there is a surjective, two-to-one map 
\[
L_{4} : {\Pi}({\GSp}_4) \longrightarrow \Phi({\GSp}_4).
\]
satisfying several natural conditions which determine the map uniquely (see \cite[p.1842]{gt} for details).

The LLC for $\Sp_4$ \cite{gtsp10} states that  there is a surjective, finite-to-one map 
\[
{\L}_{4} : {\Pi}({\Sp}_4) \longrightarrow \Phi({\Sp}_4)
\]
defined by $\L_{4}(\sigma) = \std_{4}(L_{4}(\ts))$ with $\ts \in \Pi({\GSp}_{4})$ such that 
\[
\sigma \hookrightarrow {\Res}_{{\Sp}_{4}}^{{\GSp}_{4}}(\ts).
\]
Note that $\L_{4}$ is not depending on the choice of the lifting $\ts,$ since another lifting must be of the form $\ts \otimes \chi$ for some quasi-character $\chi$ of $F^{\times}$ by Proposition \ref{pro for lifting} and $L_{4}(\ts \otimes \chi) = L_{4}(\ts) \otimes \chi$ for any quasi-character $\chi$ of $F^{\times}$  \cite[Proposition 2.2]{gtsp10}.
For each $\vp \in \Phi({\Sp}_{4}),$ the $L$-packet $\Pi_{\vp}({\Sp}_{4})$ is given by
\begin{equation} \label{def of L-packet sp4}
\Pi_{\vp}({\Sp}_{4}) = \bigcup_{\ts \in \Pi_{\tvp}({\GSp}_{4})} \Pi_{\ts}({\Sp}_{4}),
\end{equation}
where $\tvp$ lies in $\Phi({\GSp}_{4})$ such that $\std_{4} \circ \tvp=\vp$ (see Theorem \ref{thm by Labesse}). 
Note that the union in \eqref{def of L-packet sp4} turns out to be disjoint and the $L$-packet does not depend on the choice of $\tvp$ \cite[Theorem 2.3]{gtsp10}.

To state an analogue of Theorem \ref{1-1 for Sp11} for $\Sp_4$ (Theorem \ref{1-1 for Sp4}) below,
we need three mutually exclusive possibilities of $\tvp \in \Phi(\GSp_{4})$ from \cite[Section 7]{gt}) as follows.
\begin{itemize}
\item \textbf{Case i}: $\tvp$ is of the form $\tvp_1 \oplus \tvp_2,$ where $\tvp_i \in \Phi_{\el}({\GL}_2)$ and $\det \tvp_1 = \det \tvp_2.$ 
Since $\Phi_{\el}({\GL}_2)=\Phi_{\el}({\GL}_1(D)),$ we thus note that $\tvp \in \Phi(\GSO_{4,0}).$

\item \textbf{Case ii}: $\tvp$ is of the form $\tvp_1 \oplus \tvp_2,$ where $\tvp_i \in \Phi({\GL}_2)$ with at least one of $\tvp_1$ and $\tvp_2$ in $\Phi({\GL}_2) \smallsetminus \Phi_{\el}({\GL}_2),$ and $\det \tvp_1 = \det \tvp_2.$ 
We thus note that $\tvp \in \Phi({\GSO}_{2,2}),$ but not in $\Phi({\GSO}_{4,0}).$

\item \textbf{Case iii}: $\tvp$ is irreducible 4-dimensional, and its image via the map $\iota_{3,3}$ lies in $\Phi({\GSO}_{3,3}).$
\end{itemize}
Next, we recall the $L$-packets $\Pi_{\tvp}(\GSp_{4})$ for $\GSp_{4}$ which were constructed in \cite[Section 7]{gtan12}.

\textbf{Case i}: $\Pi_{\tvp}(\GSp_{4})=\{ 
\ts_1:=\theta(\tau_1 \boxtimes \tau_2),  \ts_2:=\theta(JL(\tau_1) \boxtimes JL(\tau_2))
\},$ where $\tau_i \in \Pi_{\ess, \disc}(\GL_2)$ is corresponding to $\tvp_i$ via the local Langlands correspondence for $\GL_2$ \cite{ht01, he00, scholze13}, 
the first $\theta$ stands for theta correspondence from $\GSO_{2,2}$ to $\GSp_{4},$ 
the second $\theta$ does for that from $\GSO_{4,0}$ to $\GSp_{4},$ 
and $JL$ denotes the local Jacquet-Langlands lift from $\GL_2(F)$ to $\GL_1(D).$ 
Note that $\Pi_{\tvp}(\GSp_{4})$ consists of essentially square-integrable representations.

\textbf{Case ii}: 
$\Pi_{\tvp}(\GSp_{4})=\{ 
\ts:=\theta(\tau_1 \boxtimes \tau_2)
\},$ where $\tau_i \in \Pi(\GL_2)$ is corresponding to $\tvp_i$ via the local Langlands correspondence for $\GL_2$ \cite{ht01, he00, scholze13} and
$\theta$ stands for theta correspondence from $\GSO_{2,2}$ to $\GSp_{4}.$ 
Note that $\theta(\tau_1 \boxtimes \tau_2)$ is not an essentially square-integrable representation.

\textbf{Case iii}: $\Pi_{\tvp}(\GSp_4)=\{ \ts:=\pi \},$ where $\pi$ is an irreducible admissible representation of $\GSp_{4}(F)$ 
whose theta lift $\theta(\pi)$ to $\GSO_{3,3}$ is $\Pi \boxtimes \mu \in \Pi(\GSO_{3,3}).$
Note that $\mu= \simi(\tvp)$ via the local class field theory and $\omega_{\Pi}=\mu^2.$

\begin{thm} \label{1-1 for Sp4}
With the notation above, given an $L$-parameter $\vp \in \Phi(\Sp_{4}),$ we fix its lifting $\tvp \in \Phi(\GSp_{4}).$
Then, there is a one-one bijection 
\begin{equation*} 
\Pi_{\vp}({\Sp}_{4}) \overset{1-1}{\longleftrightarrow} \Irr(\cS_{\vp}(\widehat{{\Sp}_{4}})),
\end{equation*}
sending $\sigma \mapsto \rho_{\sigma},$ 
such that we have isomorphisms:
\begin{align*} \label{decompositions for sp4}
V_{\ts_i} ~ ~  & \s \bigoplus_{\sigma \in \Pi_{\ts_i}({\Sp}_{4})} \rho_{\sigma} \boxtimes  \si ~ (i=1, 2), ~~ \mbox{ for Case i}, \\ 
V_{\ts} ~ ~ &\s \bigoplus_{\sigma \in \Pi_{\ts}({\Sp}_{4})} \rho_{\sigma} \boxtimes  \si,  ~~  \mbox{ for Cases ii and iii},
\end{align*}
as representations of the semi-direct product $\cS_{\vp}(\widehat{{\Sp}_{4}}) \rtimes \Sp_{4}(F).$
Here,
$\Pi_{\tilde{\sharp}}({\Sp}_{4})$ denotes the set of equivalence classes of all irreducible constituents of ${\Res}_{{\Sp}_{4}}^{{\GSp}_{4}}(\tilde{\sharp})$ with $\sharp \in \{ \ts_1, ~ \ts_2, ~ \ts\}.$
\end{thm}
\begin{proof}
We follow the proof of the case of $\Sp_{1,1}$ in Section \ref{pf of 1-1 for Sp11}.

\textbf{Case i}: 
Since $\cS_{\tvp}(\widehat{{\GSp}_{4}}) = S_{\tvp}(\widehat{{\GSp}_{4}}) \s \ZZ/2\ZZ,$ we note that the trivial character $\mathbbm{1}$ on  $S_{\tvp}(\widehat{{\GSp}_{4}})$ corresponds to $\ts_1$ and the other $\sgn$  corresponds to $\ts_2$ (see from \cite[Section 7]{gt}).
Since the trivial character $\mathbbm{1}$ on $\mu_2(\CC)$ is lifted to the two characters $\mathbbm{1} \times \mathbbm{1} = \zeta_{2,2}$ and $\sgn \times \sgn = \zeta_{4,0}$ on $\mu_2(\CC) \times \mu_2(\CC)$ under the embedding $a \mapsto (a, a)$ from  $\mu_2(\CC)$ to $\mu_2(\CC) \times \mu_2(\CC),$
we have the following bijection 
\begin{equation*} 
\Irr(S_{\vp, \scn}(\widehat{{\Sp}_{1,1}}), \mathbbm{1}) = \Irr(S_{\vp}(\widehat{{\Sp}_{4}})) \overset{1-1}{\longleftrightarrow} 
\Irr(S_{\vp, \scn}(\widehat{{\SO}^*_{1,1}}), \zeta_{2,2}) \sqcup \Irr(S_{\vp, \scn}(\widehat{{\SO}^*_{1,1}}), \zeta_{4,0}).
\end{equation*}
Note that, via the Kottwitz isomorphism \cite[Theorem 1.2]{kot86}, the characters $\zeta_{2,2}$ and $\zeta_{4,0}$ respectively correspond to ${\SO}_{2,2}$ and ${\SO}_{4,0},$ which are non quasi-split $F$-inner forms of $\SO_4.$
Considering the characters $\zeta_{2,2}$ and $\zeta_{4,0}$ as characters on $S_{\tvp, \scn}(\widehat{{\GSp}_{1,1}}),$  due to \eqref{case I-(a)}, we have the following bijections:
\begin{equation*} 
\Irr(S_{\vp, \scn}(\widehat{{\Sp}_{1,1}}), \zeta_{2,2}) \overset{1-1}{\longleftrightarrow} 
\Irr(S_{\vp, \scn}(\widehat{{\SO}^*_{1,1}}), \zeta_{2,2}) = \Irr(S_{\vp}(\widehat{{\SO}_{2,2}}) 
\overset{1-1}{\longleftrightarrow} \Pi_{\tau_1 \boxtimes \tau_2}({\SO}_{2,2}),
\end{equation*}
\begin{equation*} 
\Irr(S_{\vp, \scn}(\widehat{{\Sp}_{1,1}}), \sgn \times \sgn ) \overset{1-1}{\longleftrightarrow} 
\Irr(S_{\vp, \scn}(\widehat{{\SO}^*_{1,1}}), \zeta_{4,0}) 
\overset{1-1}{\longleftrightarrow} \Pi_{JL(\tau_1) \boxtimes JL(\tau_2)}({\SO}_{4,0}).
\end{equation*}
Since the character $\zeta_{2,2}$ corresponds to $\ts_1$ and the other $\zeta_{4,0}$ corresponds to $\ts_2$ from \cite[Section 7.2]{gtan12},
Proposition \ref{bij bw consti} and Theorem \ref{1-1 for SO4} yield: 
\begin{equation*} 
\Irr(S_{\vp, \scn}(\widehat{{\SO}^*_{1,1}}), \zeta_{2,2}) 
\overset{1-1}{\longleftrightarrow} \Pi_{\tau_1 \boxtimes \tau_2}({\SO}_{2,2})
\overset{1-1}{\longleftrightarrow} 
\Pi_{\ts_1}({\Sp}_{4}),
\end{equation*}
\begin{equation*} 
\Irr(S_{\vp, \scn}(\widehat{{\SO}^*_{1,1}}), \zeta_{4,0}) 
\overset{1-1}{\longleftrightarrow} \Pi_{JL(\tau_1) \boxtimes JL(\tau_2)}({\SO}_{4,0})
\overset{1-1}{\longleftrightarrow} 
\Pi_{\ts_2}({\Sp}_{4}).
\end{equation*}
Using Proposition \ref{bij 333}, Theorem \ref{1-1 for SO4}, and the isomorphism 
$S_{\vp, \scn}(\widehat{{\Sp}_{1,1}}) \s S_{\vp, \scn}(\widehat{{\SO}^*_{1,1}})$ in \eqref{case I-(a)}, 
we thus have the following isomorphism
\[
V_{\ts_i} ~ ~  \s \bigoplus_{\sigma \in \Pi_{\ts_i}({\Sp}_{4})} \rho_{\sigma} \boxtimes  \si ~ (i=1, 2),
\]
as representations of the semi-direct product $S_{\vp}(\widehat{{\Sp}_{4}}) \rtimes \Sp_{4}.$
This completes the proof of Theorem \ref{1-1 for Sp4} for {Case i}.

\textbf{Case ii}: 
From \eqref{case II} and \eqref{iso S Sp11 and SO11}, we have
\[
\Irr(\cS_{\vp, \scn}(\widehat{{\SO}^*_{1,1}}), \zeta_{2,2}) = \Irr(\cS_{\vp}(\widehat{{\SO}_{2,2}})) \overset{1-1}{\longleftrightarrow} 
\Irr(\cS_{\vp, \scn}(\widehat{{\Sp}_{1,1}}), \mathbbm{1}) = \Irr(\cS_{\vp}(\widehat{{\Sp}_{4}}).
\]
From Proposition \ref{bij 333} and Theorem \ref{1-1 for SO4}, we thus proved Theorem \ref{1-1 for Sp4} for {Case ii}.

\textbf{Case iii}: 
From the isomorphism \eqref{case III below}, we have
\[
S_{\vp} (\widehat{{\SO}^*_{3,0}}) \s  S_{\vp}(\widehat{{\Sp}_{4}}).
\]
From Proposition \ref{bij 333} and Theorem \ref{1-1 for SO6}, Theorem \ref{1-1 for Sp4} for {Case iii} follows.
Therefore, the proof of Theorem \ref{1-1 for Sp4} is complete. 
\end{proof}
\subsection*{Acknowledgements}
The author is deeply indebted to Wee Teck Gan for fruitful and inspiring discussions with him. 
Much of this work was done during the author's visit at Max-Planck-Institut f\"{u}r Mathematikat in June and July 2014.
The author is grateful to the institute for their incredible research environment.
This work was partially supported by AMS-Simons Travel Grants. 
%

\end{document}